\documentclass[11pt]{article}

\usepackage{hyperref}
\usepackage{songmei}

\def\tgrad{{\textup{grad}}}
\def\thess{{\textup{Hess}}}
\def\Crt{{\textup{Crt}}}
\def\Vol{{\textup{Vol}}}
\def\oR{\overline{\mathbb R}}

\def\cX{{\mathcal X}}
\def\cB{{\mathcal B}}
\def\cUo{{\mathcal U}}
\def\cTo{{\mathcal T}}
\def\cEo{{\mathcal E}}
\def\PC{{\mathcal C}}

\def\cP{{\mathcal P}}

\def\Mb{{\overline M}}
\def\Eb{{\overline E}}

\def\cUb{\overline { \mathcal U}}
\def\cTb{\overline { \mathcal T}}
\def\cEb{\overline { \mathcal E}}
\def\cA{{\mathcal A}}
\def\cI{{\mathcal I}}
\def\cU{{\mathcal U}}

\def\hbu{\hat{\boldsymbol u}}

\def\bH{{\boldsymbol H}}
\def\bU{{\boldsymbol U}}

\def\bA{{\boldsymbol A}}
\def\bB{{\boldsymbol B}}
\def\bC{{\boldsymbol C}}

\def\ba{{\boldsymbol a}}
\def\bb{{\boldsymbol b}}
\def\bc{{\boldsymbol c}}
\def\bx{{\boldsymbol x}}
\def\by{{\boldsymbol y}}
\def\tbg{\tilde{\boldsymbol g}}
\def\bg{{\boldsymbol g}}
\def\bv{{\boldsymbol v}}
\def\bu{{\boldsymbol u}}
\def\be{{\boldsymbol e}}
\def\bsigma{{\boldsymbol \sigma}}
\def\id{{\boldsymbol I}}
\def\cP{{\mathcal P}}
\def\bperp{{\boldsymbol \perp}}
\def\Proj{{\mathsf P}}
\def\Projp{{\mathsf P}^{\perp}}
\def\supp{{\rm supp}}
\def\He{{\boldsymbol J}}

\def\GOE{{\textup{GOE}}}
\def\normal{{\sf N}}
\def\knot{{0}}

\def\Seta{{\mathcal T_1}}
\def\Setb{{\mathcal T_2}}

\def\bG{{\boldsymbol G}}

\def\bY{{\boldsymbol Y}}
\def\bW{{\boldsymbol W}}
\def\bX{{\boldsymbol X}}
\def\Perm{{\mathfrak S}}
\def\sBayes{\mbox{\tiny\rm Bayes}}
\def\sML{\mbox{\tiny\rm ML}}
\def\reals{{\mathbb R}}

\begin{document}

\title{The landscape of the spiked tensor model}

\author{Gerard Ben Arous\thanks{Courant Institute of Mathematical Sciences, New York University}, \;\; Song Mei\thanks{Institute for Computational and Mathematical Engineering, Stanford University},\;\;
Andrea  Montanari\thanks{Department of Electrical Engineering and Department
  of Statistics, Stanford University}, \;and\;
Mihai Nica\thanks{Department of Mathematics, University of Toronto}}

\maketitle

\begin{abstract}
We consider the problem of estimating a large rank-one tensor  $\bu^{\otimes k}\in(\reals^{n})^{\otimes k}$, $k\ge 3$ in Gaussian noise. 
Earlier work characterized a critical signal-to-noise ratio $\lambda_{\mbox{\tiny\rm Bayes}}= O(1)$ above which an ideal estimator
achieves strictly positive correlation with the unknown vector of interest. Remarkably no polynomial-time algorithm is known that 
achieved this goal unless $\lambda\ge C n^{(k-2)/4}$ and even powerful semidefinite programming relaxations appear to fail for
$1\ll \lambda\ll n^{(k-2)/4}$.

In order to elucidate this behavior, we consider the maximum likelihood estimator, which requires maximizing a degree-$k$
homogeneous polynomial over the unit sphere in $n$ dimensions. We compute the expected number of critical points
and local maxima of this objective function and show  that it is exponential in the dimensions $n$, and give exact formulas for the exponential growth
rate. We show that (for $\lambda$ larger than a constant) critical points are either very close to the unknown vector $\bu$, or  are confined in a band of 
width $\Theta(\lambda^{-1/(k-1)})$  around the maximum circle that is orthogonal to $\bu$.
For local maxima, this band shrinks to be of size $\Theta(\lambda^{-1/(k-2)})$. 
These `uninformative' local maxima are likely to cause the failure of optimization algorithms.
\end{abstract}

\section{Introduction}

Non-convex formulations are the most popular approach for a number of problems across 
high-dimensional  statistics and machine learning. Over the last few years, substantial effort has been devoted to
establishing rigorous guarantees for these methods in the context of important applications. A small subset of examples include 
matrix completion \cite{MR2683452,ge2016matrix}, phase retrieval \cite{chen2017solving,sun2016geometric}, high-dimensional regression with missing data \cite{loh2012high}, two-layers neural networks 
\cite{janzamin2015beating,zhong2017recovery}, and so on.
The general picture that emerges from theses studies --as formalized in \cite{mei2016landscape}-- is that non-convex losses can sometimes be `benign,' and allow
for nearly optimal statistical estimation using gradient descent-type optimization algorithms. Roughly speaking,
this happens when the population risk does not have flat regions, i.e. regions in which the gradient is small and the Hessian is nearly rank-deficient. 

In this paper we explore the flipside of this picture, namely what happens when the population risk has large `flat regions.'  We focus on a simple problem,
tensor principal component analysis under the spiked tensor model, and show that the empirical risk can easily become extremely complex in these
cases. This picture matches recent computational complexity results on the same model.
 
The spiked tensor model \cite{richard2014statistical} captures --in a highly simplified fashion-- a number of statistical estimation tasks in
which we need to extract information from a noisy high-dimensional data tensor, see e.g.  \cite{li2010tensor,morup2011applications,liu2013tensor,kreimer2013tensor}.
We are given a tensor $\bY\in(\reals^n)^{\otimes k}$ of the form
\begin{align}
\bY = \lambda \bu^{\otimes k} + \frac{1}{\sqrt{2n}} \bW\, ,\label{eq:FirstSpiked}
\end{align}
where $\bW$ is a noise tensor,  and would like to estimate the unit vector $\bu \in \S^{n-1}$.
The parameter $\lambda \geq 0$ corresponds to the signal to noise ratio. The noise tensor $\bW \in (\R^n)^{\otimes k}$ is distributed as 
$\bW\eqnd \sum_{\pi \in \Perm_n} \bG^{\pi}/(k!)$, where $\{G_{i_1 \cdots i_k}\}_{1\leq i_1, \ldots, i_k \leq n} \simiid \normal(0,1)$, $\Perm_n$ are permutations of the set
$[n]$, and $(\bG^{\pi})_{i_1\cdots i_k} = G_{\pi(i_1) \cdots \pi(i_k)}$. Throughout the paper $k\ge 3$.

We say that the weak recovery problem is solvable for this model if there exists
an estimator (a measurable function) $\hbu :(\reals^n)^{\otimes k}\to \S^{n-1}$ such that
\begin{align}
\lim\inf_{n\to\infty}\E|\<\hbu(\bY),\bu\>| \ge \eps \, ,\label{eq:WeakRecovery}
\end{align}
for some $\eps>0$. It was proven in \cite{richard2014statistical}  that weak recovery is solvable provided $\lambda\ge \lambda_1(k)$ and
in \cite{montanari2015limitation} that it is unsolvable for $\lambda<\lambda_0(k)$,  for some constant $0<\lambda_0(k)<\lambda_1(k)<\infty$. In fact,
for $\lambda<\lambda_0(k)$ it is altogether impossible to distinguish between the distribution (\ref{eq:FirstSpiked}) and the null model $\lambda=0$.
A sharp theshold $\lambda_{\sBayes}(k)$ for the weak recovery problem
was established in \cite{lesieur2017statistical} (see also \cite{barbier2017layered} for related results), and better lower bounds
for the hypothesis testing problem were proved in \cite{perry2016statistical}.

In light of these contributions, it is fair to say that optimal statistical estimation for the model (\ref{eq:FirstSpiked}) is well understood. In contrast,
many questions are still open for what concerns computationally efficient procedures. Consider the maximum likelihood estimator, 
that requires solving
\begin{equation} \label{eqn:objective}
\begin{aligned}
\mbox{\rm maximize}& \quad f(\bsigma) =\<\bY, \bsigma^{\otimes k} \>, \\
\mbox{\rm subject to} & \quad \bsigma \in \S^{n-1}.
\end{aligned}
\end{equation}
It was shown in \cite{richard2014statistical} that the maximum likelihood estimator achieves weak recovery, cf. Eq.~(\ref{eq:WeakRecovery}), provided
that $\lambda>\lambda_{\sML}(k)$ for some constant\footnote{Indeed, an exact characterization of $\lambda_{\sML}(k)$ should be possible using
the `one-step replica symmetry breaking' formula proven in \cite{talagrand2006free}. A non-rigorous analysis of the implications of this formula was carried out 
in \cite{gillin2000p}, yielding $\lambda_{\sML}(k)= \lambda_{\sBayes}(k)$.} $\lambda_{\sML}(k)<\infty$. However solving the problem (\ref{eqn:objective})
(maximizing an homogeneous degree-$k$ polynomial over the unit sphere)
is NP-hard for all $k\ge 3$ \cite{bhattiprolu2016multiplicative}.

Note that the population risk associated to the problem (\ref{eqn:objective}) is 
\begin{align}
f_0(\bsigma) \equiv \E \<\bY, \bsigma^{\otimes k} \> = \lambda\<\bu,\bsigma\>^k\, .
\end{align}
For $k\ge 3$, the (Riemannian) gradient and Hessian of $f_0(\bsigma)$ vanishes on the hyperplane orthogonal to $\bu$:
$\{\bsigma\in\S^{n-1}:\;\<\bu,\bsigma\>=0\}$. In the intuitive language used above, the population risk has a large flat region.
Since most of the volume of the sphere concentrates around this hyperplane \cite{ledoux2005concentration}, this is expected to have a dramatic impact on
the optimization problem (\ref{eqn:objective}).

Polynomial-time computable estimators have been studied in a number of papers. In particular \cite{richard2014statistical} considers a spectral algorithm based on tensor
unfolding and proved that is succeeds for $k$ even, provided $\lambda\ge C\, n^{(k-2)/4}$. 
(Here and below, we denote by $C$ a constant that might depend on $k$ but is independent of $n$.)
This result was generalized in \cite{hopkins2015tensor} to arbitrary $k\ge 3$,  
using a sophisticated semidefinite programming relaxation from the sum-of-squares hierarchy.  A lower complexity spectral algorithm
that succeeds under the same condition was developed in
\cite{hopkins2016fast}, and further results can be found in 
\cite{anandkumar2016homotopy,bhattiprolu2016certifying}. However,
no polynomial-time algorithm is known that achieves weak recovery for $1\ll \lambda\ll n^{(k-2)/4}$, and it is possible that 
statistical estimation in the spiked tensor model is hard in this
regime. 

A large gap between known polynomial-time algorithms  and statistical
limits arises in the  tensor completion problem, which shares many similarities with the spiked tensor model 
\cite{gandy2011tensor,yuan2015tensor,montanari2016spectral}.  In the setting  of tensor completion, hardness under Feige's hypothesis was proven in \cite{barak2016noisy}
for a certain regime of the number of observed entries.

Here we reconsider the maximum likelihood estimator  and we explore the landscape of the optimization problem
 (\ref{eqn:objective}). In what regime it is hard to maximize the function $f(\,\cdot\,)$ for a typical realization of the random tensor $\bY$?
In \cite{richard2014statistical} 
a power iteration algorithm was studied that attempts to compute the maximum likelihood estimator, and it was proven that
it is successful for $\lambda\ge C\, n^{(k-2)/2}$. What is the origin of this threshold at $n^{(k-2)/2}$? In this  paper we compute
the expected number of critical points of the likelihood function $f(\bsigma)$ to the leading exponential order. 

\begin{figure}[t]
\includegraphics[width=0.45\textwidth]{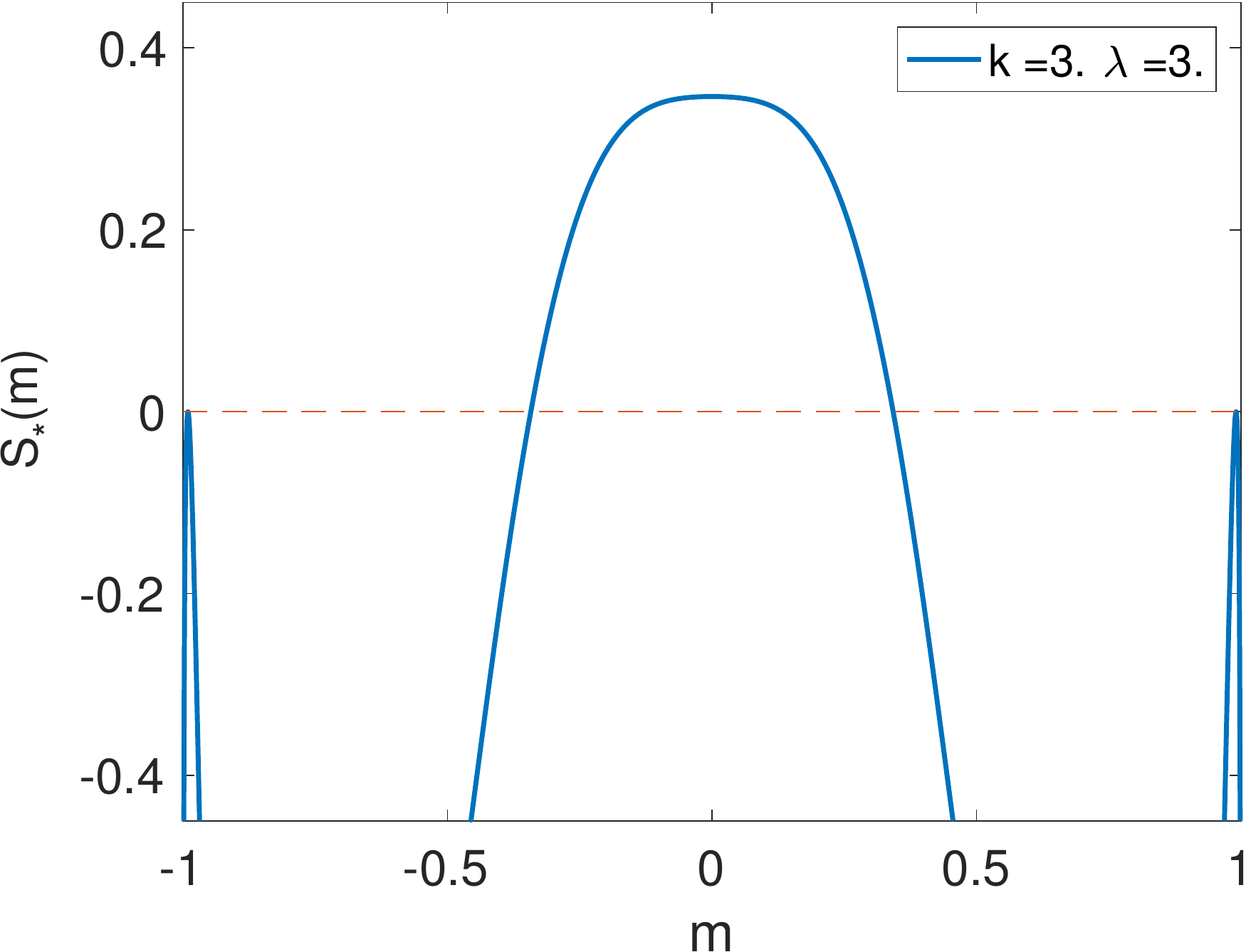}\hspace{0.5cm}
\includegraphics[width=0.45\textwidth]{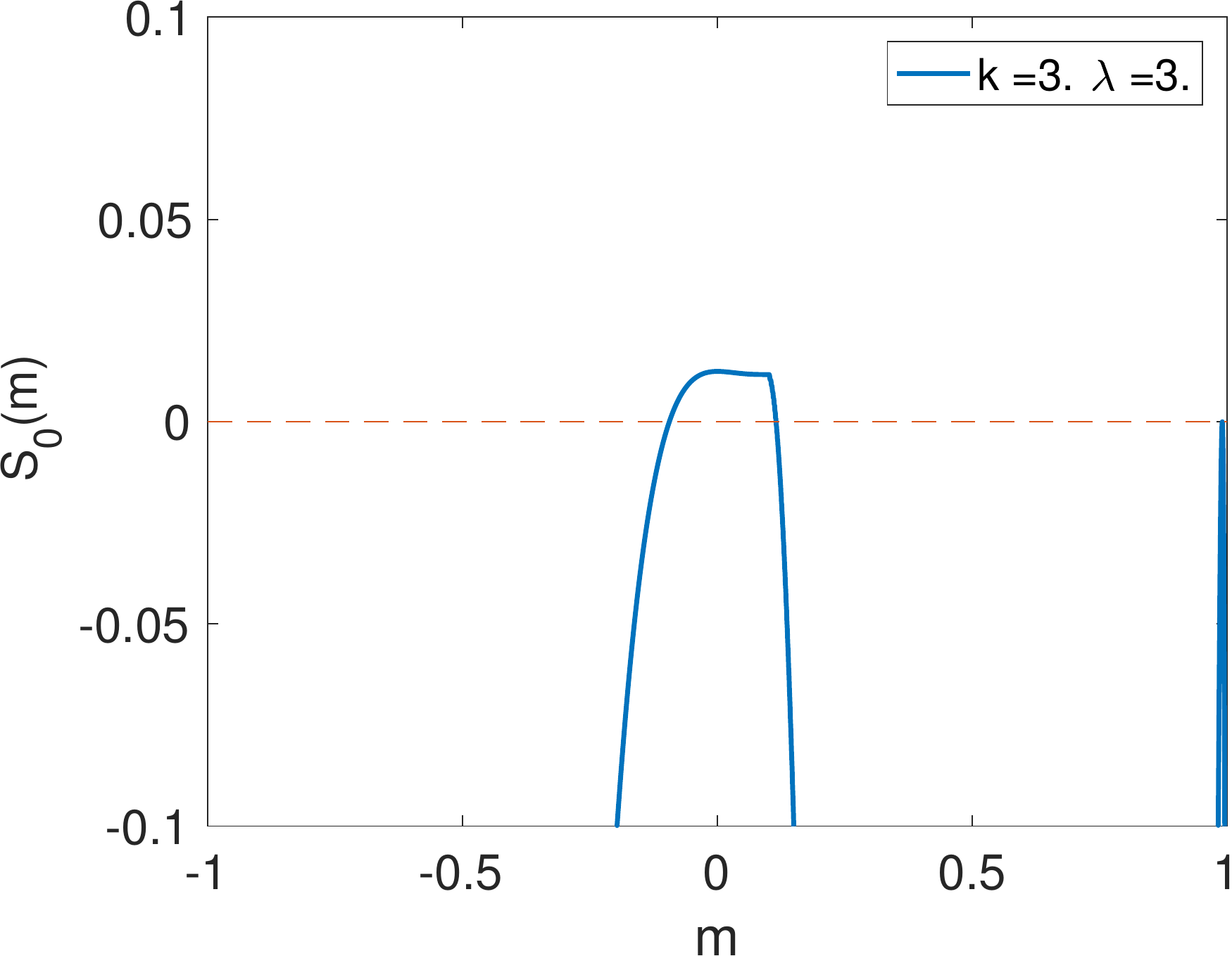}
\caption{Complexity of the spiked tensor model of order $k = 3$ at signal-to-noise ratio $\lambda = 3$:
exponential growth rate of the number of critical points $\bsigma\in\S^{n-1}$, as a function of the scalar product $m=\<\bu,\bsigma\>$.
Left: complexity for the total number of critical points
$S_\star(m)$. Right: complexity for local maxima $S_\knot(m)$.}\label{fig:landscapeFirst}
\end{figure}
Let us summarize the qualitative picture that emerges from our
results. For clarity of exposition, we summarize only our results on local maxima, but similar results will be presented about
generic critical points. 

The expected number of local maxima grows exponentially with the dimension  $n$. We compute the exponential growth rate,
denoted  by  $S_\knot(m,x)$, as a function of the value of the cost function  $x = f(\bsigma)$ and of the scalar product $m = \<\bsigma,\bu\>$.
Namely,  the expected number of local maxima with $f(\bsigma)\approx x$ and $\<\bsigma,\bu\>\approx m$ is $\exp\{nS_\knot(m,x)+o(n)\}$,
with $S_\knot(m,x)$ given explicitly below.  
The exponent $S_\knot(m,x)$ and its variants $S_\knot(m)$, $S_\star(m,x)$, and so on, are referred to as `complexity' functions.
In Figure \ref{fig:landscapeFirst} we plot $S_\knot(m) = \max_x S_\knot(m,x)$, which is the exponential growth rate of the number of local maxima
with scalar product $\<\bsigma,\bu\>\approx m$, for the case $k=3$, $\lambda=3$. (We also plot the analogous quantity for general critical points, $S_\star(m)$.)

The expected number of local maxima with scalar product $m=\<\bsigma,\bu\>\approx 0$,
i.e. lying close to the space orthogonal to the unknown vector $\bu$, is exponentially large. The complexity function $S_0(m)$ decreases as $|m|$ increases, i.e. as we move away form this plane,  and eventually vanishes.

For $\lambda$ sufficiently large  (in particular, for $\lambda>\lambda_c(k)$ given explicitly in Section \ref{sec:Explicit}),  
the complexity $S_0(m)$ reveals an interesting structure.  
It is positive in an interval $m\in (m_1(\lambda,k), m_2(\lambda,k))$, where $m_{1}(\lambda,k),m_2(\lambda,k) = \Theta(\lambda^{-1/(k-2)})$ and  becomes 
non-positive outside this interval. 
however it increases again and touches zero for $m= m_*(\lambda,k)$ close to one
(for $k$ even it  also becomes zero for $m =  -m_*(\lambda,k)$ by symmetry). 
In other words, all the local maxima are either very close to the unknown vector $\bu$ (and to the global maximum) 
or they are on a narrow spherical annulus orthogonal to $\bu$. 

It is interesting to discuss the behavior of local ascent optimization algorithms in such a landscape. 
While at this point the discussion is necessarily heuristic, it points at some interesting directions for future work.
The expected exponential number of local maxima in the annulus $|\<\bu,\bsigma\>|\le \Theta(\lambda^{-1/(k-2)})$
suggests that algorithms can  converge to a local maximum that is well correlated with $\bu$, 
only if they are initialized outside that annulus. In other words, the initialization $\bsigma_0$ must be
such that $\<\bu,\bsigma_0\> \ge C\lambda^{-1/(k-2)}$.  If no side information is available on $\bu$, a random initialization will be used.
This achieves $\<\bu,\bsigma_0\> =\Theta(n^{-1/2})$  with positive probability, 
and hence will escape local maxima provided $\lambda\ge Cn^{(k-2)/2}$. Remarkably, this is the same scaling as the
threshold for power iteration obtained in  \cite{richard2014statistical}. It would be interesting to make rigorous this connection.

Let us emphasize that our results only concerns the \emph{expected number} of critical points. As is
customary with random variables that fluctuate on the exponential scale, this is not necessarily close to the typical
number of critical points. While we expect that several qualitative features found in this work will hold when considering the typical
number, a rigorous justification is still open (see Section \ref{sec:Related} for further discussion of this point).

The rest of the paper is organized as follows.
We  state formally our main results in Section \ref{sec:Main}, which also sketches the main ideas of the proofs. 
We will then review earlier literature in Section \ref{sec:Related}, and present proofs in Section \ref{sec:Proof}.

\section{Main results}
\label{sec:Main}

Our main results concern the number of critical points and the number of local maxima of the function $f(\bsigma)$ introduced 
in Eq.~(\ref{eqn:objective}), where $\bY\in(\reals^{n})^{\otimes k}$ is distributed as per Eq.~(\ref{eq:FirstSpiked}).

Throughout, we denote by $\nabla f(\bsigma)$ and $\nabla^2 f(\bsigma)$ be the Euclidean gradients and Hessians of $f$ at 
$\bsigma$, and $\tgrad f(\bsigma)$ and $\thess f(\bsigma)$ be the Riemannian gradients and Hessians at $\bsigma$. 
The completed real line is denoted by $\oR = \R\cup\{+\infty,-\infty\}$. For a set $S\subseteq \R$, we denote by $\bar{S}$ its closure,
and $S^o$ its interior.

\subsection{Complexity of critical points}

For any Borel sets $E \subset \R$ and $M \in [-1,1]$, we define $\Crt_{n, \star} (M,E)$ to be the number of critical points of $f$ with function value in $E$ and correlation in $M$:
\begin{equation}
\Crt_{n, \star} (M,E) \eqndef \sum_{\bsigma: \tgrad f(\bsigma) = \bzero} \ones\{\< \bsigma, \bu \> \in M\} \ones\{ f(\bsigma) \in E \}.
\end{equation}

We define function $S_{\star}: [-1,1] \times \R \rightarrow \oR$ as
\begin{align}
S_{\star}(m,x) \eqndef \frac{1}{2}(\log(k-1) + 1) + \frac{1}{2}\log (1-m^2)- k \lambda^2 m^{2k-2} (1-m^2) - (x - \lambda m^k)^2  + \Phi_{\star}\Big(\sqrt{\frac{2k}{k-1}} x\Big) , \label{eq:StarDef}
\end{align}
where
\begin{align}
\Phi_\star(x) = \left\{\begin{array}{lr}
        x^2/4 - 1/2, & \vert x \vert \leq 2,\\
        x^2/4 - 1/2 - \vert x \vert/4 \cdot \sqrt{x^2 - 4} + \log\{\sqrt{x^2/4  - 1} + \vert x \vert/2 \}, & \vert x \vert > 2.
\end{array}\right. \label{eq:PhiDef}
\end{align}

\begin{theorem}\label{thm:critical_point}
For any Borel sets $M \subset [-1, 1]$ and $E \subset \R$, assume $\lambda$ is fixed.
Then, we have 
\begin{align}
\limsup_{n\rightarrow \infty} \Big\{ \frac{1}{n} \log \E \big[ \Crt_{n, \star}(M,E) \big] - \sup_{m\in \Mb, e \in \Eb }S_{\star}(m,e)\Big\} & \le 0\, ,\label{eq:ThmCrUB}\\
\liminf_{n\rightarrow \infty} \Big\{\frac{1}{n} \log \E \{\Crt_{n, \star}(M,E)\} - \sup_{m\in M^o, e \in E^o }S_{\star}(m,e)\Big\} & \ge 0 \, . \label{eq:ThmCrLB}
\end{align}
\end{theorem}

\subsection{Complexity of local maxima}

For any Borel set $E \subset \R$ and $M \in [-1,1]$, we define $\Crt_{n, \knot} (M,E)$ to be the number of local maxima of $f$ with function value in $E$ and correlation in $M$:
\begin{equation}
\Crt_{n, \knot} (M,E) \eqndef \sum_{\bsigma: \tgrad f(\bsigma) = \bzero} \ones\{\< \bsigma, \bu \> \in M\} \ones\{ f(\bsigma) \in E \} \ones\{ \thess f(\bsigma) \preceq 0 \}. 
\end{equation}

We define function $S_{\knot}: [-1,1] \times \R \rightarrow \bar \R$ as
\begin{align}
S_{\knot}(m,x) \eqndef S_{\star}(m,x) - L(\theta(m), t(x)), 
\end{align}
where
\begin{align}\label{eqn:LDP_spiked}
L(\theta, t) =\left\{ \begin{array}{lr} 
\frac{1}{4} \int_{\theta + \frac{1}{\theta}}^t \sqrt{y^2 - 4} \cdot \d y - \frac{1}{2} \theta \Big[t - \Big(\theta + \frac{1}{\theta}\Big)\Big] + \frac{1}{8}\Big[t^2 - \Big(\theta + \frac{1}{\theta}\Big)^2\Big], &  2 \leq t < \theta + \frac{1}{\theta}, 1 < \theta, \\
\infty, & t < 2,\\
0, & \text{otherwise}.
\end{array} \right.
\end{align}
and $\theta = \theta(m) = \sqrt{2k(k-1)} \cdot \lambda m^{k-2} (1-m^2)$, $t = t(x) = \sqrt{2k/(k-1)} \cdot x$. We also note that
\begin{align}
\int_{2}^t \sqrt{y^2 - 4} \cdot \d y = t\sqrt{\frac{t^2}{4}-1}-2\log\Big(\frac{t}{2}+\sqrt{\frac{t^2}{4}-1}\Big)\, .
\end{align}

\begin{theorem}\label{thm:local_maxima}
For any Borel sets $M \subset [-1, 1]$ and $E \subset \R$, assume $\lambda$ is fixed. Then, we have
\begin{align}
\limsup_{n\rightarrow \infty} \Big\{\frac{1}{n} \log \E \big[\Crt_{n, \knot}(M,E) \big] - \sup_{m\in \Mb, e \in \Eb }S_{\knot}(m,e)\Big\}& \le 0\,, \label{eq:ThmLmUB}\\
\liminf_{n\rightarrow \infty} \Big\{\frac{1}{n} \log \E \{\Crt_{n, \knot}(M,E)\} - \sup_{m\in M^o, e \in E^o }S_{\knot}(m,e)\Big\}& \ge 0\, . \label{eq:ThmLmLB}
\end{align}
\end{theorem}

\subsection{Evaluating the complexity function} \label{sec:eval_complexity}

\begin{figure}[t]
\centering
\includegraphics[width=0.45\textwidth]{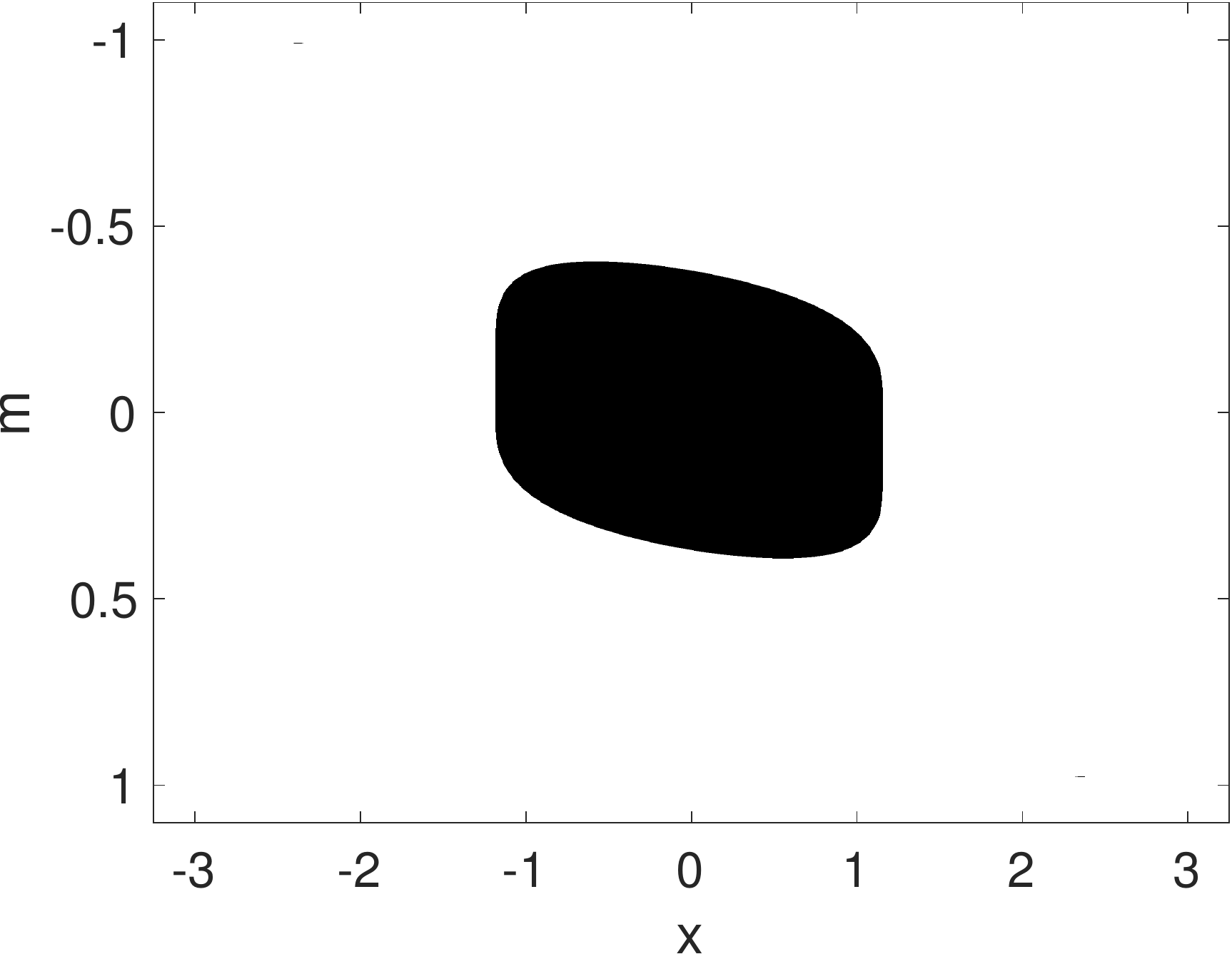}
\includegraphics[width=0.45\textwidth]{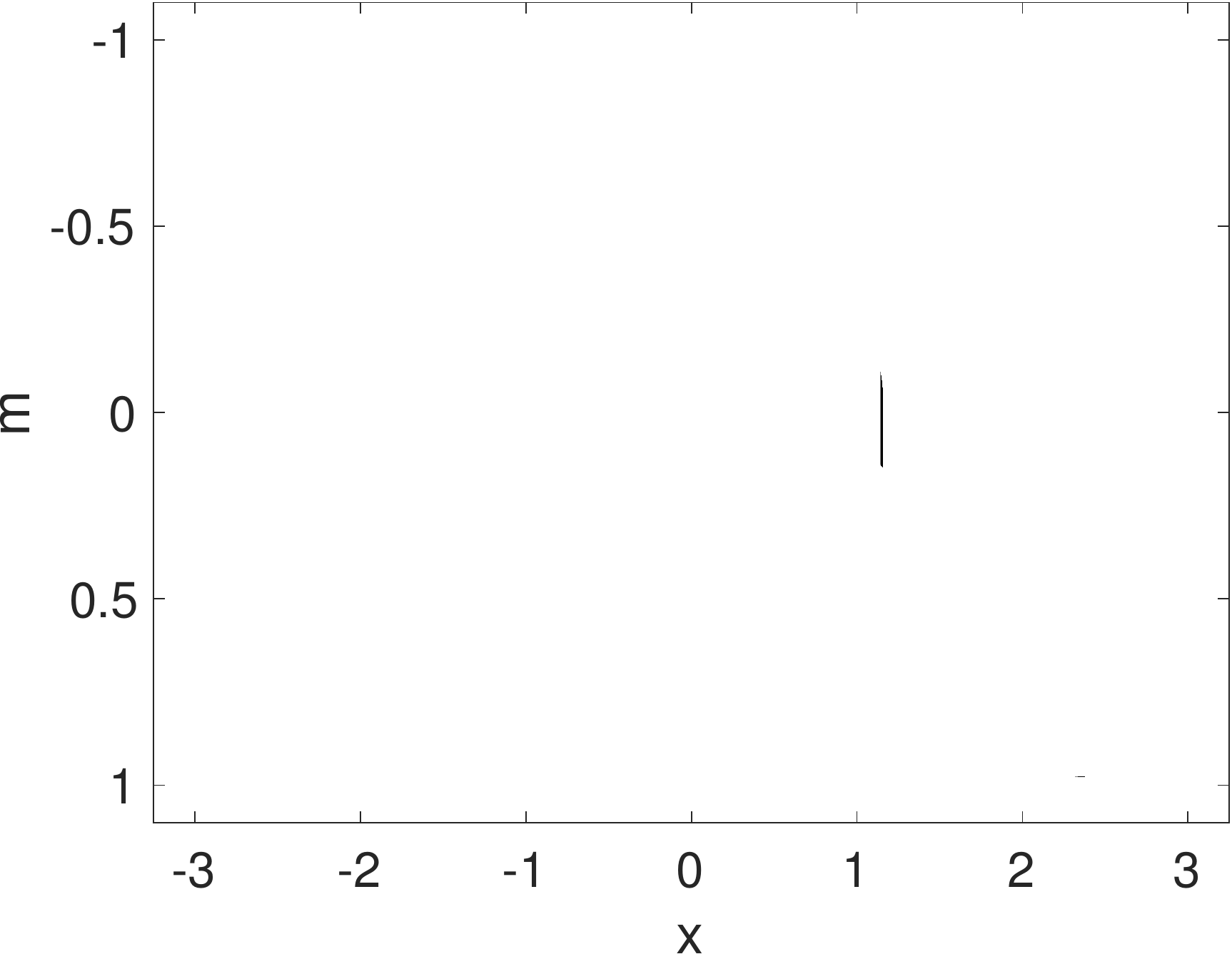}
\put(-245,38){$\downarrow$}
\put(-29,38){$\downarrow$}
\caption{Spiked tensor model with $k = 3$ and $\lambda = 2.25$. The black region corresponds to non-negative complexity:
$S_\star(m,x) \ge 0$ (left) and  $S_\knot(m,x) \ge 0$ (right). The arrows indicate the point where the complexity touches zero,
in correspondence with the `good' local maxima.}\label{fig:landscapeRegion}
\end{figure}

\begin{figure}[tp]
\phantom{A}
\vspace{-1cm}
\begin{tabular}{l}
\includegraphics[width=0.4\textwidth]{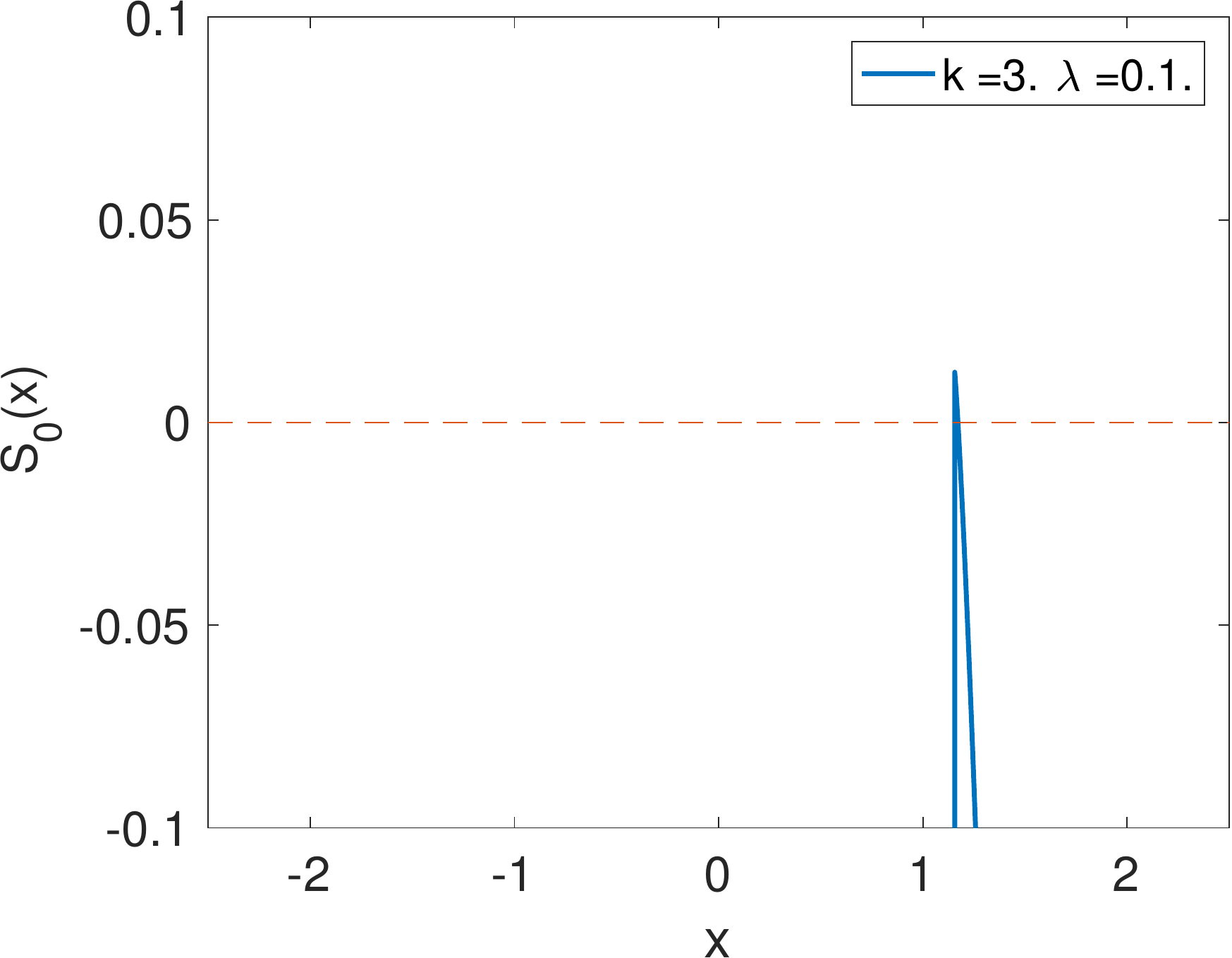}\phantom{AAAA}
\includegraphics[width=0.4\textwidth]{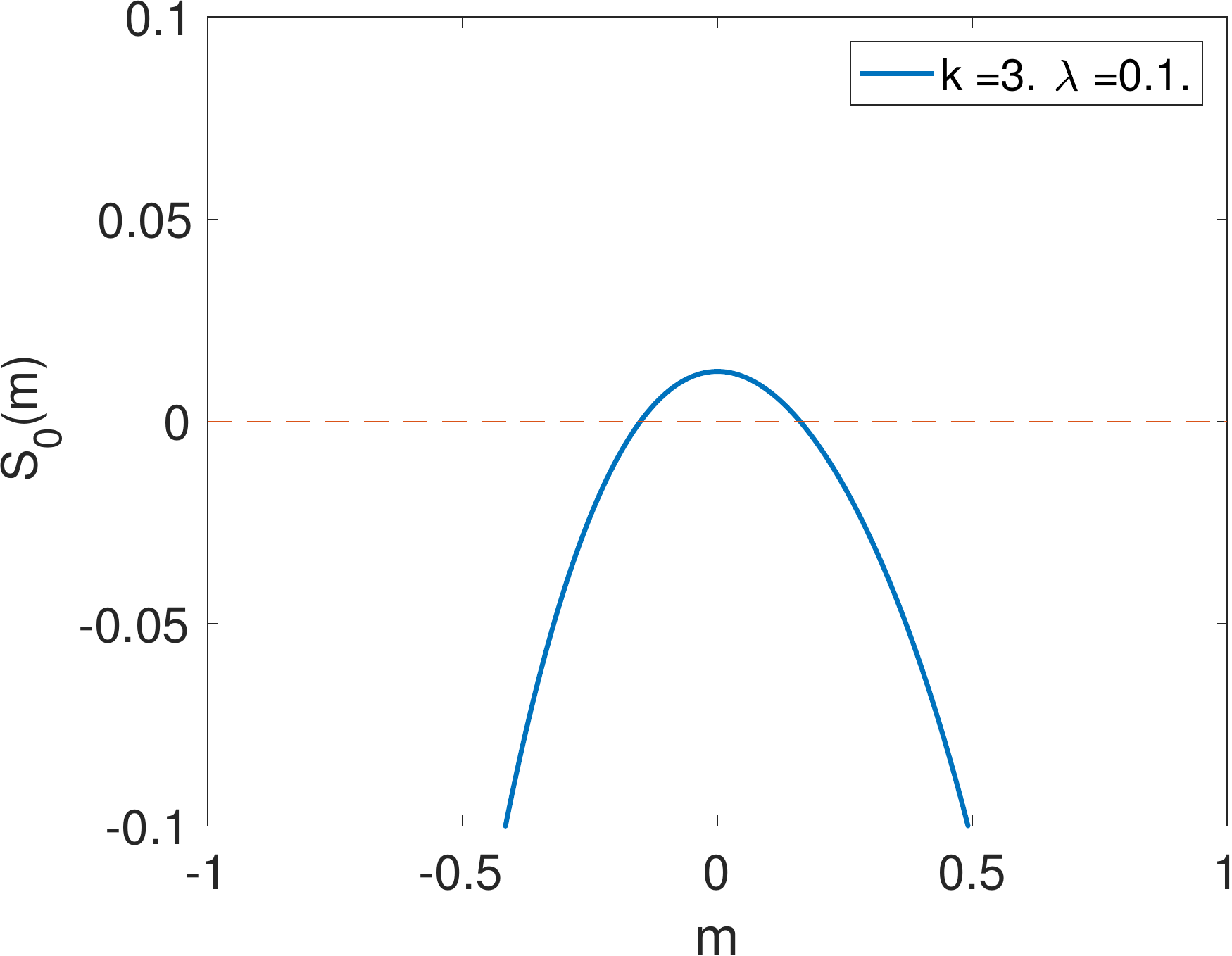}\\
\includegraphics[width=0.4\textwidth]{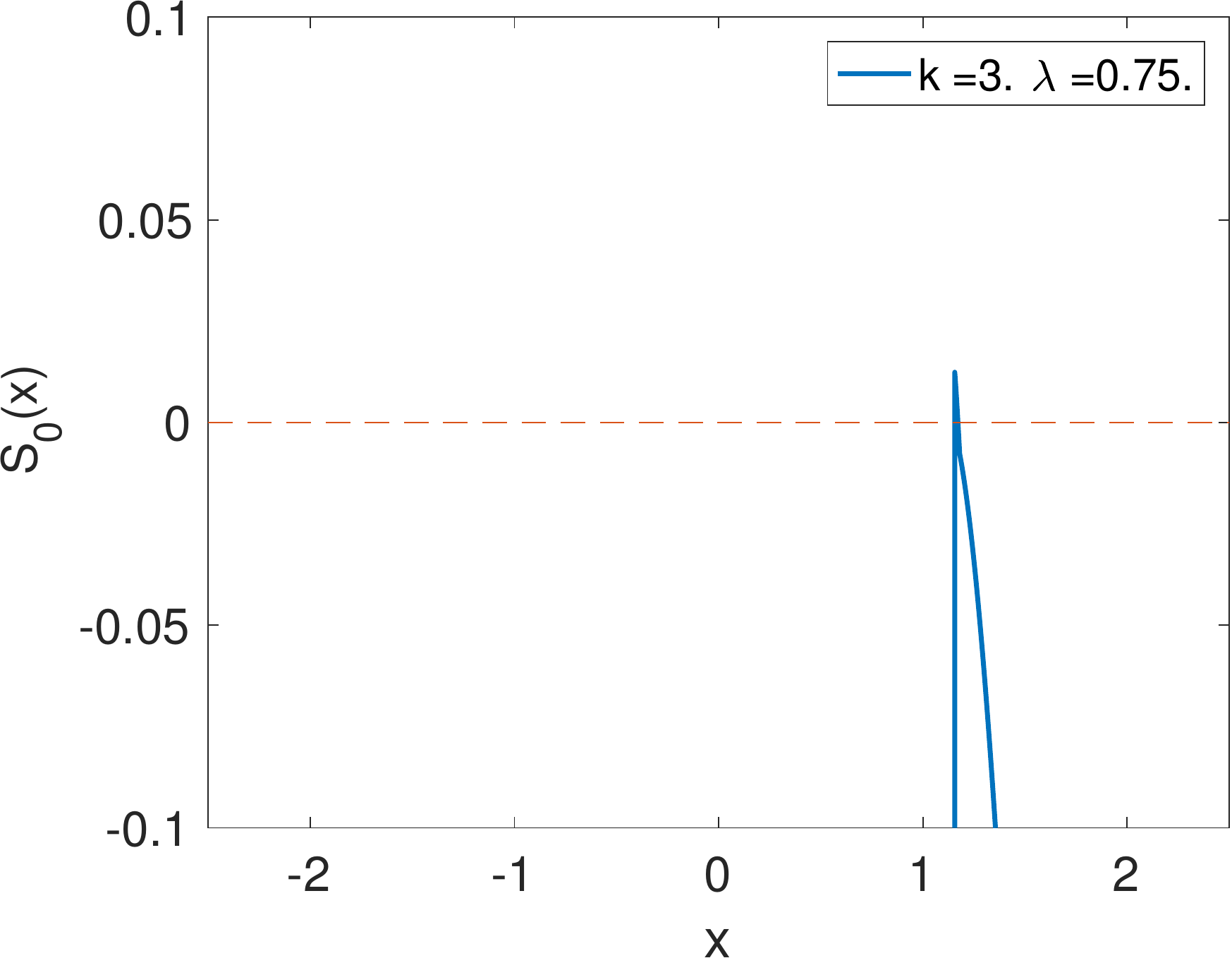}\phantom{AAAA}
\includegraphics[width=0.4\textwidth]{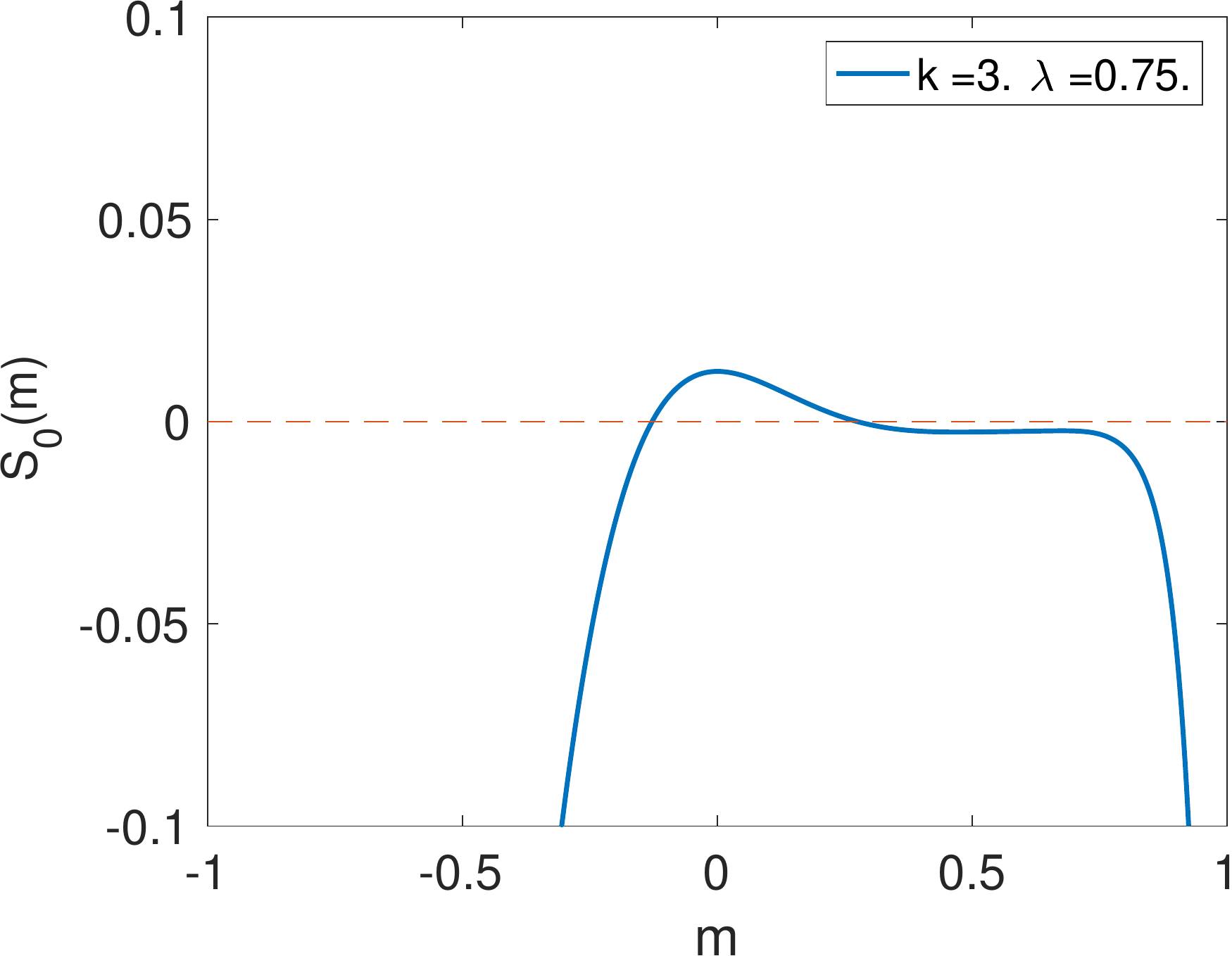}\\
\includegraphics[width=0.4\textwidth]{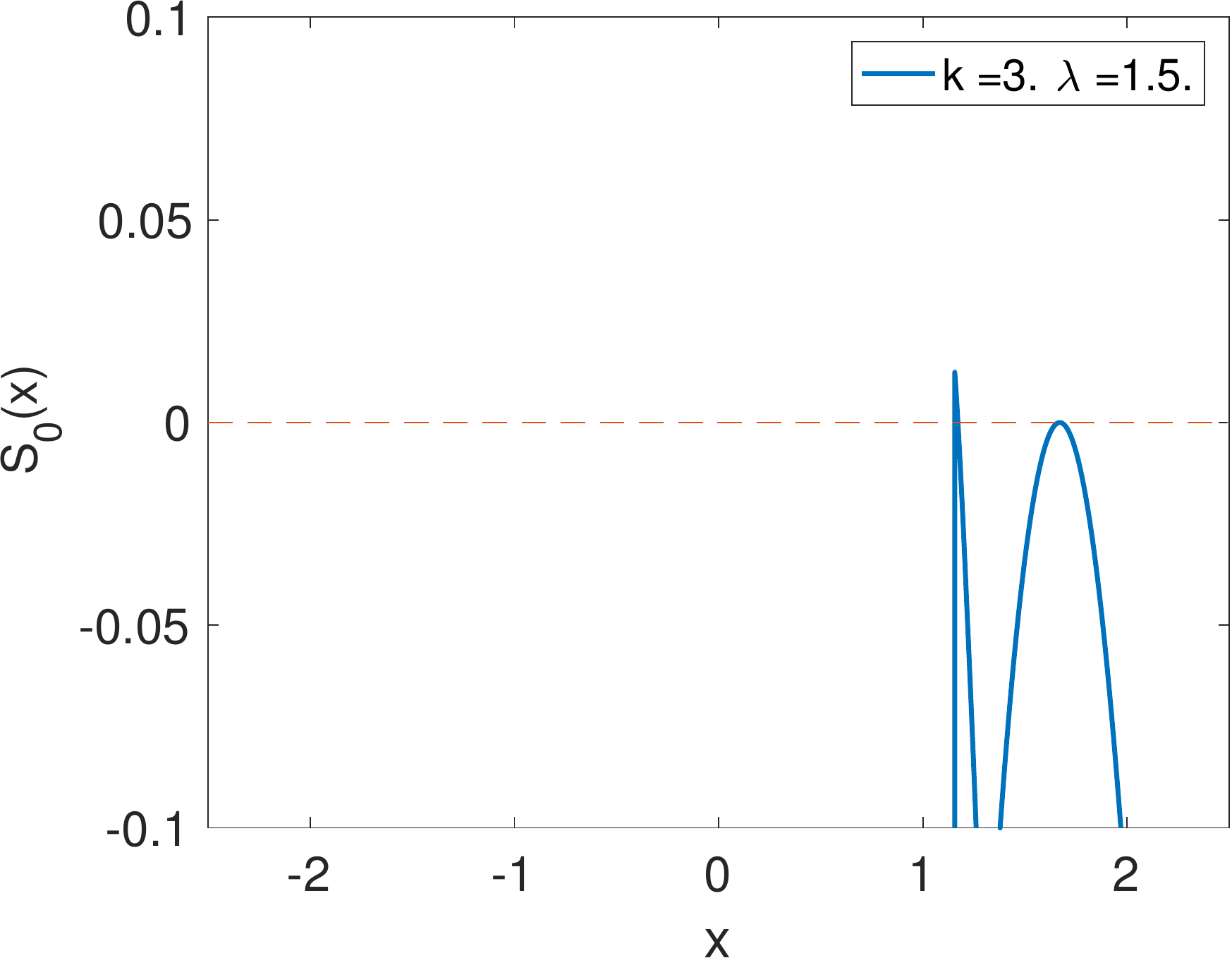}\phantom{AAAA}
\includegraphics[width=0.4\textwidth]{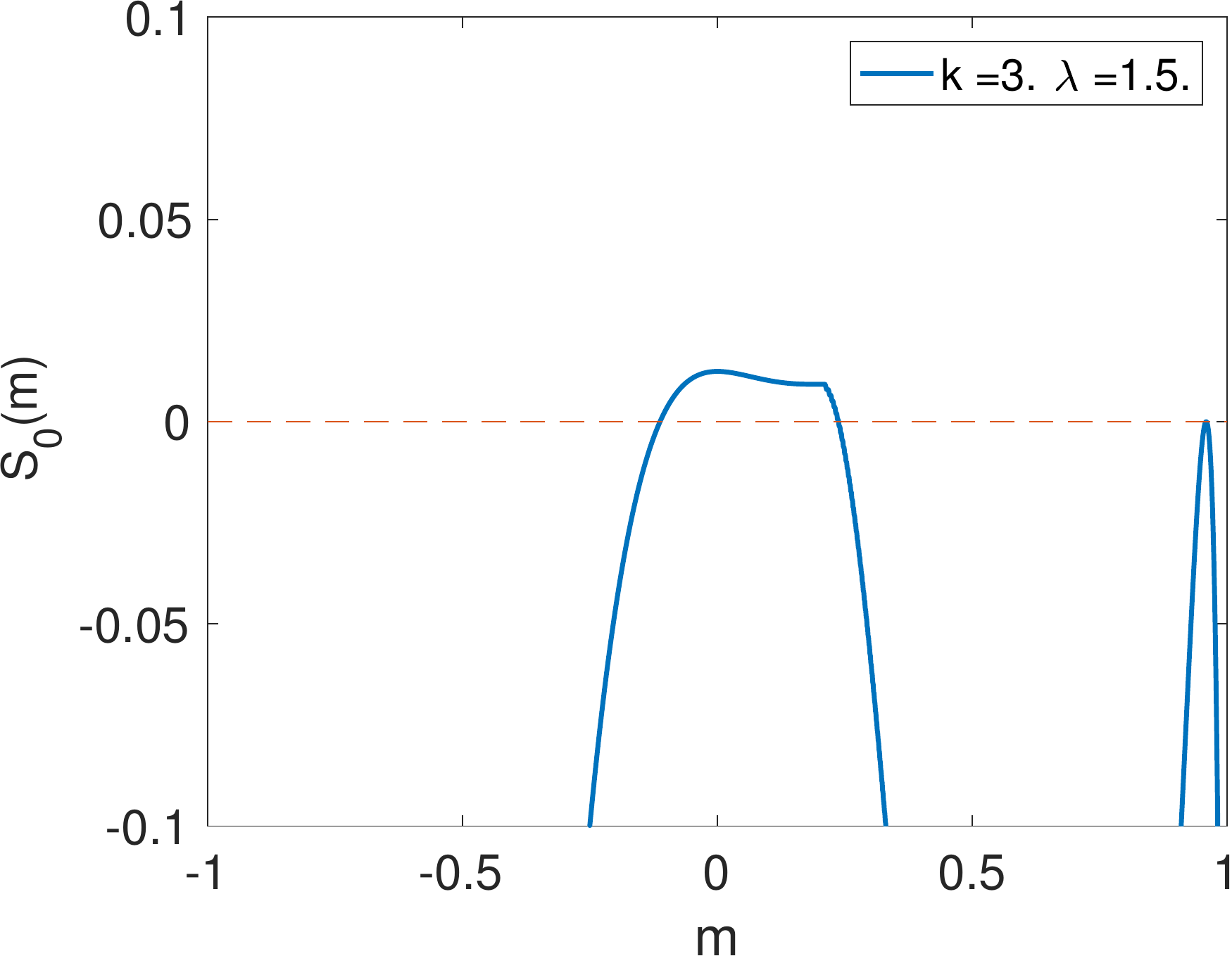}\\
\includegraphics[width=0.4\textwidth]{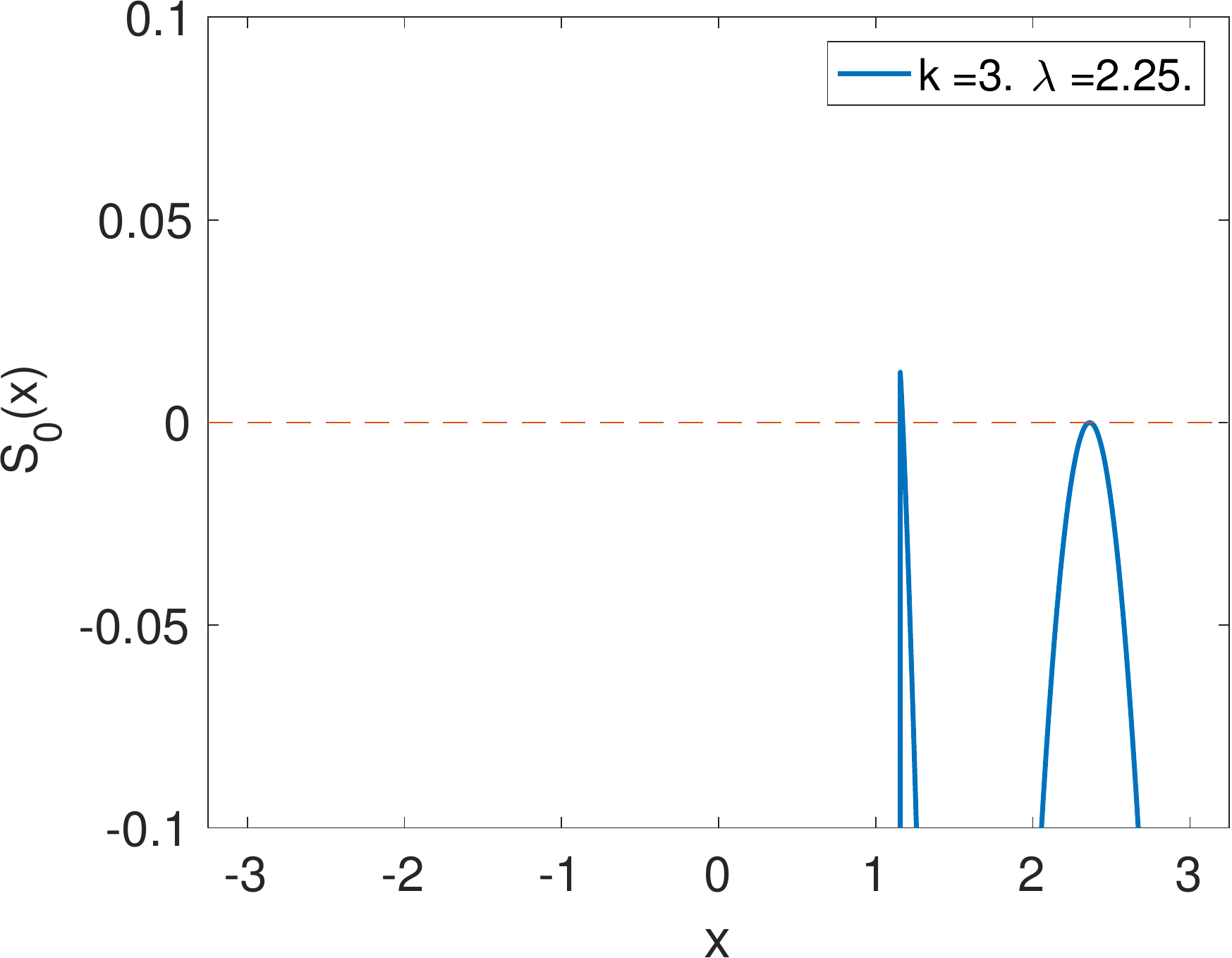}\phantom{AAAA}
\includegraphics[width=0.4\textwidth]{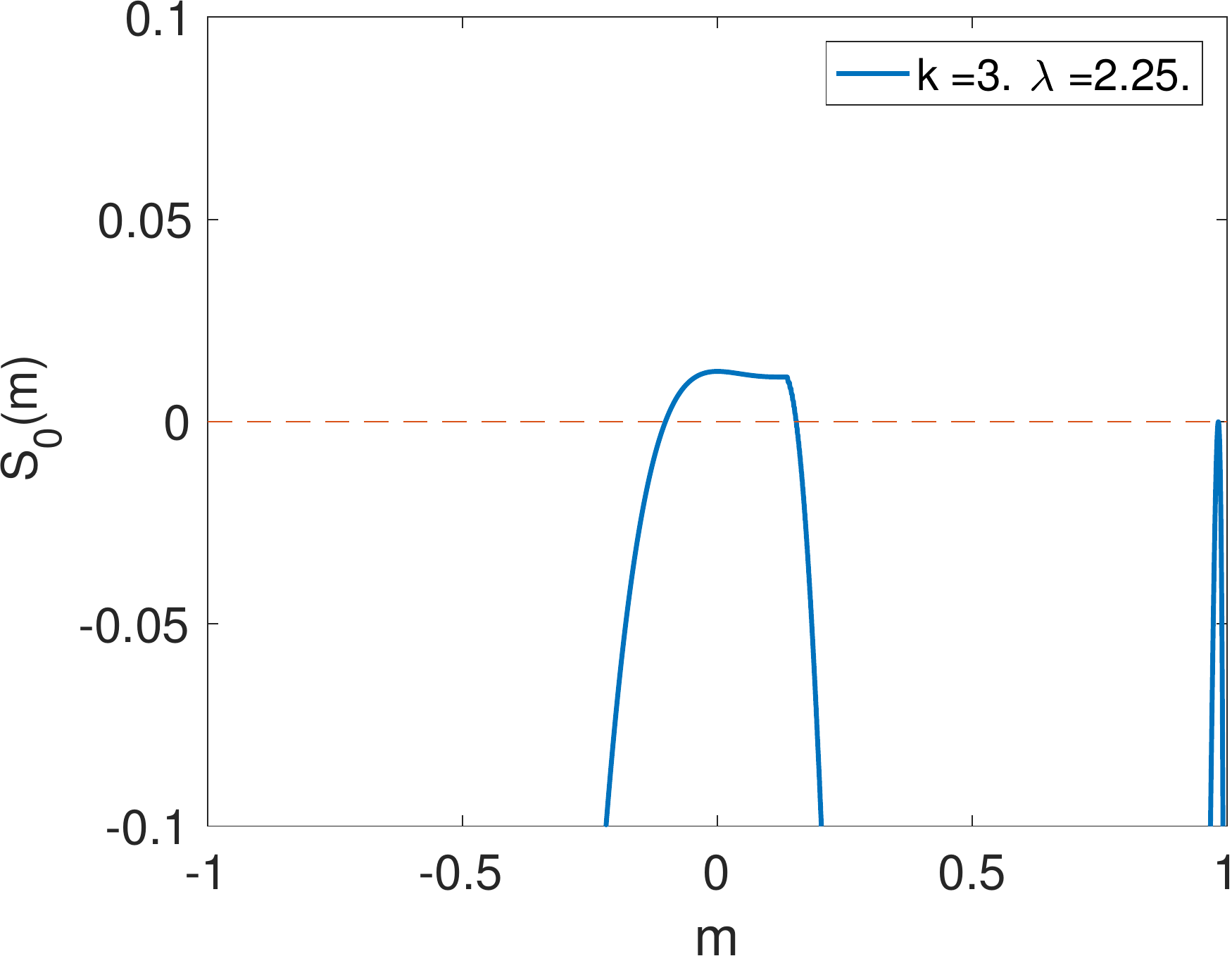}
\end{tabular}
\caption{Complexity (exponential growth rate of the expected number of local maxima) in the spiked tensor model with  $k=3$ and 
(from top to bottom) $\lambda\in\{0.1,0.75,1.5,2.25\}$. Left column: complexity as a function of the objective value 
$x = f(\bsigma)$, $S_\knot(x) = \max_m S_\knot(m,x)$. Right column: complexity as a function of the scalar product
$m=\<\bu,\bsigma\>$, $S_\knot(m)=\max_x S_{\knot}(m,x)$.}\label{fig:landscapeEvolution1} 
\end{figure}

\begin{figure}[tp]
\phantom{A}
\vspace{-1cm}
\begin{tabular}{l}
\includegraphics[width=0.4\textwidth]{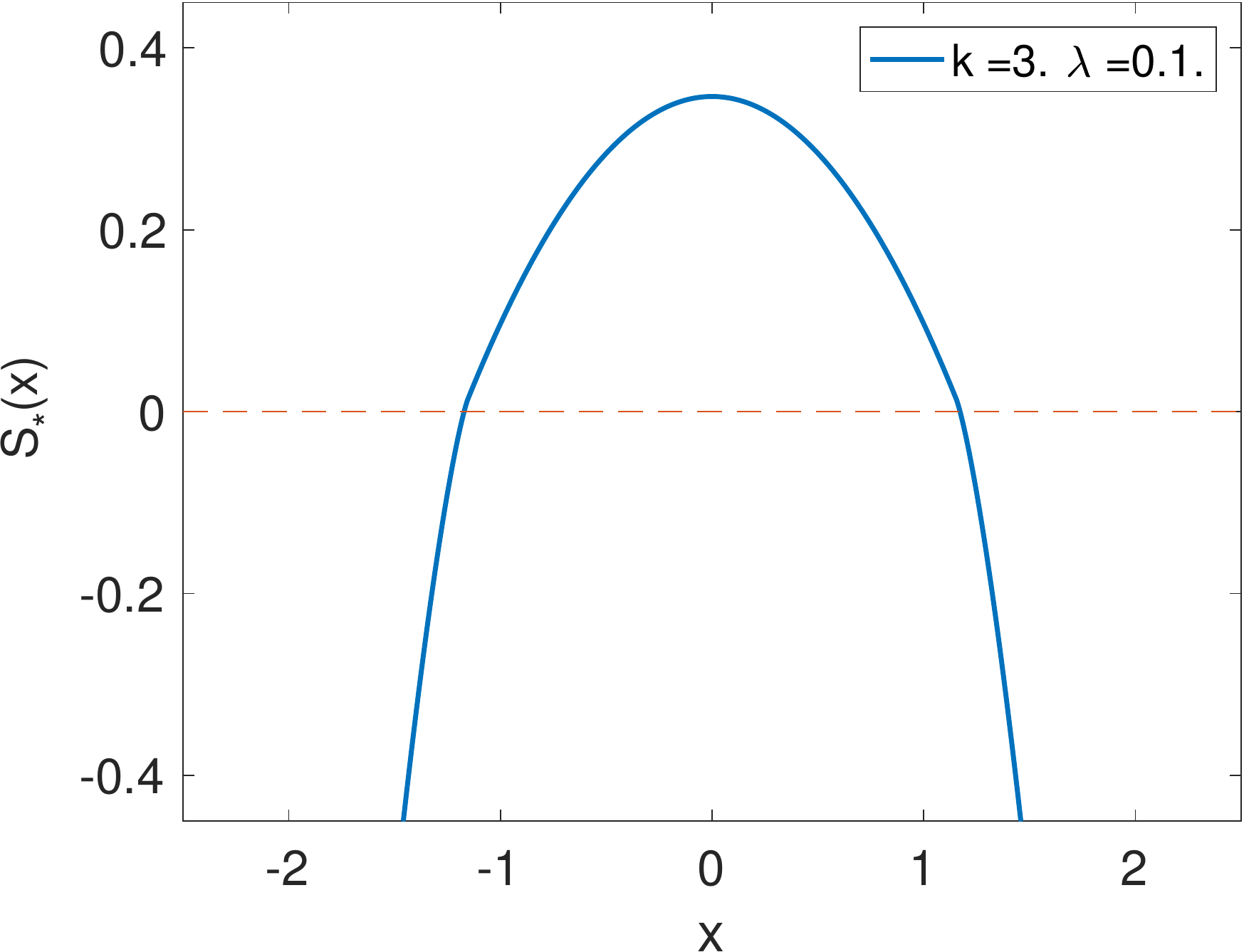}\phantom{AAAA}
\includegraphics[width=0.4\textwidth]{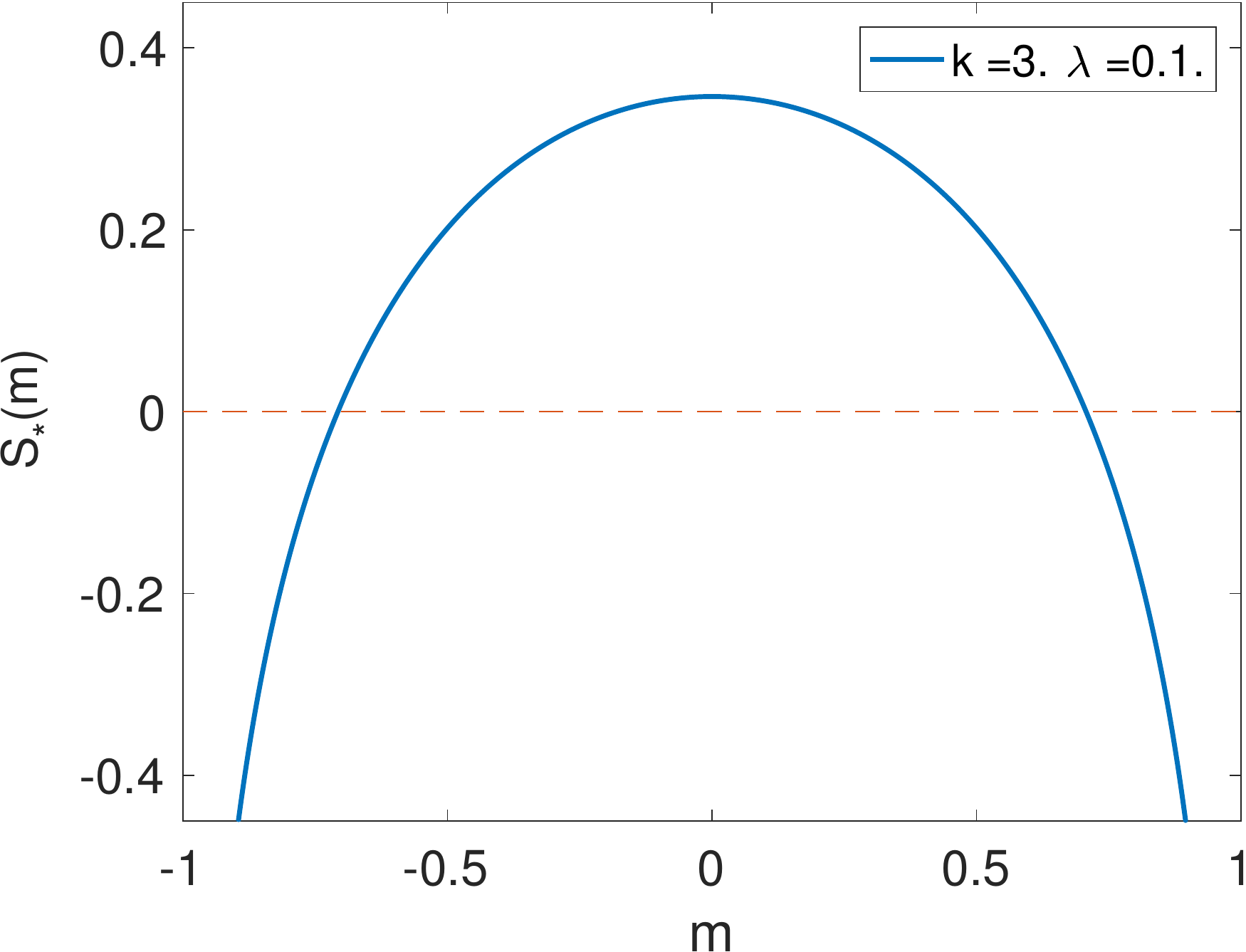}\\
\includegraphics[width=0.4\textwidth]{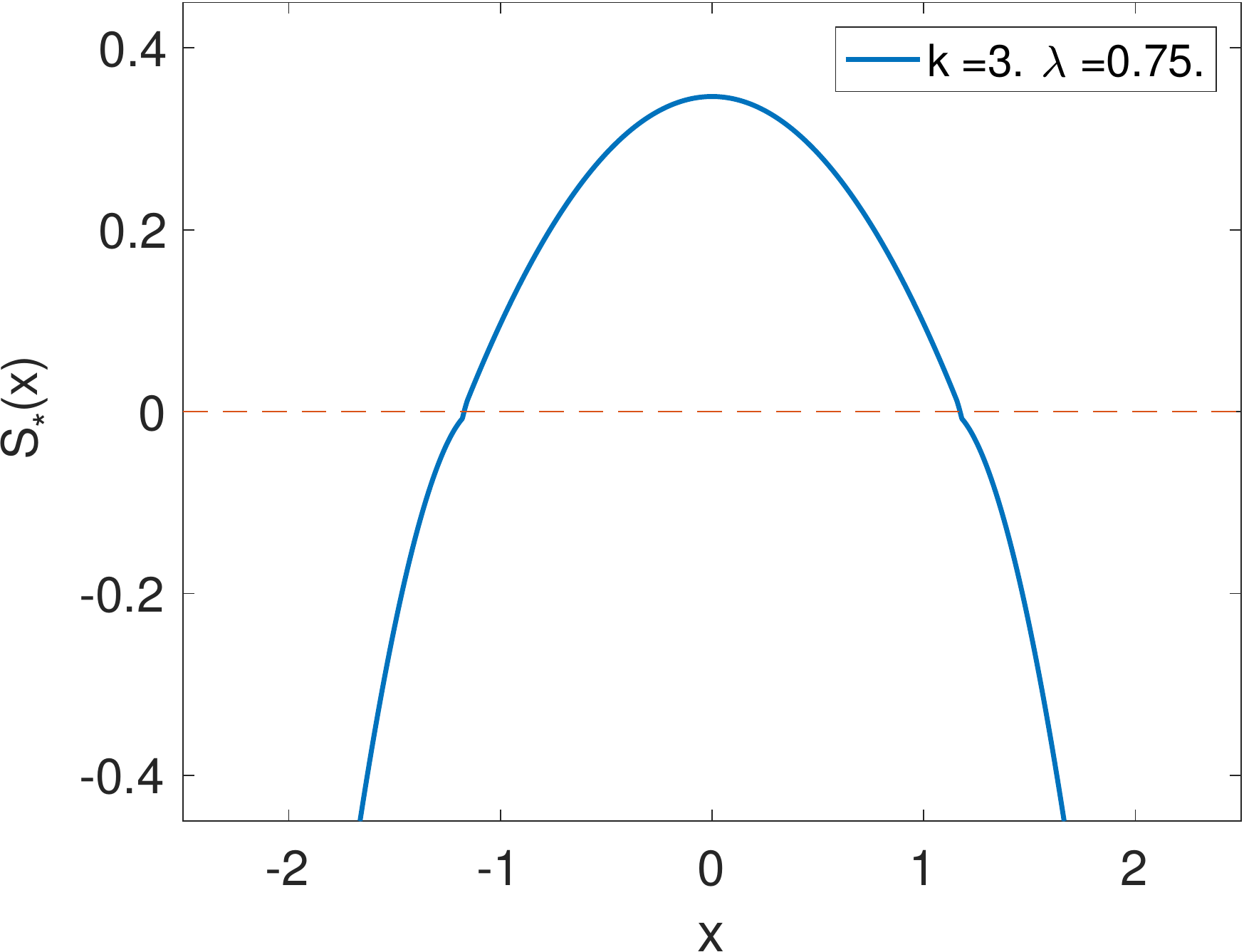}\phantom{AAAA}
\includegraphics[width=0.4\textwidth]{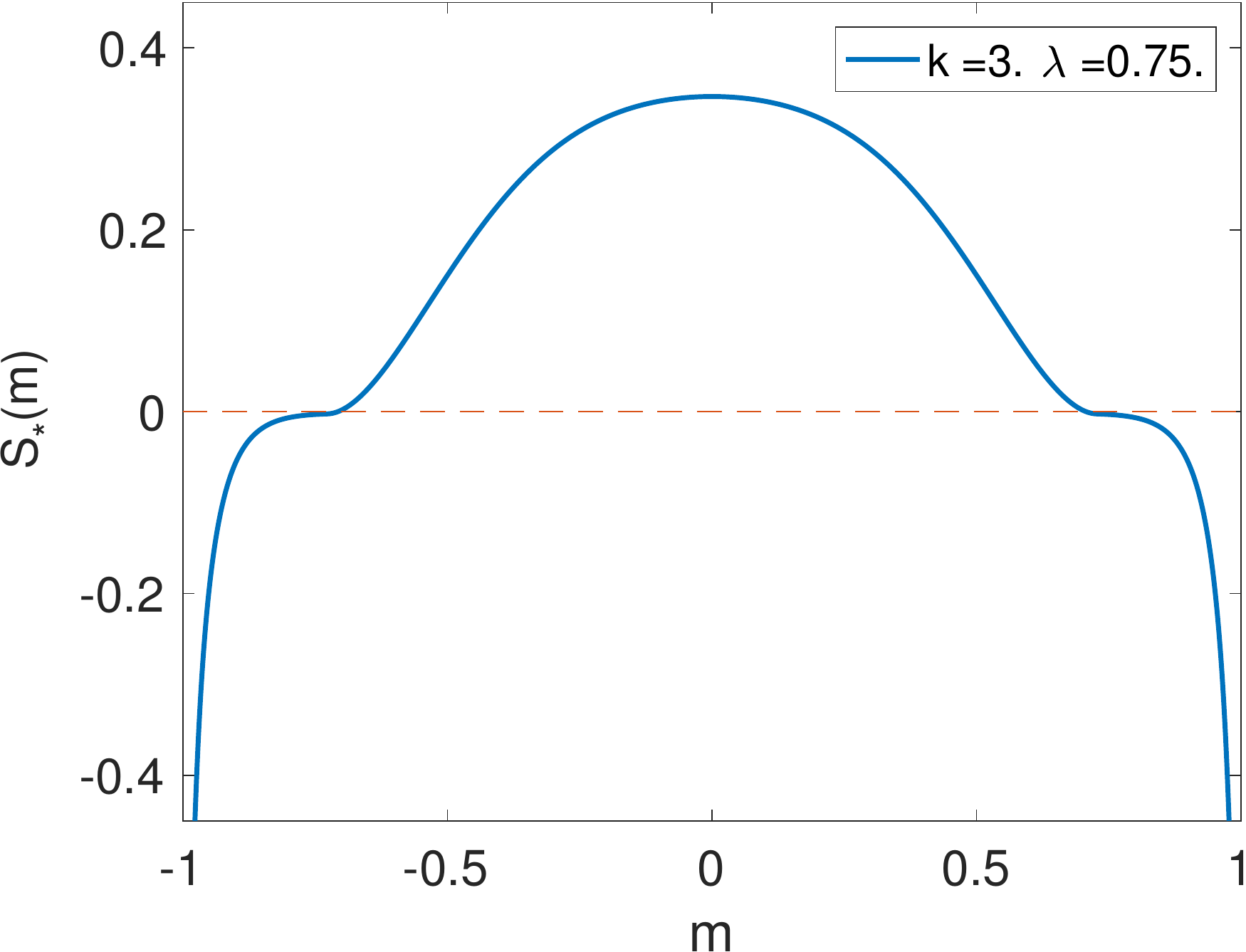}\\
\includegraphics[width=0.4\textwidth]{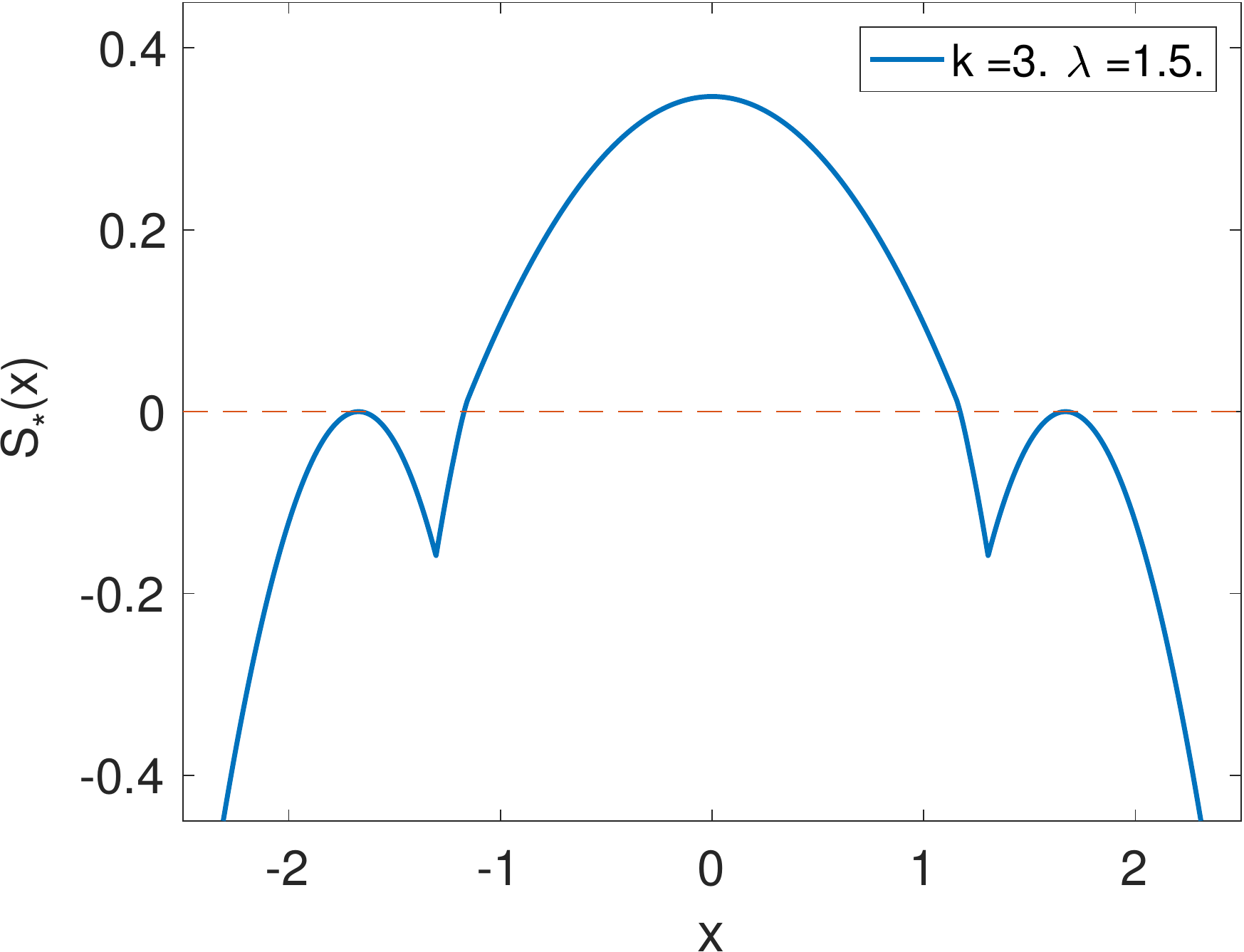}\phantom{AAAA}
\includegraphics[width=0.4\textwidth]{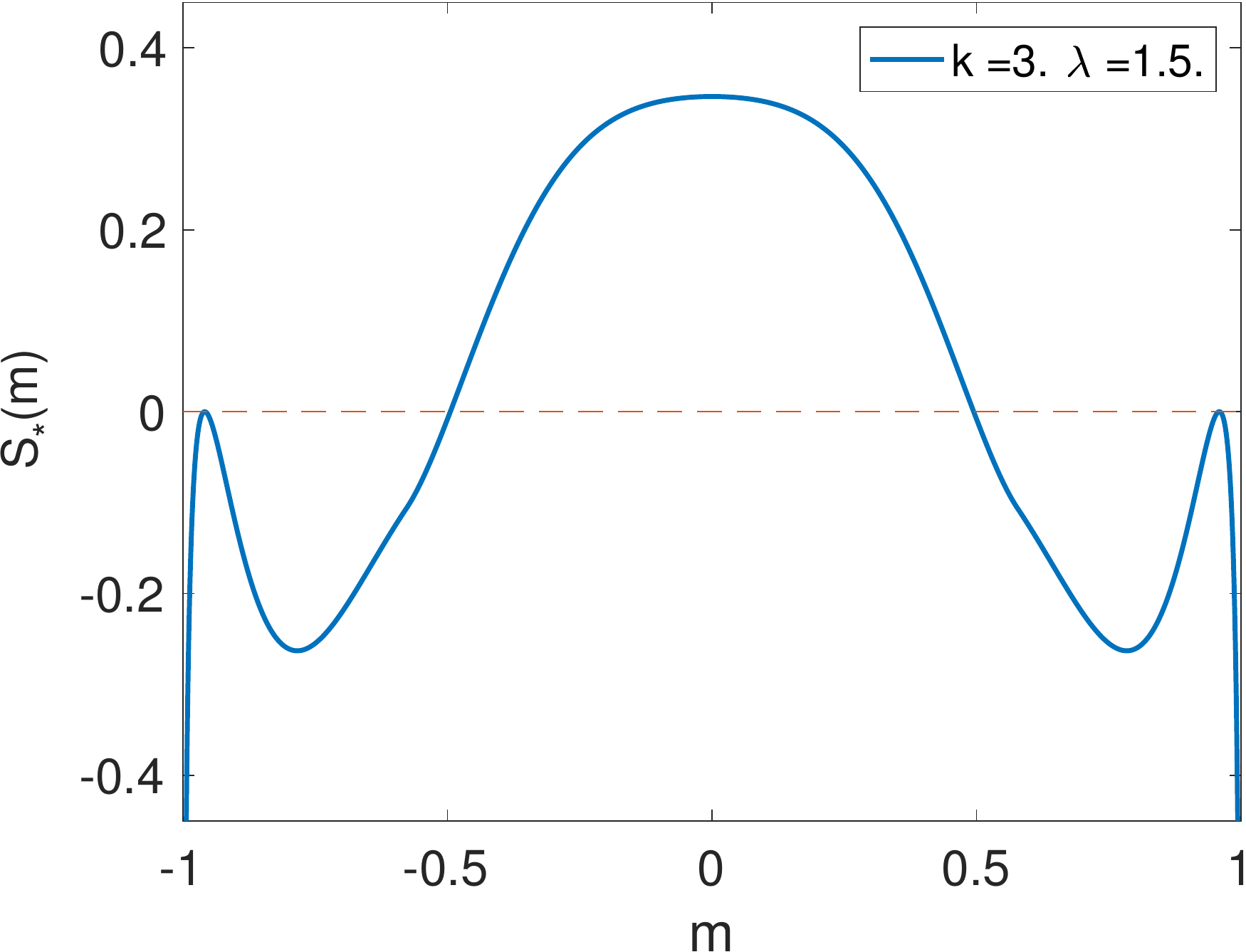}\\
\includegraphics[width=0.4\textwidth]{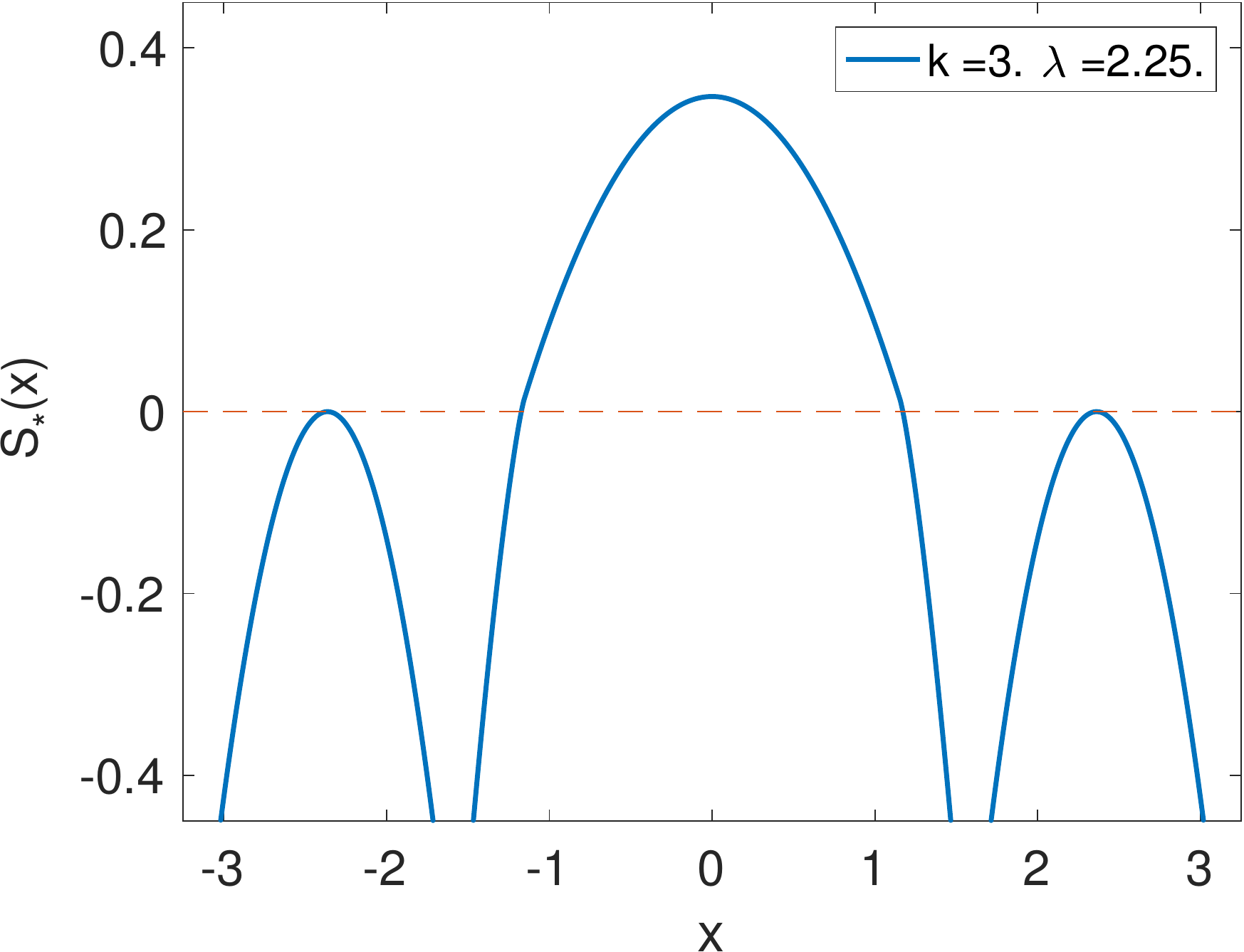}\phantom{AAAA}
\includegraphics[width=0.4\textwidth]{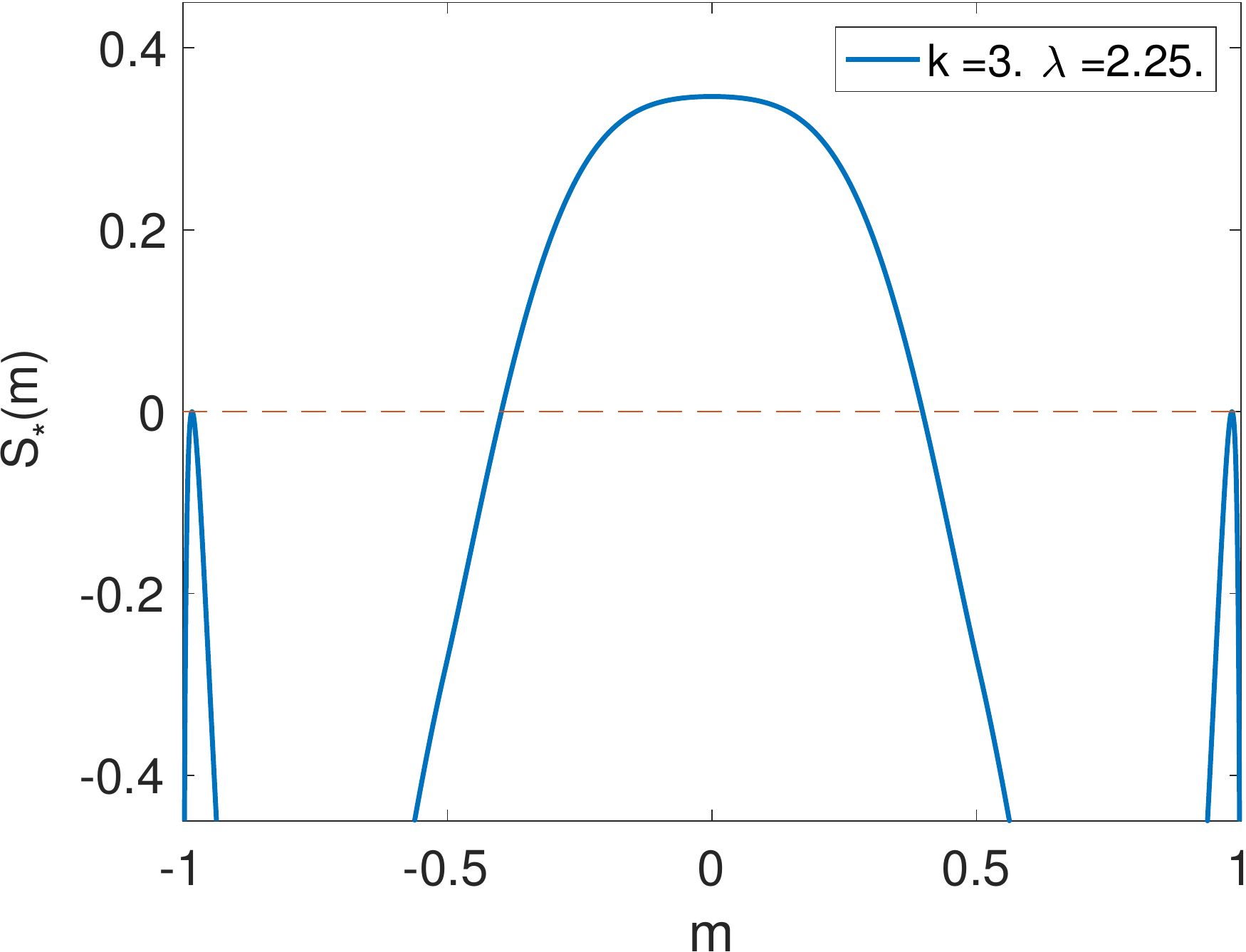}
\end{tabular}
\caption{Complexity (exponential growth rate of the expected number of critical points) in the spiked tensor model with  $k=3$ and 
(from top to bottom) $\lambda\in\{0.1,0.75,1.5,2.25\}$. Left column: complexity as a function of the objective value 
$x = f(\bsigma)$, $S_\star(x) = \max_m S_\star(m,x)$. Right column: complexity as a function of the scalar product
$m=\<\bu,\bsigma\>$, $S_\star(m)=\max_x S_{\star}(m,x)$. }\label{fig:landscapeEvolution2}
\end{figure}

The expressions for $S_\star(m,x)$ and $S_\knot(m,x)$ given in the previous section can be easily evaluated numerically:
the figures in this section demonstrate such evaluations. Throughout this section we consider $k=3$, but the behavior
for other values of $k\ge 3$ is qualitatively similar (with the important difference that, for $k$ even, the landscape is symmetric under 
change of sign of $m$). In Figure \ref{fig:landscapeFirst} we plot the region of the $(m,x)$ plane in which $S_\star(m,x)$ and $S_\knot(m,x)$ are non-negative, 
for $\lambda=2.25$. By Markov inequality, the probability of any critical point or any local maximum to be present outside these regions is exponentially small.

As anticipated in the introduction, we can identify two sets of local maxima:
\begin{itemize}
\item[$(i)$] \emph{Uninformative local maxima.} These have small $x$ (i.e. small value of the objective) and small $m$ (small correlation with the ground truth $\bu$).
They are also exponentially more numerous and we expect them to trap descent algorithms.
\item[$(ii)$] \emph{Good local maxima.} These have large $x$ (i.e. large  value of the objective) and large $m$ (large correlation with the ground truth $\bu$).
Reaching such a local maximum results in accurate estimation.
\end{itemize}

Figure \ref{fig:landscapeEvolution1} shows the evolution of the two `projections' $S_\knot(x) =\max_m S_{\knot}(m,x)$ and  $S_\knot(m) =\max_x S_{\knot}(m,x)$ 
which give the exponential growth rate of the number of local maxima as functions of the objective value $x=f(\bsigma)$ and the scalar product
$m=\<\bu,\bsigma\>$. Similar plots for the total number of critical points are found in Figure \ref{fig:landscapeEvolution2}.
We can identify several regimes of the signal-to-noise ratio $\lambda$:
\begin{enumerate}
\item For $\lambda$ small enough, we know that the landscape is the qualitatively similar to the case $\lambda=0$:
local maxima are uninformative. While they are spread along the $m$ direction, this is purely because of random fluctuations.
Local maxima with $m\approx 0$ are exponentially more numerous and have larger value.
\item As $\lambda$ crosses a threshold $\lambda_c$, the complexity develop a secondary maximum that touches $0$ at $m_*(\lambda)$ close to $1$.
This signals a group of local maxima  (or possibly only one of them) that are highly correlated with $\bu$. These are good local maxima, but
have smaller value than the best uninformative local maxima. Maximum likelihood estimation still fails.
\item Above a third threshold in $\lambda$, good local maxima acquire a larger value of the objective than uninformative ones. Maximum likelihood
succeeds. However, the most numerous local maxima are still uncorrelated with the signal $\bu$ and are likely to trap
algorithms.
\end{enumerate}

Let us emphasize once more that this qualitative picture is obtained by considering the \emph{expected number} of critical points.
In order to confirm that it holds for a typical realization of $\bY$, it would be important to compute the typical number as well.

\subsection{Explicit formula for complexity of critical points at a given location}
\label{sec:Explicit}

The projection $S_\star(m) = \max_{x} S_\star(m,x)$, which gives the expected number of critical points at a given scalar product
$m=\<\bu,\bsigma\>$,  has a simple explicit formula in the hemisphere $m \in [0,1]$. This is derived using elementary calculus by analyzing equation (\ref{eq:StarDef}). 

\begin{proposition} \label{prop:S_star_formula} The projection $S_\star(m) = \max_{x} S_\star(m,x)$ has the following explicit formula for $m \in [0,1]$:
\begin{equation}
S_\star(m)=\begin{cases}
S_{U}(m) & 0 \leq m <m_{c}\\
S_{G}(m) & m\geq m_{c}
\end{cases},
\end{equation}
where
\begin{align}
m_{c} := & \left(\frac{1}{\lambda}\frac{k-2}{\sqrt{2k(k-1)}}\right)^{1/k},\\
S_{U}(m) := & \frac{1}{2}\log(1-m^{2})-k\lambda^{2}m^{2k-2}(1-m^{2})+\frac{k}{k-2}\lambda^{2}m^{2k}+\frac{1}{2}\log(k-1),\\
S_{G}(m) := & \frac{1}{2}\log\left(1-m^{2}\right)-k\lambda^{2}m^{\left(2k-2\right)}(1-m^{2})-\left(\sqrt{\frac{1}{2} k}\lambda\cdot m^{k}\right)^{2}\\
 & +\sqrt{\frac{1}{2} k}\cdot\lambda \cdot m^{k}\cdot\sqrt{\left(1+\frac{1}{2} k\cdot\lambda^{2}m^{2k}\right)}+\sinh^{-1}\left(\sqrt{\frac{1}{2} k}\lambda m^{k}\right).
\end{align}

\end{proposition}

Analysis of this formula confirms some of the qualitative observations from Section \ref{sec:eval_complexity}. For $\lambda$ very small, namely $\lambda < (k-2)/\sqrt{2k(k-1)}$, we have that $m_c > 1$. In this case, $S_\star(m) \equiv S_U(m)$ and landscape is qualitatively similar to the case $\lambda = 0$. When $\lambda \geq (k-2)/\sqrt{2k(k-1)}$, we have that $m_c \leq 1$ and the function $S_G$ captures the behavior of possible ``good'' critical points which may exist at $m > m_c$. Further analysis of the function $S_G$ is carried out in Proposition \ref{prop:S_G_analysis}.

\begin{proposition} \label{prop:S_G_analysis}
The function $S_{G}$ is non-positive, $S_{G}(m)\leq0$ for all $m\in[0,1]$. Moreover, $S_{G}(m) = 0$ if and only if $m$ satisfies:
\begin{equation}\label{eq:good_location}
m^{2k-4}(1-m^{2})=\frac{1}{2k\lambda^{2}}\, .
\end{equation}
In particular, if we set:
\begin{equation}
\lambda_c := \sqrt{\frac{1}{2k} \frac{(k-1)^{(k-1)} }{ (k-2)^{(k-2)}  }  }, \label{eq:lambda_c}
\end{equation}
then we have that if $\lambda < \lambda_c$, then $S_{G}(m) < 0$ and if $\lambda \geq \lambda_c$, then $S_{G}$ has a unique zero in the domain $m \in [m_c,1]$.
\end{proposition}

The critical point $\lambda_c$ identified in Proposition \ref{prop:S_G_analysis} represents a qualitative change in the energy landscape. When $\lambda < \lambda_c$, then $S_G < 0$ and ``good'' critical points are exponentially rare. On the other hand, when $\lambda  \geq \lambda_c$ then $S_G$ and has a unique zero. This is the only location in the region $m > m_c$ where critical points are not exponentially rare, and this represents the best correlation with the signal $\bu$ that is achievable.

The proofs of Proposition \ref{prop:S_star_formula} and \ref{prop:S_G_analysis} are deferred to Appendix \ref{sec:calc_proofs}

\subsection{Proof ideas}

The proof of Theorems \ref{thm:critical_point} and \ref{thm:local_maxima} relies on a representation of the expected number of
critical points of a given index using the Kac-Rice formula. This approach was pioneered in \cite{fyodorov2004complexity,auffinger2013random} to study the case 
$\lambda=0$ of the present problem. 

Evaluating the expression produced by the Kac-Rice formula requires to estimate the expectation of determinant of $\thess f(\bsigma)$.
In the case  $\lambda=0$ considered in \cite{auffinger2013random}, $\thess f(\bsigma)$ is distributed as  $a\, \bW_n+b\, \id_n$ where $\bW_n\sim \GOE(n)$
is a matrix form the Gaussian Orthogonal Ensemble. This fact, together with the explicitly known joint distribution of the eigenvalues of $\bW_n$,
 is used in \cite{auffinger2013random} to express the expected determinant in terms of  the distribution of one eigenvalue, and a normalization 
that is computed using Selberg's integral.

In the present case, $\thess f(\bsigma)$ is distributed as  $a\, \bW_n+b\, \id_n+c\, \be_1\be_1^{\sT}$, i.e. a rank-one deformation of the previous matrix. 
Instead of an exact representation, we use the asymptotic distribution of the eigenvalue of this matrix, as well as its large deviation properties,
obtained in \cite{maida2007large}.

\section{Related literature}
\label{sec:Related}

The complexity of random functions  has been the object of large amount of
work within statistical physics, in particular in the context of mean field glasses and spin glasses.
The function of interest is --typically-- the Hamiltonian or energy function, and its local minima 
are believed to capture the long-time behavior of dynamics, as well as thermodynamic properties.

In particular, the energy function (\ref{eqn:objective}) was first studied by Crisanti and Sommers in \cite{crisanti1995thouless},
for the case  $\lambda =0$. This is referred to as the \emph{spherical $p$-spin model} in the physics literature.
The paper \cite{crisanti1995thouless},  uses non-rigorous methods from statistical physics to derive the complexity function,
which corresponds to $S_0(x)$ in the notations used here. An alternative derivation using random matrix theory was proposed by 
Fyodorov \cite{fyodorov2004complexity}.
Connections with thermodynamic quantities can be found in \cite{crisanti1993sphericalp}.
The impact of the rough energy landscape on the behavior of Langevin dynamics was studied in a number of papers, see e.g.
\cite{crisanti1993sphericalp,bouchaud1998out}.

A mathematically rigorous calculation often expected number of critical points of any index --and the associated complexity-- 
was first carried out in \cite{auffinger2013random}, again for the pure noise case case $\lambda=0$. (See also \cite{auffinger_gba} for mathematically rigorous results for the complexity of some more general ``pure noise'' random surfaces.) As mentioned above, the 
expected number of critical points is not necessarily representative of typical instances. However, for the pure noise case $\lambda=0$,
it was expected that the number of critical points concentrates around its expectation. This was recently confirmed by Subag
via a second moment calculation \cite{subag2015complexity}. (See also \cite{subag2017geometry} and \cite{subag_zeitouni} for additional information about the 
landscape geometry.)

Finally, the typical number of critical points of the spike-tensor model and variants 
is derived in the forthcoming article \cite{physics_spiked_models} using  non-rigorous 
methods from statistical physics. This computation indicates that typical and expected number of critical points 
do not always coincide for the spiked tensor model, contrary to what happens for 
the pure noise case $\lambda = 0$.  Also \cite{physics_spiked_models}  studies generalized spiked models
for which a large number of additional local maxima appear that are strongly correlated with the spike.

\section{Proofs}
\label{sec:Proof}

In this section we prove Theorem \ref{thm:critical_point} and \ref{thm:local_maxima}. We begin by introducing some definitions
and notations in Section \ref{sec:Defs}. We next state some useful lemmas in Section \ref{sec:Preliminary}, with proofs in Sections \ref{sec:ProofJointDensity}
and \ref{sec:ProofExact}. Finally, we prove  Theorems \ref{thm:critical_point} and \ref{thm:local_maxima} in Section \ref{sec:ProofCritical} and
\ref{sec:ProofMaxima}.

\subsection{Definitions and notations}
\label{sec:Defs}

We will generally use lower case symbols (e.g. $a,b,c$) for scalars, lower-case boldface symbols (e.g. $\ba, \bb, \bc$) for vectors,
and upper case boldface (e.g. $\bA,\bB,\bC$) for matrices. The identity matrix in $n$ dimensions is denoted by $\id_n$, and the 
canonical basis in $\R^n$ is denoted by $\be_1, \ldots, \be_n$.
Given a vector $\bv\in\reals^n$, we write $\Proj_{\bv} = \bv\bv^{\sT}/\|\bv\|_2^2$ for the orthogonal projector onto the subspace spanned by $\bv$,
and by $\Projp_{\bv} = \id-\Proj_{\bv}$ the projector onto the orthogonal subspace.

For symmetric matrix $\bB_n\in\reals^{n\times n}$, we denote by $\lambda_1(\bB_n)\ge \lambda_2(\bB_n)\ge \cdots\ge \lambda_n(\bB_n)$ the eigenvalues of $\bB_n$ 
in decreasing order. We will also write $\lambda_{\max}(\bB_n) = \lambda_1(\bB_n)$ and $\lambda_{\min}(\bB_n) = \lambda_n(\bB_n)$ for the maximum and minimum eigenvalues.

We denote by $\GOE(n)$ the Gaussian Orthogonal Ensemble in $n$ dimensions. Namely,
for a symmetric random matrix $\bW$ in $\reals^{n\times n}$, we write $\bW\sim\GOE(n)$ if the entries $(W_{ij})_{i\le j}$
are independent, with $(W_{ij})_{1\le i < j\le n}\sim_{iid}\normal(0,1/n)$ and $(W_{ii})_{1\le i\le n}\sim_{iid}\normal(0,2/n)$.

For a sequence of functions $f_n: \R^d \rightarrow \R_{\ge 0}$, $n \in \N_+$, we say that $f_n(\bx)$ is \emph{exponentially finite} on a set $\cX \subset \R^d$, if
\begin{align}
\limsup_{n\rightarrow \infty}\sup_{\bx \in \cX}\Big\l\frac{1}{n}\log f_n(\bx)  \Big\l  < \infty\, .
\end{align}
We say that $f_n(\bx)$ is \emph{exponentially vanishing} on a set $\cX \subset \R^d$, if 
\begin{align}
\lim_{n\rightarrow \infty}\sup_{\bx \in \cX} \frac{1}{n}\log f_n(\bx) = -\infty\, .
\end{align}
We say that $f_n(\bx)$ is \emph{exponentially trivial} on a set $\cX \subset \R^d$, if 
\begin{align}
\lim_{n\rightarrow \infty}\sup_{\bx \in \cX} \Big\l \frac{1}{n}\log f_n(\bx) \Big\l  = 0\, .
\end{align}
We say $f_n(\bx)$ is \emph{exponentially decaying} on a set $\cX \subset \R^d$, if 
\begin{align}
\limsup_{n\rightarrow \infty}\sup_{\bx \in \cX}  \frac{1}{n}\log f_n(\bx)  < 0\, .
\end{align}

For a metric space $(\mathcal S, d)$, we denote the open ball at $x \in \mathcal S$ with radius $r > 0$ by $\ball(x, r) = \{z \in \mathcal S: d(z, x) < r \}$. In 
$\R^d$, we will always use Euclidean distance. For any $x \in \R$ and $\delta > 0$, the open ball in $\R$ is denoted by $\ball(x, r) = (x - r, x + r)$. 
Let $\cP(\reals)$ be the space of probability measures on $\reals$. We will equip $\cP(\reals)$  with the Dudley distance: for two probability measures 
$\mu,\nu \in \cP(\reals)$, this is defined as
\begin{align}
d(\mu, \nu) = \sup\Big\{ \Big \l \int f \d \mu - \int f \d \nu \Big\l; \l f(x) \l \vee\Big\l \frac{f(x) - f(y)}{x-y} \Big\l \leq 1, \forall x \neq y \Big\}\,.
\end{align}
The open ball $\ball(\mu, \delta)$ contains the probability measures with Dudley distance less than $\delta$ to $\mu$.

Suppose $\mu$ is a probability measure on $\R$, we denote $H_\mu(z)$ as the Stieltjes transform of $\mu$ defined by 
(here $\text{conv}$ denotes the convex hull and $\C_+$ the upper half plane)
\begin{equation}
\begin{aligned}
H_\mu: \C_{+} \cup \R \setminus \text{conv}(\supp \mu) ~~ \rightarrow &~~ \C \\
z ~~ \mapsto &~~ \int_\R \frac{1}{z - \lambda} \mu(\d \lambda).
\end{aligned}
\end{equation}
$H_\mu$ is always injective, so we can define its inverse $G_{\mu}$: $G_\mu( H_\mu(z)) = z$. Denote $R_\mu$ as the R-transform defined by 
\begin{equation}
\begin{aligned}
R_\mu(w) : \text{Image}(H_\mu) ~~ \rightarrow &~~ \C,\\
w ~~ \mapsto & ~~G_\mu(w) - 1/w.
\end{aligned}
\end{equation} 

We denote $\sigma_{\rm sc}(\d \lambda) = \bfone_{|\lambda|\le 2}\sqrt{4-\lambda^2}/(2\pi) \d\lambda$ as the semi-circular law. The Stieltjes transform for the semi-circular law
is
\begin{equation}
H_{\sigma_{\rm sc}}(z) = \frac{z - \sqrt{z^2 - 4}}{2},
\end{equation}
and its R-transform is
\begin{equation}
R_{\sigma_{\rm sc}}(w) = w. 
\end{equation}

\subsection{Preliminary lemmas}
\label{sec:Preliminary}

We start by stating a form of the Kac-Rice formula that we will be  a key tool for our proof.
Essentially the same statement was used in  \cite{auffinger2013random}, and we refer to \cite{adler2009random} for general proofs and broader context.
\begin{lemma}
Let $f$ be a centered Gaussian field on $\S^{n-1}$ and let $\cA = (\cU_\alpha, \Psi_\alpha)_{\alpha \in \cI}$ be a finite atlas on $\S^{n-1}$. Set $f^{\alpha} = f \circ \Psi_\alpha^{-1} : \Psi_\alpha(U_\alpha) \subset \R^{n-1} \rightarrow \R$ and define $f_i^\alpha = \partial f^\alpha / \partial x_i$, $f_{ij}^\alpha = \partial^2 f^\alpha / \partial x_i \partial x_j$. Assume that for all $\alpha \in \cI$ and all $x, y \in \Psi_\alpha(\cU_\alpha)$, the joint distribution of $(f_i^\alpha(x), f_{ij}^\alpha (x))_{1 \leq i \leq j \leq n}$ is non-degenerate, and 
\begin{align}
\max_{i,j}\l \Var(f_{ij}^\alpha (x)) + \Var(f_{ij}^\alpha (y)) - 2 \Cov(f_{ij}^\alpha (x), f_{ij}^\alpha (y)) \l \leq K_\alpha \l \ln \l x - y \l \l^{-1 - \beta}  
\end{align}
for some $\beta > 0$ and $K_\alpha > 0$. For Borel sets $E \subset \R$ and $M \subset [-1,1]$, let 
\begin{align}
\Crt_{n,k}^f (M, E) = \sum_{\bsigma: \tgrad f(\bsigma) = 0} \ones \{ i(\thess f(\bsigma)) = k, f(\bsigma) \in E, \< \bsigma, \bu \> \in M\}.
\end{align}
Then, using $\d \bsigma$ to denote the usual surface measure on $\S^{n-1}$, and denoting by $\vphi_{\bsigma}(\bx)$ the density of $\grad f(\bsigma)$ at $\bx$,
we have
\begin{align}\label{eqn:ctk}
\E \{\Crt_{n,k}^f(M, E)\} = \int_{\< \bsigma, \bu \> \in M} \E[ \l \det(\thess f(\bsigma)) \l \cdot \ones \{ i(\thess f(\bsigma)) = k, f(\bsigma) \in E\} \l \tgrad f (\bsigma) = \bzero] \cdot \vphi_\bsigma( \bzero ) \cdot \d \bsigma
\end{align}
\end{lemma}

The next lemma specialize the last formula to our specific choice of $f(\,\cdot\,)$, cf. Eq.~(\ref{eqn:objective}). Its proof can be found in Section \ref{sec:ProofJointDensity}.
\begin{lemma}\label{lem:jointdensity}
We have
\begin{align}
\E \{\Crt_{n,\star}(M,E)\} =& \int_{M} V_n(m) \cdot \E\{\, \l \det(\bH) \l \cdot \ones \{ f \in E\} \} \cdot \vphi_\bsigma(\bzero) \cdot (1-m^2)^{-1/2} \cdot \d m, \label{eqn:ctk1}\\
\E \{\Crt_{n,0}(M,E)\}=& \int_{M} V_n(m) \cdot \E\{\, \l \det(\bH) \l \cdot \ones\{ \bH \preceq 0 \} \cdot \ones \{ f \in E\} \} \cdot \vphi_\bsigma(\bzero) \cdot (1-m^2)^{-1/2} \cdot \d m,\label{eqn:ctk2}
\end{align}
where 
\begin{align}
V_n(m) = \Vol(\partial \ball^{n-1} ((1-m^2)^{1/2}))
\end{align}
is the area of the $(n-1)$-th dimensional sphere with radius $(1-m^2)^{1/2}$, and $\vphi_\bsigma(\bzero)$ is the density of $\bg$ at $\bzero$. 
Further the joint distribution of $f \in \R$, $\bg \in \R^{n-1}$, and $\bH \in \R^{(n-1) \times (n-1)}$ is given by
\begin{align}
f \eqnd & \lambda m^k + \frac{1}{\sqrt{2n}} Z,\\
\bg \eqnd & k \lambda m^{k-1} \sqrt{1-m^2} \cdot \be_1 + \sqrt \frac{k}{2n} \cdot \tbg_{n-1},\\
\bH \eqnd& k(k-1) \lambda m^{k-2} (1-m^2) \be_1 \be_1^\sT +  \sqrt \frac{k(k-1)(n-1)}{2n} \bW_{n-1} -  k\Big(\lambda m^k + \frac{1}{\sqrt{2n}} Z\Big)\id_{n-1},
\end{align}
where $Z \sim \cN(0,1)$, $\tbg_{n-1} \sim \cN(\bzero, \id_{n-1})$ and $\bW_{n-1} \sim \GOE(n-1)$ are independent. 
\end{lemma}

The next lemma provides a simplified expression. Its proof is deferred to Section
\ref{sec:ProofExact}.
\begin{lemma}\label{lem:exact}
We have
\begin{equation}\label{eqn:cp_st}
\begin{aligned}
&\E\{\Crt_{n, \star}(M,E)\} =  \PC_n \cdot \int_{E} \d x \int_{M} (1-m^2)^{-3/2} \d m \cdot \E \{ \l \det (\bH_n) \l \} \\
& \quad \quad \times \exp\Big\{n\Big[ \frac{1}{2} (\log(k-1) + 1) + \frac{1}{2} \log(1-m^2) - k \lambda^2 m^{2k-2} (1-m^2)  - (x - \lambda m^k)^2 \Big]\Big\}  
\end{aligned}
\end{equation}
and 
\begin{equation}\label{eqn:lm_st}
\begin{aligned}
&\E\{\Crt_{n, \knot}(M,E)\} =  \PC_n \cdot \int_{E} \d x \int_{M} (1-m^2)^{-3/2} \d m \cdot \E \{ \l \det (\bH_n) \l \cdot \ones\{\bH_n \preceq 0 \} \} \\
& \quad \quad \times \exp\Big\{n\Big[ \frac{1}{2} (\log(k-1) + 1) + \frac{1}{2} \log(1-m^2) - k \lambda^2 m^{2k-2} (1-m^2) - (x - \lambda m^k)^2 \Big] \Big\}  
\end{aligned}
\end{equation}
where, for $\bW_{n-1}\sim \GOE(n-1)$, 
\begin{align}
\bH_n = & \theta_n(m)  \cdot \be_1 \be_1^\sT +   \bW_{n-1} -  t_n(x) \cdot \id_{n-1},\\
t_n(x) =& \Big( \frac{2kn}{(k-1)(n-1)}\Big)^{1/2} \cdot x,\\
\theta_n(m) = & \Big( \frac{2k(k-1)n}{(n-1)} \Big)^{1/2} \cdot  \lambda m^{k-2} (1-m^2),\\
\PC_n =& 2\cdot \Big( \frac{n-1}{2 e} \Big)^{\frac{n-1}{2}} \cdot \Gamma\Big(\frac{n-1}{2}\Big)^{-1} \cdot \Big(\frac{n}{(k-1)e \pi}\Big)^{1/2}.
\end{align}
Further, $\PC_n$ is exponentially trivial. 
\end{lemma}

The next lemma contains a well known fact that we will use several times in the proofs. It follows immediately from the joint distribution of eigenvalues in the 
$\GOE$ ensemble \cite{anderson2010introduction} follows, see for instance \cite{maida2007large}.
\begin{lemma}[Joint density of the eigenvalues of the spiked model]
Let $\bX_n = \theta \be_1 \be_1^\sT + \bW_{n}$, where $\bW_n \sim \GOE(n)$ and $\theta \geq 0$. The density joint for the eigenvalues of $\bX_n$ 
is given by
\begin{align}
\P_n^\theta(\d x_1, \ldots, \d x_n) = \frac{1}{Z_{n}^\theta} \cdot \prod_{i < j} \l x_i - x_j \l \cdot I_n(\theta, x_1^n) \cdot \exp\Big\{ - \frac{n}{4} \sum_{i=1}^n x_i^2 \Big\} \d x_1 \cdots \d x_n,
\end{align}
where $x_1^n$ denotes the vector $(x_1,\ldots, x_n)^\sT$, and $I_n$ is the spherical integral defined by 
\begin{align}
I_n(\theta, x_1^n) \eqndef \int_{\mathcal O_n} \exp\Big\{ \frac{n\theta}{2} \cdot (\bU \cdot \diag(x_1^n) \cdot \bU^\sT)_{11} \Big\} \d m_n(\bU),
\end{align}
with $m_n$ the Haar probability measure on $\mathcal O_n$ the orthogonal group of size $n$. 
\end{lemma}
Next, we state a lemma regarding the large deviations of the largest eigenvalue of the spiked model, proven\footnote{Notice that the formula  
in \cite{maida2007large} contains a typo, that is corrected here. Also, the normalization of $\bW_n$ is different from the
one in \cite{maida2007large}. Here the empirical spectral distribution converges to a semicircle supported on $[-2,2]$, while in 
\cite{maida2007large} the support is $[-\sqrt{2},\sqrt{2}]$.} in \cite{maida2007large}. 
\begin{lemma}[Large deviation of largest eigenvalue of the spiked model \cite{maida2007large}]\label{lem:LDP_spiked}
Let $\bX_n = \theta \be_1 \be_1^\sT + \bW_n$, where $\bW_n \sim \GOE(n)$, and denote by $\lambda_{\max}(\bX_n)$ the largest eigenvalue of $\bX_n$. Then we have
\begin{align}
\limsup_{n\rightarrow \infty} \frac{1}{n} \log \P( \lambda_{\max}(\bX_n) \leq t) \leq -  L(\theta,t),\\
\liminf_{n\rightarrow \infty} \frac{1}{n} \log \P( \lambda_{\max}(\bX_n) < t) \geq -  L(\theta,t_-),
\end{align}
where $L(\theta,t)$ is as defined in Eq. (\ref{eqn:LDP_spiked}).
\end{lemma}
For symmetric matrix $\bB_n\in\reals^{n\times n}$, denote by $L_{n-1}(\bB_n) = 1/(n-1) \cdot \sum_{i=2}^n \delta_{\lambda_i(B_n)}$
the empirical distribution of the $n-1$ smallest eigenvalues. 

We next state three useful lemmas on the spherical integral from the papers \cite{maida2007large,guionnet2005fourier}. 
\begin{lemma}[Continuity of spherical integral I, \cite{guionnet2005fourier}, Lemma 14]
\label{lem:continuity_si}
For any $\theta, \eta > 0$,  there exists a function $g_{\theta, \eta} (\delta): \R_{\ge 0} \rightarrow \R_{\ge 0}$ with $\lim_{z \rightarrow 0_+} g_{\theta, \eta} (z) = 0$, such that the following holds.
Let $\bx, \by\in\reals^n$ be two vectors, with $x_{\max} = \max_{i\le n} x_i$, $x_{\min} = \min_{i\le n} x_i$, $y_{\max} = \max_{i\le n} y_i$, $y_{\min} = \min_{i\le n} y_i$.
Let $\mu_x$, $\mu_y$ be their empirical distributions and define $H_{\bx}(z) =(1/n)\sum_{i=1}^n 1/(z - x_i)$. 
If $d(\mu_x, \mu_y) \le \delta$ and  $\theta \in H_{\bx}([x_{\min} - \eta, x_{\max} + \eta]^c) \cap H_{\by}( [y_{\min} - \eta, y_{\max} + \eta]^c )$, then 
for sufficiently large $n$
\begin{align}
\Big\l \frac{1}{n} \log I_n(\theta,\bx)  - \frac{1}{n}\log I_n(\theta, \by)  \Big\l \leq g_{\theta, \eta} (\delta).
\end{align}
\end{lemma}

\begin{lemma}[Continuity for spherical integral II, \cite{maida2007large}, Proposition 2.1]\label{lem:continuity_si2}
For any $\theta, \kappa, M > 0$, there exists a function $g_{\kappa,\theta,M}: \R_{\ge 0}\rightarrow \R_{\ge 0}$
with $\lim_{z\to 0} g_{\kappa,\theta,M}(z) = 0$, such that the following holds.
For $\bx,\by\in\reals^n$, denote by $\mu'_x, \mu'_y$ the empirical distributions of the $(n-1)$
smallest entries of $\bx, \by$, and $x_1, y_1$ the largest elements of $\bx, \by$. 
If $d(\mu'_x,\mu'_y) \leq n^{-\kappa}$, $\l x_1 - y_1 \l \leq \delta$, and $ \dl \bx\dl_\infty , \dl \by \dl_\infty \le  M$, 
then for sufficiently large $n$
\begin{align}
\Big\l \frac{1}{n} \log I_n(\theta, \bx) - \frac{1}{n} \log I_n(\theta, \by) \Big\l \leq g_{\kappa,\theta,M}(\delta)\, .
\end{align}
\end{lemma}

\begin{lemma}[Limiting distribution of spherical integral \cite{guionnet2005fourier}, Theorem 6]\label{lem:ldp_si}
Let $\theta > 0$, $\{ \bx(n) \}_{n \in \N_+}$ be a sequence of vectors with empirical measure $L_n$ converging weakly to a compactly supported measure $\mu$, 
and limiting largest element $x_{\max} \ge \sup\{x: \; x\in{\rm supp}(\mu)\}$ and limiting smallest element $x_{\min} \le \inf\{x: \; x\in{\rm supp}(\mu)\} < 0$. Then the function
\begin{align}
J(\mu, x_{\max}, \theta) = \lim_{n\rightarrow \infty} \frac{1}{n} \log I_n(\theta, \bx(n))
\end{align}
is finite and well defined (which does not depend on $x_{\min}$).  

Moreover, letting $x \ge \sup\{x: \; x \in \supp(\mu) \}$, we have
\begin{equation}
\begin{aligned}
J(\mu, x, \theta) =& \frac{\theta \cdot v(x, \theta)}{2} - \frac{1}{2} \int_{\R} \log(1 + \theta \cdot v(x, \theta) - \theta \cdot \lambda) \mu (\d \lambda),
\end{aligned}
\end{equation}
where
\begin{equation}
v(x, \theta) = \begin{cases}
R_\mu(\theta), ~~ \text{ if } H_\mu(x) \ge \theta,\\
x - 1/\theta, ~~ \text{ otherwise }. 
\end{cases}
\end{equation}
See Section \ref{sec:Defs} for the definitions of Stieltjes transform $H_\mu(x)$ and R-transform $R_\mu(x)$.
\end{lemma}

Setting $\mu = \sigma_{\rm sc}$ in the above lemma, with some simple calculations, we get the following expression for $J(\sigma_{\rm sc}, x, \theta)$. 
\begin{lemma}
Since $\sup\{ x: \; x \in \supp(\sigma_{\rm sc}) \} = 2$, $J(\sigma_{\rm sc}, x, \theta)$ is defined as $x \ge 2$. We have 
\begin{equation}\label{eqn:JDef}
\begin{aligned}
J(\sigma_{\rm sc}, x, \theta) = \begin{cases}
\theta^2 / 4, &\text{ if } 0 < \theta \le 1, x \in [2, \rho(\theta) ],\\
1/2 \cdot [\theta x -1 - \log(\theta) - \Phi_\star(x) ], & \text{ if } \theta \ge 1, x \ge 2, \text{ or } 0 < \theta \le 1,  x > \rho(\theta).\\
\end{cases}
\end{aligned}
\end{equation}
See Eq. (\ref{eq:PhiDef}) for the definition of $\Phi_\star(x)$.
\end{lemma}

\subsection{Proof of Lemma \ref{lem:jointdensity}}
\label{sec:ProofJointDensity}

We rewrite the objective function as 
\begin{equation}
f(\bsigma) = \<\bY, \bsigma^{\otimes k} \>= \lambda \cdot \< \bu, \bsigma\>^k + h(\bsigma),
\end{equation}
where 
\begin{equation}
h(\bsigma) = \frac{1}{\sqrt{2n}} \< \bW, \bsigma^{\otimes k} \> =  \frac{1}{\sqrt{2n}} \sum_{i_1, \ldots, i_k = 1}^n  G_{i_1 \cdots i_k} \sigma_{i_1} \cdots \sigma_{i_k}. 
\end{equation}

The Euclidean gradients and Hessians of the $f$ gives
\begin{align}
\nabla f(\bsigma) =& k\lambda  \, \< \bu, \bsigma \>^{k-1} \cdot  \bu + \nabla h(\bsigma),\\
\nabla^2 f(\bsigma) = & k(k-1)\lambda \cdot  \< \bu, \bsigma \>^{k-2} \cdot \bu \bu^\sT + \nabla^2 h(\bsigma),
\end{align}
where
\begin{equation}
\begin{aligned}
\nabla h(\bsigma)_i =& \frac{k}{\sqrt{2n}} \cdot \sum_{i_1,\ldots, i_{k-1}=1}^nW_{i i_1 \cdots i_{k-1}} \sigma_{i_1}\cdots \sigma_{i_{k-1}} \\
=& \frac{k}{\sqrt{2n}}\cdot \frac{1}{k!} \cdot \sum_{\pi\in \cP_n} \sum_{i_1,\ldots, i_{k-1} = 1}^n (G^{\pi})_{i i_1 \cdots i_{k-1} } \sigma_{i_1} \cdots \sigma_{i_{k-1}}, 
\end{aligned}
\end{equation}
and 
\begin{equation}
\begin{aligned}
\nabla^2 h(\bsigma)_{ij} = & \frac{k(k-1)}{\sqrt{2n}} \cdot \sum_{i_1,\ldots, i_{k-2} = 1}^n W_{i j i_1 \cdots i_{k-2}} \sigma_{i_1} \cdots \sigma_{i_{k-2}} \\
=& \frac{k(k-1)}{\sqrt{2n}} \cdot \frac{1}{k!} \cdot \sum_{\pi \in \cP_n }\sum_{i_1,\ldots, i_{k-2} = 1}^n (G^{\pi})_{i j i_1\cdots i_{k-2}} \sigma_{i_1} \cdots \sigma_{i_{k-2}}. \\
\end{aligned}
\end{equation}

We will denote by $T_\sigma \S^{n-1}$ the tangent space of the unit sphere $\S^{n-1}$ at the point $\bsigma$, which
we will identify isometrically with the Euclidean subspace of $\R^n$ orthogonal to $\bsigma$.
The Riemannian gradients and Hessians of $f$ on the manifold $\S^{n-1}$, restricted on the tangent space 
are given by
\begin{align}
\tgrad f(\bsigma) = & \Proj_\bsigma^\bperp \nabla f(\bsigma) =  k\lambda \< \bu, \bsigma \>^{k-1} \Proj_\bsigma^\bperp \bu + \Proj_\bsigma^\bperp \nabla h(\bsigma),\\
\thess f(\bsigma) = & \Proj_\bsigma^\bperp \nabla^2 f(\bsigma) \Proj_\bsigma^\bperp  - \< \bsigma, \nabla f(\bsigma) \> \cdot \Proj_\bsigma^\bperp \\
=& k(k-1)\lambda \< \bu, \bsigma \>^{k-2}  \cdot (\Proj_\bsigma^\bperp \bu) (\Proj_\bsigma^\bperp \bu)^\sT - k \lambda \< \bu, \bsigma\>^k \cdot \Proj_\bsigma^\bperp \\
&+ \Proj_\bsigma^\bperp \nabla^2 h(\bsigma) \Proj_\bsigma^\bperp  - \< \bsigma, \nabla h(\bsigma) \> \cdot \Proj_\bsigma^\bperp.
\end{align}

Taking $\bsigma = \be_n$, and $\bu = m \be_n + \sqrt{1-m^2} \be_1$, we have (and identifying $T_{\bsigma}\S^{n-1}$ with $\R^{n-1}$)
\begin{align}
f(\bsigma) \eqnd & \lambda m^k +  \frac{1}{\sqrt{2n}} Z, \quad Z \sim \cN(0,1),\\
\Proj_\bsigma^\bperp \nabla f(\bsigma) \big\l_{T_\sigma \S^{n-1}} \eqnd & k \lambda m^{k-1} \sqrt{1-m^2} \be_1 + \sqrt \frac{k}{2n} \bg_{n-1},
\quad \bg_{n-1} \sim \cN(\bzero, \id_{n-1}),\\
\Proj_\bsigma^\bperp \nabla^2 f(\bsigma) \Proj_\bsigma^\bperp  \big\l_{T_\sigma \S^{n-1}} \eqnd & k(k-1) \lambda m^{k-2} (1-m^2) \be_1 \be_1^\sT +  
\sqrt \frac{k(k-1)(n-1)}{2n} \bW_{n-1}, \quad \bW_{n-1} \sim \text{GOE}(n-1).
\end{align}
Thus, the Riemannian Hessian  restricted to the tangent space is distributed as
\begin{align}
\thess f(\bsigma) \big\l_{T_\sigma \S^{n-1}} \eqnd & k(k-1) \lambda m^{k-2} (1-m^2) \be_1 \be_1^\sT +  \sqrt \frac{k(k-1)(n-1)}{2n} \bW_{n-1} -  k(\lambda m^k + \frac{1}{\sqrt{2n}} Z)\id_{n-1}.
\end{align}
Further note that $\tgrad f(\bsigma)$ and $\thess f(\bsigma)$ are independent. 

Plug in these quantities into Eq. (\ref{eqn:ctk}) and using rotational invariance gives Eq. (\ref{eqn:ctk2}). Summing Eq. (\ref{eqn:ctk}) over $k$ gives Eq. (\ref{eqn:ctk1}).

\subsection{Proof of Lemma \ref{lem:exact}}
\label{sec:ProofExact}

In Eq. (\ref{eqn:ctk2}), the determinant of Hessian is given by
\begin{align}
\l \det(\bH_n) \l = (k(k-1)(n-1)/2n)^{(n-1)/2} \det(\theta_n(m) \be_1 \be_1^\sT + \bW_{n-1} - t_n(f) \id_{n-1}). 
\end{align}
We denote the density of $f$ to be $p_f(x)$, we have  
\begin{align}
p_f(x) = \sqrt{n / \pi} \cdot \exp\{-n (x - \lambda m^{k})^2 \}.
\end{align}
The inner expectation yields
\begin{equation}
\begin{aligned}
 &\E\{\, \l \det(\bH_n) \l \cdot \ones\{ \bH_n \preceq \bzero \} \cdot \ones \{ f \in E\} \}\\
=& (k(k-1)(n-1)/(2n))^{(n-1)/2} \int_{E}  \E\{ \l\det(\bH_n)\l \cdot \ones\{H_n \preceq 0\} \}\, p_f(x) \d x \\
=& (k(k-1)(n-1)/(2n))^{(n-1)/2}\cdot  (n / \pi)^{1/2}\int_{E}  \E\{ \l \det(\bH_n)\l \cdot \ones\{ \bH_n \preceq 0 \}\} \exp\{-n (x - \lambda m^{k})^2 \}\,  \d x\, .
\end{aligned}
\end{equation}

We also have 
\begin{align}
V_n(m) =& 2\pi^{(n-1)/2}/\Gamma((n-1)/2) \cdot (1-m^2)^{(n-2)/2}\, , \\
\vphi_\bsigma(\bzero) =& (n/(\pi k))^{(n-1)/2} \cdot \exp\{ -n k \lambda^2 m^{2k-2} (1-m^2)\}\, . 
\end{align}

Plug these into Eq. (\ref{eqn:ctk2}), we have the form of Eq. (\ref{eqn:lm_st}), with pre-factor
\begin{equation}
\begin{aligned}
\PC_n =& (k(k-1)(n-1)/(2n))^{(n-1)/2} (n / \pi)^{1/2} \times 2\pi^{(n-1)/2}/\Gamma((n-1)/2) \\
& \times (n/(\pi k))^{(n-1)/2} \times (1/(k-1)e)^{n/2}\\
=&2 ((n-1)/(2e))^{(n-1)/2}  / \Gamma((n-1)/2) \times (n/(k-1)e \pi)^{1/2}\, .
\end{aligned}
\end{equation}

Expand the $\Gamma$ function in $\PC_n$ using Stirling's formula, it is easy to see that $\PC_n$ is exponentially trivial. 

Eq. (\ref{eqn:cp_st}) follows essentially by the same calculation.

\subsection{Proof of Theorem \ref{thm:critical_point}}
\label{sec:ProofCritical}

Throughout the proof, we will use the notations
\begin{equation}
\begin{aligned}
\He_n = &~ \theta \cdot \be_1 \be_1^\sT + \bW_{n} - t \cdot \id_{n},\\
\bX_n = & ~ \theta \cdot \be_1 \be_1^\sT + \bW_{n} ,\\
\bH_n = &~\theta_n(m) \cdot \be_1 \be_1^\sT + \bW_{n-1} - t_n(x) \cdot \id_{n-1},\\
\theta(m) =&~ \sqrt{2k(k-1)} \cdot \lambda m^{k-2} (1-m^2),\\
t(x) =&~ \sqrt{2k/(k-1)} \cdot x,\\
\theta_n(m) =&~ \sqrt{2k(k-1)n/(n-1)} \cdot \lambda m^{k-2} (1-m^2),\\
t_n(x) =&~ \sqrt{2kn/((k-1)(n-1))} \cdot x.
\end{aligned}
\end{equation}

In order to prove Theorem \ref{thm:critical_point}, we will establish the following key Proposition,
whose proof is deferred to Sections \ref{sec:PropoCritA}, \ref{sec:PropoCritB}, \ref{sec:PropoCritC}.
\begin{proposition}\label{prop:critical_point}
The following statements hold
\begin{enumerate}[label = (\alph*)]
\item \emph{Exponential tightness:} 
\begin{align}\label{eqn:exponential_tightness}
\lim_{z \rightarrow \infty}\limsup_{n\rightarrow \infty}\frac{1}{n}\log\E\{\Crt_{n,\star}([-1,1], ( -\infty, -z] \cup [z, \infty))\}  = -\infty.
\end{align}
\item  \emph{Upper bound.} For any fixed large $U_0>0$ and $T_0>0$, let  $\cUb_0 \subset [- U_0, U_0]$ and $\cTb_0 \subset [ -T_0, T_0]$ be
  two compact sets, and define $\cEb_0 \eqndef \cUb_0  \times  \cTb_0$. Then we have (for $\Phi_{\star}$ defined as per Eq.~(\ref{eq:PhiDef}))
\begin{align}\label{eqn:critical_point_ub}
\limsup_{n\rightarrow \infty} \sup_{(\theta, t) \in \cEb_0} \frac{1}{n}\log \E\{ \l \det(\He_n) \l\} \leq \sup_{ t \in \cTb_0}  \Phi_\star(t).
\end{align}
\item \emph{Lower bound.} For any fixed $\delta > 0$, $\theta_0$ and $t_0$, define $\cUo_0^\delta = (\theta_0 - \delta, \theta_0 + \delta)$, $\cTo_0^\delta = (t_0 - \delta, t_0 + \delta)$, and  $\cEo_0^\delta \eqndef \cUo_0^\delta \times \cTo_0^\delta$. Then we have
\begin{align}
\liminf_{n\rightarrow \infty}  \frac{1}{n}\log \int_{(\theta, t) \in \cEo_0^\delta}  \E\{\l \det(\He_n) \l \} \d \theta \d t \geq  \Phi_\star(t_0).
\end{align}
\end{enumerate}
\end{proposition}

Using this proposition, we can prove Theorem \ref{thm:critical_point}.

\begin{proof}[Proof of Theorem \ref{thm:critical_point}]
Because of the exponential tightness property, we only need to consider the case when the set $E$ is bounded.  
We will prove first the upper bound of Eq.~(\ref{eq:ThmCrUB}), and then the lower bound, cf. Eq.~(\ref{eq:ThmCrLB}).

\noindent {\bf Step 1. Upper bound.} 

First, letting $E_0 = (x_0 - \delta_0, x_0 + \delta_0)$, we claim that
\begin{equation}\label{eqn:Thm1UpperFirst}
\begin{aligned}
\lim_{\delta_0 \rightarrow 0_+}\limsup_{n\rightarrow \infty}\frac{1}{n}\log \E\{\Crt_{n, \star}(M,E_0) \} \leq \sup_{m \in \Mb} S_\star(m, x_0). 
\end{aligned}
\end{equation}

Assuming this claim holds, to prove Eq. (\ref{eq:ThmCrUB}), we consider a general compactly supported set $E$. Fix an $\eps > 0$, for each $x \in E$, there exists a radius $\delta_x$ such that 
\begin{equation}
\limsup_{n\rightarrow \infty}\frac{1}{n}\log \E\{\Crt_{n, \star}(M,(x - \delta_x, x + \delta_x)) \} \le \sup_{m \in \Mb} S_\star(m, x) + \eps.
\end{equation}
Then $\{ (x - \delta_x, x+\delta_x): x \in E \}$ is an open cover of $\Eb$. Due to compactness of $\Eb$, there exists finite number of intervals $\{ (x_i - \delta_{x_i}, x_i + \delta_{x_i}) \}_{i=1}^m$ that form a cover of $\Eb$, and such that the above equation holds. Therefore
\begin{equation}
\begin{aligned}
&\limsup_{n\rightarrow \infty}\frac{1}{n}\log \E\{\Crt_{n, \star}(M,E) \} \\
\le & \max_{i \in [m]} \limsup_{n\rightarrow \infty}\frac{1}{n}\log \E\{\Crt_{n, \star}(M, (x_i - \delta_{x_i}, x_i + \delta_{x_i})) \} \le \sup_{m \in \Mb, x \in \Eb} S_\star(m, x) + \eps.
\end{aligned}
\end{equation}
Eq. (\ref{eq:ThmCrUB}) holds by choosing arbitrarily small $\eps$.

Therefore, we just need to prove Eq. (\ref{eqn:Thm1UpperFirst}). For $x\in\R$, $S\subseteq \R$, define $d(x,S) = \inf\{|x-y|:\, y\in S\}$. 
For a given small $\delta > 0$, define 
\begin{equation}
\begin{aligned}
\Mb_\delta \eqndef& \{ m: d(m, \Mb) \leq \delta \},\\
\Eb_\delta \eqndef& \{ x: d(x, \Eb_0) \leq \delta \},\\
\cUb_\delta \eqndef& \{ \theta: \theta = \sqrt{2k(k-1)}  \cdot \lambda m^{k-2} (1-m^2), \, m \in \Mb_\delta \},\\
\cTb_\delta \eqndef& \{ t: t = \sqrt{2k/(k-1)}  \cdot x, \, x\in \Eb_\delta \},\\
\cEb_\delta \eqndef& \cUb_\delta \times\cTb_\delta.
\end{aligned}
\end{equation}
Since $E_0$ is bounded, we can define finite constants $U_0$, $T_0$ such that  $\cUb_\delta \subset [-U_0, U_0]$ and $\cTb_\delta \subset [-T_0, T_0]$. 

For any $\delta > 0$, there exists $N_{\delta}$ large enough, such that $t_n(x) \in \cTb_\delta$ and $\theta_n(m) \in \cUb_\delta$ for all $n \geq N_{\delta}$ and $(m,x) \in \Mb \times \Eb_0$. According to Proposition \ref{prop:critical_point}.$(b)$, there exists $N_{\eps, \delta} \ge N_\delta$, such that for all $n \geq N_{\eps, \delta}$, 
\begin{equation}
\begin{aligned}
&\sup_{m \in \Mb, x \in \Eb_0}\E\{ \l \det( \bH_n) \l\} = \sup_{m \in \Mb, x \in \Eb_0}\E\{ \l \det(\theta_n(m) \cdot \be_1 \be_1^\sT + \bW_{n-1} - t_n(x) \cdot \id_{n-1}) \l\} \\
\leq& \sup_{(\theta, t) \in \cEb_\delta}\E\{ \l \det(\theta \cdot \be_1 \be_1^\sT + \bW_{n-1} - t \cdot \id_{n-1}) \l\} 
\leq \exp\Big\{ (n-1) \Big[\sup_{t \in \cTb_\delta}\Phi_\star(t) + \eps\Big]  \Big\} 
\end{aligned}
\end{equation}
According to the expression for the expected number of critical point in Lemma \ref{lem:exact}, Eq.~(\ref{eqn:cp_st}),
\begin{equation}
\begin{aligned}
&\E\{\Crt_{n, \star}(M,E_0) \} \\
\le & \sup_{m \in \Mb, x \in \Eb_0} \E \{ \l \det (\bH_n) \l \} \cdot \PC_n \cdot \int_{E_0} \d x \int_{M} (1-m^2)^{-3/2} \d m  \\
& \quad \quad \times \exp\Big\{n\Big[ \frac{1}{2} (\log(k-1) + 1) + \frac{1}{2} \log(1-m^2) - k \lambda^2 m^{2k-2} (1-m^2)  - (x - \lambda m^k)^2 \Big]\Big\}  \\
\leq& \sup_{m \in \Mb, x \in \Eb_0} 4 \PC_n R_0 \times \exp\Big\{n\Big[ \frac{1}{2} (\log(k-1) + 1) - k \lambda^2 m^{2k-2} (1-m^2)  - (x - \lambda m^k)^2 \Big]\Big\}  \\
& \times \exp\Big\{(n-3)\Big[\frac{1}{2} \log(1-m^2) \Big] + (n-1) \cdot \sup_{t \in \cTb_\delta} [\Phi_\star(t) + \eps]\Big\}.
\end{aligned}
\end{equation}
Note that the pre-factor $2 \PC_n R_0$ is exponentially trivial. We have
\begin{equation}
\begin{aligned}
&\limsup_{n\rightarrow \infty}\frac{1}{n}\log \E\{\Crt_{n, \star}(M,E_0) \} \\
\leq &\sup_{m \in \Mb, x \in \Eb_0}\Big\{ \frac{1}{2}(\log(k-1)+1) + \frac{1}{2}\log(1-m^2) - k \lambda^2 m^{2k-2} (1-m^2) - (x- \lambda m^k) ^2 \Big\} + \sup_{t \in \cTb_\delta} \Phi_\star(t) + \eps.
\end{aligned}
\end{equation}
Letting $\eps, \delta \rightarrow 0_+$, and using the continuity of $\Phi_\star(t)$ and compactness of $\cEb_0$, we have
\begin{equation}
\begin{aligned}
&\limsup_{n\rightarrow \infty}\frac{1}{n}\log \E\{\Crt_{n, \star}(M,E_0) \} \\ 
\leq& \sup_{m \in \Mb, x \in \Eb_0}\Big\{ \frac{1}{2}(\log(k-1)+1) + \frac{1}{2}\log(1-m^2) - k \lambda^2 m^{2k-2} (1-m^2) - (x- \lambda m^k) ^2 \Big\} + \sup_{t \in \cTb_0} \Phi_\star(t).
\end{aligned}
\end{equation}
Note that $E_0 = (x_0 - \delta_0, x_0 + \delta_0)$, letting $\delta_0 \rightarrow 0$ and using the continuity of $\Phi_\star(t)$, we proved Eq. (\ref{eqn:Thm1UpperFirst}).

\noindent
{\bf Step 2. Lower bound. }
For any Borel sets $M \subset [-1,1]$ and $E \subset \R$, and for any $\eps > 0$, there exists $(m_0, x_0) \in M^o \times E^o$ such that 
\begin{align}
S_\star(m_0, x_0) \ge \sup_{(m,x) \in M^o \times E^o} S_\star(m,x) - \eps.
\end{align}

Denote $\theta_0 = \theta(m_0)$ and $t_0 = t(x_0)$. For a given small $\delta > 0$, define 
\begin{equation}
\begin{aligned}
M_0^\delta \eqndef& (m_0 - \delta, m_0 + \delta),\\
E_0^\delta \eqndef& (x_0 - \delta, x_0 + \delta),\\
\cB_0^\delta \eqndef& M_0^\delta \times E_0^\delta,\\
\cUo_{n}^\delta \eqndef&\{ \theta: \theta = \sqrt{2k(k-1)n/(n-1)}  \cdot \lambda m^{k-2} (1-m^2), \, m \in M_0^\delta  \},\\
\cTo_{n}^\delta \eqndef & \{ t: t = \sqrt{2kn/((k-1)(n-1))} \cdot x, \, x \in E_0^\delta \},\\
\cEo_n^\delta \eqndef& \cUo_n^\delta \times \cTo_n^\delta.
\end{aligned}
\end{equation}
We fix $\delta$ sufficiently small, so that $M_0^\delta \subset M^o$ and $E_0^\delta \subset E^o$. It is easy to see that $\cUo_{n}^\delta$ and $\cTo_{n}^\delta$ are open sets and $\theta_0 \in \cUo_{n}^\delta$, $t_0 \in \cTo_{n}^\delta$ are inner points. 

For this choice of $\delta$ and $\eps$, according to Proposition \ref{prop:critical_point}.$(c)$, for any $\eps_0 > 0$, we can find $N_{\eps, \eps_0, \delta}$ and $\delta_0 > 0$ such that as $n \geq N_{\eps, \eps_0, \delta}$, 
\begin{equation}
\begin{aligned}
\cEo_{0}^{\delta_0} \eqndef&~ (\theta_0 - \delta_0, \theta_0 + \delta_0) \times (t_0 - \delta_0, t_0 + \delta_0) \subset  \cEo_n^\delta,\\
\end{aligned}
\end{equation}
and 
\begin{equation}
\begin{aligned}
 \int_{(\theta, t) \in \cEo_{0}^{\delta_0}} \E\{\l \det(\theta \cdot \be_1 \be_1^\sT + \bW_{n-1} - t \cdot \id_{n-1}) \l \} \d \theta \d t \geq \exp\{ (n-1) [\Phi_\star(t_0)- \eps_0]\}.
\end{aligned}
\end{equation}

According to the expression for the expected number of critical point as in Eq. (\ref{eqn:lm_st}) in Lemma \ref{lem:exact}, 
\begin{equation}
\begin{aligned}
&\E\{\Crt_{n, \star}(M,E) \}  \geq \E\{\Crt_{n, \star}(M_0^\delta, E_0^\delta) \}\\
\ge & \PC_n \cdot \int_{\cB_0^\delta} \E \{ \l \det (H_n) \l \} \d x \d m \times \inf_{(m, x) \in \cB_0^\delta} \exp \Big\{ (n-3) \cdot \Big[ \frac{1}{2} \log(1-m^2)\Big] \\
& \quad \quad +n\Big[ \frac{1}{2} (\log(k-1) + 1) - k \lambda^2 m^{2k-2} (1-m^2)  - (x - \lambda m^k)^2 \Big]\Big\}  \\
\ge & \PC_n \cdot \int_{\cEo_{0}^{\delta_0}} \E \{ \l \det(\theta \cdot \be_1 \be_1^\sT + W_{n-1} - t \cdot \id_{n-1}) \l \} \frac{n-1}{2k\lambda n [(k-2) \cdot m(\theta)^{k-3} - k\cdot m(\theta)^{k-1}]}\d \theta \d t\\
& \times \inf_{(m, x) \in \cB_0^\delta} \exp \Big\{ n\Big[ \frac{1}{2} (\log(k-1) + 1) +\frac{1}{2} \log(1-m^2) - k \lambda^2 m^{2k-2} (1-m^2)  - (x - \lambda m^k)^2 \Big]\Big\}  \\
\ge & \frac{\PC_n}{8k^2 \lambda} \cdot  \exp\Big\{ (n-1) \cdot [\Phi_\star(t_0) - \eps_0] \Big\}\\
& \times \inf_{(m, x) \in \cB_0^\delta} \exp \Big\{ n\Big[ \frac{1}{2} (\log(k-1) + 1) +\frac{1}{2} \log(1-m^2) - k \lambda^2 m^{2k-2} (1-m^2)  - (x - \lambda m^k)^2 \Big]\Big\}.
\end{aligned}
\end{equation}
Note that the pre-constant $\PC_n /8k^2 \lambda$ is exponentially trivial on compact set. We have
\begin{equation}
\begin{aligned}
&\liminf_{n\rightarrow \infty}\frac{1}{n}\log \E\{\Crt_{n, \star}(M,E) \} \\
\geq &\inf_{(m,x) \in  \cB_0^\delta } \Big\{\frac{1}{2}(\log(k-1)+1) + \frac{1}{2}\log(1-m^2) - k \lambda^2 m^{2k-2} (1-m^2) - (x- \lambda m^k) ^2 \Big\} +  \Phi_\star(t_0) - \eps_0.
\end{aligned}
\end{equation}
Letting $\eps_0, \delta \rightarrow 0_+$, we have
\begin{align}
\liminf_{n\rightarrow \infty}\frac{1}{n}\log \E\{\Crt_{n, \star}(M,E) \} \geq  S_\star(m_0,x_0) \ge \sup_{m \in M^o, x \in E^o} S_\star(m,x) - \eps.
\end{align}
Letting $\eps \rightarrow 0_+$ gives the desired result. 

\end{proof}

In the following we prove Proposition \ref{prop:critical_point}.

\subsubsection{Proof of Proposition \ref{prop:critical_point}.$(a)$: Exponential tightness}
\label{sec:PropoCritA}

We need to upper bound $\E\{\Crt_{n,\star}([-1,1], ( -\infty, -z] \cup [z, \infty))\}$. Starting from Eq. (\ref{eqn:cp_st}), we have a crude upper bound
\begin{equation}
\begin{aligned}
&\E\{\Crt_{n, \star}([-1,1],(-\infty, z] \cup [z, \infty))\} \\
\leq &  4 \PC_n \cdot \int_{z}^\infty \d x \cdot \E \{ [4  x  + 2k \lambda + \dl W_{n-1} \dl_\op ]^n \} \times \exp\Big\{n\Big[ \frac{1}{2} (\log(k-1) + 1) - ( x - \lambda)^2 \Big]\Big\}.
\end{aligned}
\end{equation}
We  let $D_n = 4 \PC_n \cdot \exp\{ n [1/2 \cdot (\log(k-1) + 1)] \}$. It is easy to check that $D_n$ is exponentially finite. 

Taking $z \geq \max(2 k \lambda,1)$ (note that we consider $k \geq 2$) and let $Y_n = \dl W_{n-1} \dl_{\op}$, we have 
\begin{align}\label{eqn:et_bd1}
\E\{\Crt_{n, \star}([-1,1],(-\infty, z] \cup [z, \infty))\} &\leq   D_n \cdot \int_{z}^\infty \E \{ (5  x  + Y_n )^n \} \cdot \exp\{-nx^2/4 \}  \d x\\
& \le D_n \E \{ (1  + Y_n )^n \} \int_{z}^\infty (5  x )^n  \cdot \exp\{-nx^2/4 \}  \d x\, .
\end{align}
The  operator norm of a GOE matrix has sub-Gaussian tails, cf. Lemma \ref{lem:concentration_goe}. This immediately implies
\begin{align}
\E \{ (1  + Y_n )^n \}  \le \E\{e^{nY_n}\} \le  C^n\, ,
\end{align}
for some universal constant $C$, whence
\begin{align}
\E\{\Crt_{n, \star}([-1,1],(-\infty, z] \cup [z, \infty))\} &\le D_n C^n \int_{z}^\infty (5  x )^n  \cdot \exp\{-nx^2/4 \}  \d x\, ,
\end{align}
and the claim in Eq.~(\ref{eqn:exponential_tightness}) follows by Lemma \ref{lem:exp_tight_bound}.

\subsubsection{Proof of Proposition \ref{prop:critical_point}.$(b)$: Upper bound}
\label{sec:PropoCritB}

Recall that $\He_n = \theta  \be_1 \be_1^\sT + \bW_{n} - t \id_{n}$ and $\bX_n = \theta \be_1 \be_1^\sT + \bW_n$. 
Let $\sigma_{\rm sc}(\d \lambda) = \bfone_{|\lambda|\le 2}\sqrt{4-\lambda^2}/(2\pi) \d\lambda$ be the semicircle law, and 
denote by $\ball(\sigma_{\rm sc}, \delta)$ the ball of radius $\delta$ around $\sigma_{\rm sc}(\d \lambda)$, with the Dudley metric defined in Section \ref{sec:Defs}. 
Let $\ball_{R}(\sigma_{\rm sc},\delta)$ be the set of probability measures in $\ball(\sigma_{\rm sc}, \delta)$ with support in $[-R, R]$. For $\mu$ a probability measure on $\R$, define (for all $x$ such that the integral on the right-hand
side is well defined)
\begin{align}\label{eqn:boundonphi}
\Phi(\mu, x) = \int_{\R} \log \l \lambda - x \l \cdot \mu(\d \lambda).
\end{align}

We will often make use of the following fact: for any event $A$, we have 
(denoting by $L_{n} = 1/n \cdot \sum_{i=1}^{n} \delta_{x_i}$ the empirical measure of the numbers $\{ x_i\}_{i=1}^n$):
\begin{equation}\label{eqn:central_equality}
\begin{aligned}
\E\{\l \det(\He_n) \l; A \} =& \int_{\R^n} \prod_{i=1}^n \l x_i - t\l \cdot \ones_A \cdot \P_n^\theta(\d x_1, \ldots, \d x_n)\\
=& \frac{1}{Z_n^\theta}\int_{\R^n} \prod_{i=1}^n \l x_i - t\l \cdot I_n(\theta, x_1^n) \cdot \ones_A \cdot \prod_{i < j} \l x_i - x_j \l \cdot \prod_{i=1}^n  \exp\{-n x_i^2/4\} \d x_i\\
=& \frac{Z_n^0}{Z_n^\theta}\int_{\R^n} \exp\{ n \cdot \Phi(L_n, t)\} \cdot  I_n(\theta, x_1^n) \cdot \ones_A\cdot \P_n^0 (\d x_1,\ldots, \d x_n),\\
\end{aligned}
\end{equation}
where 
\begin{equation}
Z_n^\theta = Z_n^0 \int_{\R^n} I_n(\theta, x_1^n) \cdot  \P_n^0 (\d x_1,\ldots, \d x_n).
\end{equation}

We have upper bound
\begin{equation}
\begin{aligned}
\E\{\l \det(\He_n) \l\} \leq & \underbrace{\E\{\l \det(\He_n) \l; L_n \in \ball(\sigma_{\rm sc},\delta) \}}_{E_1} + \underbrace{\E\{\l \det(\He_n) \l; L_n \notin \ball(\sigma_{\rm sc},\delta) \}}_{E_2} \\
\end{aligned}
\end{equation}
where $\delta >0 $ is a fixed arbitrary small number.

According to Lemma \ref{lem:exp_vanishing}, $E_2 \le B_n^3 / A_n^2$ as a function of $(\theta, t)$ is exponentially vanishing on any compact set. 
Hence, we just need to consider the term $E_1$:
\begin{equation}
\begin{aligned}
E_1 = &\frac{ Z_{n}^0}{Z_{n}^\theta} \int_{\R^n} \exp\{n \cdot \Phi(L_{n}, t) \} \cdot I_n(\theta, x_1^n)   \cdot \ones\{ L_n \in \ball(\sigma_{\rm sc}, \delta) \} \cdot  \d \P_{n}^0\\
\leq & \exp\{ n \cdot \sup_{\mu \in \ball(\sigma_{\rm sc}, \delta)} \Phi(\mu, t)\}  \cdot  \frac{ Z_{n}^0}{Z_{n}^\theta}\int_{\R^n} I_n(\theta, x_1^n) \d \P_{n}^0\\
= & \exp\{ n \cdot \sup_{\mu \in \ball(\sigma_{\rm sc}, \delta)} \Phi(\mu, t)\}.
\end{aligned}
\end{equation}
Defining $\Phi_\eta(\mu,t) = \int_\R \log( \l t - \lambda \l \vee \eta ) \mu(\d\lambda)$, it is easy to verify that 
$\Phi_\eta(\mu,t)$ is continuous in $(\mu,t)\in M_1([-R_0,R_0]) \times \cTb_0$ for each $\eta$. Since $\Phi(\mu,t) = \inf_{\eta >0} \{ \Phi_\eta(\mu,t)\}$, it holds that 
$\Phi(\mu,t)$ is upper semicontinuous on the same domain. Further, a direct calculation yields $\Phi(\sigma_{\rm sc},t) = \Phi_\star(t)$. Therefore,
\begin{align}
\limsup_{\delta\rightarrow 0} \limsup_{n \rightarrow \infty}  \sup_{ (\theta, t) \in \cEb_0}  \frac{1}{n} \log E_1   \leq \limsup_{\delta\rightarrow 0} 
\sup_{ t \in \cTb_0,\mu \in \ball(\sigma_{\rm sc},\delta)} \Phi(\mu, t)   \leq \sup_{t \in \cTb_0} \Phi_\star(t).
\end{align}

Consequently, we have
\begin{align}
\limsup_{n\rightarrow \infty }\sup_{(\theta,t) \in \cEb_0} \frac{1}{n}\log \E\{\l \det(\He_n) \l \} \leq \sup_{t \in \cTb_0} \Phi_\star(t).
\end{align}

\subsubsection{Proof of Proposition \ref{prop:critical_point}.$(c)$: Lower bound}
\label{sec:PropoCritC}

Since $t\mapsto \Phi_\star (t)$ is continuous, we only need to prove  the lower bound for $(\theta_0, t_0)$ in a dense subset of $\R^2$. We consider two cases for $t_0$:
\begin{itemize}
\item[] \emph{Case 1:} $t_0 \in (-\infty, -2) \cup (2, \infty)$. In this case, the proof is easier, since $t_0$ is separated from the support of the semicircle law. 
We only consider the subcase $t_0 > 2$ and $\theta_0 > 1$, which is more difficult. The proof for $t_0 > 2$ and $\theta_0 < 1$ follows by a very similar argument. 
\item[] \emph{Case 2:} $t_0 \in (-2,2)$. This case is more challenging since $t_0$ is inside the support of semicircle law. We will distinguish two subcases:
subcase 2.1: $t_0 \in (-2,2)$ and 
$\theta_0 > 1$; and subcase 2.2 $t_0 \in (-2,2)$ and $\theta_0 < 1$. We use the estimate of the spherical integral in \cite{guionnet2005fourier} and \cite{maida2007large}. 
\end{itemize}

\noindent{\bf Case 1:} $t_0 \in (-\infty, -2) \cup (2, \infty)$. As mentioned, we consider  $t_0 > 2$ and $\theta_0 > 1$ here. The other cases are similar. 

Let $\rho(\theta) = \theta + 1/\theta $. Let  $\delta_0 \in (0,\delta)$ be such that $t_0 > 2+ 2 \delta_0$. We can then choose $\eps_0 \in (0,\delta)$ such that 
$\rho(\theta_0 + 2\eps_0) -\rho(\theta_0 - 2\eps_0) \le \delta_0$ and $\rho(\theta_0 - 2\eps_0) > 2$. 
Let $\Setb(\delta_0,\eps_0) = [t_0 - \delta_0, t_0 + \delta_0]\setminus [\rho(\theta_0 -2\eps_0), \rho(\theta_0 + 2\eps_0)]$, and 
$\Seta(\delta_0,\eps_0) = [\rho(\theta_0 - \eps_0), \rho(\theta_0 + \eps_0)] \cup [t_0  - 2\delta_0,t_0 + 2\delta_0]^c$. We have $d( \Seta(\delta_0,\eps_0),\Setb(\delta_0,\eps_0) ) > 0$, and the eigenvalues of the spiked matrix $\bX_n$ belongs to $\Seta(\delta_0,\eps_0)$ with probability converging to $1$ as $n\to\infty$. 

Thus, for $t\in \Setb(\delta_0,\eps_0)$, $\theta\in \cU_0^{\eps_0} = (\theta_0-\eps_0,\theta_0+\eps_0)$ we have the following lower bound, holding
for any $\delta'>0$ (here $L_n(\bX_n)$ denotes the empirical spectral distribution of the matrix $\bX_n$):
\begin{equation}\label{eqn:lb_cp}
\begin{aligned}
\E\{ \l \det(\He_n) \l \} =&\frac{ Z_{n}^0}{Z_{n}^\theta} \int_{\R^n} \exp\{n \cdot \Phi(L_{n}, t) \} \cdot I_n(\theta, x_1^n)   \cdot \d \P_{n}^0\\
\geq&\frac{ Z_{n}^0}{Z_{n}^\theta} \int_{\R^n} \exp\{n \cdot \Phi(L_{n}, t)  \} \cdot I_n(\theta, x_1^n)   \cdot \ones\{L_{n} \in \ball(\sigma_{\rm sc}, \delta'), \supp(L_n) \in \Seta(\delta_0,\eps_0) \}  \cdot \d \P_{n}^0\\
\geq&\left\{ \frac{ Z_{n}^0}{Z_{n}^\theta} \int_{\R^n}I_n(\theta, x_1^n)\ones\{L_{n} \in \ball(\sigma_{\rm sc}, \delta'), \supp(L_n) \in \Seta(\delta_0,\eps_0) \}  \d \P_{n}^0 \right\} \\
& \times \exp\Big\{n [ \inf_{\substack{\mu \in \ball(\sigma_{\rm sc}, \delta'),\\ \supp(\mu) \in \Seta(\delta_0,\eps_0)}}\Phi(\mu, t) ] \Big\},\\
\geq&\Big\{\P(\supp(L_n(\bX_n)) \subseteq \Seta(\delta_0,\eps_0))- \underbrace{ \frac{ Z_{n}^0}{Z_{n}^\theta} \int_{\R^n}I_n(\theta, x_1^n)\cdot \ones\{L_{n} \notin \ball(\sigma_{\rm sc}, \delta') \}  \d \P_{n}^0}_{G_1}  \Big\} \\
& \times \exp\Big\{n [ \inf_{\substack{\mu \in \ball(\sigma_{\rm sc}, \delta'),\\\supp(\mu) \in \Seta(\delta_0,\eps_0)}}\Phi(\mu, t) ] \Big\}.
\end{aligned}
\end{equation}
According to Lemma \ref{lem:exp_vanishing}, $G_1 = B_n^2/A_n^2$ is exponentially vanishing on compact sets, so we can drop this term. We also know that $\P(\supp(L_n(\bX_n)) \subseteq \Seta(\delta_0,\eps_0))$ is exponentially trivial on compact sets. 

This gives
\begin{equation}
\begin{aligned}
&\liminf_{n\rightarrow \infty} \frac{1}{n} \log  \int_{(\theta, t) \in \cEo_0^\delta} \E\{ \l \det(\He_n) \l\} \d \theta \d t\\
\ge & \liminf_{n\rightarrow \infty} \frac{1}{n} \log  \int_{\theta \in \cUo_0^{\eps_0}, t \in \Setb(\delta_0,\eps_0)} \E\{ \l \det(\He_n) \l\} \d \theta \d t\\
\geq& \liminf_{\delta' \rightarrow 0_+} \inf_{\substack{t \in \Setb(\delta_0,\eps_0), \mu \in \ball(\sigma_{\rm sc},\delta')\\ 
\supp(\mu) \in \Seta(\delta_0,\eps_0)}} \Phi(\mu,t) = \inf_{t \in \Setb(\delta_0,\eps_0)}\Phi_\star ( t).
\end{aligned}
\end{equation}
The last equality holds because $\Phi(\mu, t)$ is continuous with respect to $(\mu, t)$ on $\{(\mu ,t) : \mu \in \ball(\sigma_{\rm sc}, \delta'), \supp(\mu) \in \Seta(\delta_0,\eps_0), t \in \Setb(\delta_0,\eps_0) \}$.

Since $\Phi_\star(t)$ is continuous, letting first $\eps_0 \rightarrow 0_+$ and then $\delta_0 \rightarrow 0_+$, we have
\begin{equation}
\begin{aligned}
\liminf_{n\rightarrow \infty}  \frac{1}{n}\log \int_{(\theta, t) \in \cEo_0^\delta}  \E\{\l \det(\He_n) \l \} \d \theta \d t
\ge  \limsup_{\delta_0 \rightarrow 0_+}\limsup_{\eps_0 \rightarrow 0_+}  \inf_{t \in \Setb(\delta_0,\eps_0)}\Phi_\star ( t) = \Phi_\star ( t_0).
\end{aligned}
\end{equation}

{\bf Case 2.1: } We next consider the case as $t_0 \in (-2, 2)$ and $\theta_0 > 1$. We further assume $t_0 > 0$, since the case $t_0\le 0$ can be 
treated analogously. Define
\begin{align}
H_1 =& \int_{\R^n} \exp\{n \cdot \Phi(L_{n}, t) \} \cdot I_n(\theta, x_1^n)   \cdot \d \P_{n}^0(x_1^n),\\
H_2 =& \int_{\R^n} I_n(\theta, x_1^n)   \cdot \d \P_{n}^0(x_1^n).
\end{align}
We have $\E\{ \l \det(\He_n) \l \} = H_1 / H_2$. Let $\rho(\theta) = \theta + 1/\theta$. Since $\Phi_\star(t_0) = t_0^2/4 - 1/2$ for $t_0 \in (-2,2)$, it suffices to show that 
\begin{align}
\liminf_{n\rightarrow \infty} \frac{1}{n} \log \int_{(\theta,t) \in \cEo_0^\delta} H_1 \d \theta \d t ~\ge~ & t_0^2/4 + \Phi_\star(\rho(\theta_0)) - \rho(\theta_0)^2/4+ J(\sigma_{\rm sc}, \rho(\theta_0), \theta_0),\label{eqn:lb1}\\
\limsup_{\delta \rightarrow 0_+}\limsup_{n\rightarrow \infty} \sup_{\theta \in \cUo_0^\delta} \frac{1}{n} \log H_2 ~\le~ & 1/2 +  \Phi_\star(\rho(\theta_0)) - \rho(\theta_0)^2/4+ J(\sigma_{\rm sc}, \rho(\theta_0),\theta_0)\, , \label{eqn:lb2}
\end{align}
with $J(\cdots)$ defined as per Lemma \ref{lem:ldp_si}.

By \cite[Proposition 3.1]{maida2007large}, for fixed $\theta > 1$, we have
\begin{align}
\limsup_{n\rightarrow \infty}  \frac{1}{n} \log H_2 ~\le~ & 1/2 +  \Phi_\star(\rho(\theta)) - \rho(\theta)^2/4+ J(\sigma_{\rm sc}, \rho(\theta), \theta).
\end{align}
Therefore, Eq. (\ref{eqn:lb2}) is implied by the convexity of $1/n\cdot \log H_2$ as a function of $\theta$.

To prove Eq. (\ref{eqn:lb1}), first, we choose $\delta_0\in(0,\delta)$ and $\eps_0>0$ small enough such that $\rho(\theta_0 - \delta_0) - \eps_0 > t_0 + 2 \delta_0$. For any fixed $\theta \in (\theta_0 - \delta_0, \theta_0 + \delta_0)$, we have
\begin{equation}
\begin{aligned}
\int_{\cTo_0^\delta} H_1 \d t 
= & \frac{1}{Z_n^0}\int_{\cTo_0^\delta} \d t \int_{\R^n} I_n(\theta, x_1^n) \cdot \prod_{i=1}^n \l t - x_i \l  \cdot  \prod_{1 \le i < j\le n} \l x_i - x_j \l \cdot \exp\Big\{-\frac{n}{4} \sum_{i=1}^n x_i^2 \Big\}\prod_{i=1}^n \d x_i\\
=&\frac{1}{Z_n^0} \int_{x_0 \in \cTo_0^\delta} \int_{\R^n}  I_n(\theta, x_1^n) \cdot  \prod_{0 \le i < j \le n} \l x_i - x_j \l  \cdot \exp\Big\{- \frac{n}{4}\sum_{i=0}^n x_i^2 \Big\} \cdot \prod_{i=0}^n \d x_i \cdot \exp\Big\{ \frac{n}{4} x_0^2 \Big\}\\
\ge & \frac{1}{Z_n^0} \int_{x_n \in \ball(\rho(\theta) , \eps_0)} \int_{x_0 \in  \cTo_0^{\delta_0}} \int_{x_1^{n-1} \in \R^{n-1}} \prod_{i=0}^{n-1} \l x_n - x_i \l \exp\Big\{-\frac{n}{4} \,
 x_n^2\Big\} \d x_n\\
&\times I_n(\theta, x_1^n) \cdot \ones\{ L_n(x_0^{n-1}) \in \ball_{2 + \eps_0}(\sigma_{\rm sc}, n^{-1/4}) \} \prod_{0 \le i < j \le n-1} \l x_i - x_j \l  \cdot \exp\Big\{- \frac{n}{4}\sum_{i=0}^{n-1} x_i^2 \Big\} \cdot \prod_{i=0}^{n-1} \d x_i \\
&\times \exp\Big\{ \frac{n}{4}(t_0-\delta_0)^2\} \Big\}\, .\\
\end{aligned}
\end{equation}
Note that for $n$ sufficiently large, $L_n(x_0^{n-1}) \in \ball_{2 + \eps_0}(\sigma_{\rm sc}, n^{-1/4})$ implies that $L_{n-1}(x_1^{n-1}) \in \ball_{2 + \eps_0}(\sigma_{\rm sc}, 2 n^{-1/4})$. Therefore, for any $\theta \in (\theta_0 - \delta_0, \theta_0 + \delta_0)$, we have 
\begin{equation}\label{eqn:decomp1}
\begin{aligned}
\int_{\cTo_0^\delta} H_1 \d t  \ge & (\rho(\theta_0 - \delta_0) - t_0 - \delta_0 - \eps_0) \cdot \exp\{ - \rho(\theta_0 + \delta_0)^2/4 \} \times 2 \eps_0 ~~~~~~~~~~~~~~~~~ (\equiv A_1)\\
&\times \exp\Big\{ \frac{n}{4} (t_0-\delta_0)^2 \} \Big\}~~~~~~~~~~~~~~~~~~~~~~~~~~~~~~~~~~~~ (\equiv A_2)\\
&\times \inf_{\substack{L_{n-1}(x_1^{n-1}) \in \ball_{2 + \delta_0}(\sigma_{\rm sc}, 2n^{-1/4}),\\x_n \in \ball(\rho(\theta) , \eps_0 + 2 \delta_0)}} \exp\{ (n-1) [\Phi(L_{n-1}(x_1^{n-1}), x_n)-1/4 \cdot x_n^2] \} ~~(\equiv A_3)\\
& \times \inf_{\substack{L_{n-1}(x_1^{n-1}) \in \ball_{2 + \delta_0}(\sigma_{\rm sc}, 2n^{-1/4}),\\x_n \in \ball(\rho(\theta ) , \eps_0 + 2 \delta_0)}} I_n(\theta, x_1^n)~~~~~~~~~~~~~~~~~~~~~~~~~~~(\equiv A_4)\\
&\times \int_{x_0 \in  \cTo_0^{\delta}} \int_{x_1^{n-1} \in [-2-\eps_0, 2 + \eps_0]^{n-1}} \ones\{ L_n(x_0^{n-1}) \in \ball(\sigma_{\rm sc}, n^{-1/4}) \} \P_n^0( \d x_0^{n-1}) ~~~~~~~(\equiv A_5)\\
\end{aligned}
\end{equation}
The term $A_1$ is strictly positive and does not depend on $n$. Therefore it is exponentially trivial. 

Since $\Phi(\mu, t)$ is continuous on the set $\{(\mu, t): \mu \in \ball_{2 + \delta_0}(\sigma_{\rm sc}, \delta'), t \in \ball(\rho(\theta) , \eps_0 + 2 \delta_0)\}$, the term $A_3$
is lower bounded as follows
\begin{align}
\liminf_{n\rightarrow \infty} \frac{1}{n} \log A_3 \ge \inf_{x \in \ball(\rho(\theta ) , \eps_0 + 2 \delta_0)} \Big[\Phi_\star (x)-\frac{1}{4} x^2\Big].
\end{align}

For the term $A_4$, using the continuity of spherical integral Lemma \ref{lem:continuity_si2} and \ref{lem:ldp_si}, we have 
\begin{align}
\liminf_{n\rightarrow \infty}\frac{1}{n}\log A_4 \ge & J(\sigma_{\rm sc}, \rho(\theta ), \theta) - g_{\theta}(2 \eps_0 +4 \delta_0)\, , 
\end{align}
where $g_\theta(\,\cdot\,) = g_{1/4,\theta,\rho(\theta)+1}(\,\cdot\,)$

For the term $A_5$, we have 
\begin{equation}
\begin{aligned}
A_5 \ge & \E_{{\rm GOE}, n}\Big[ \frac{1}{n} \#\{\lambda_i: \lambda_i \in \cTo_0^{\delta}\} \Big] - \P_{{\rm GOE}, n}\Big(\max_{i \in [n]}\l \lambda_i \l  \geq 2 + \delta_0\Big) - \P_{{\rm GOE}, n}\Big( L_n \notin \ball(\sigma_{\rm sc}, n^{-1/4})\Big). \\
\end{aligned}
\end{equation}
The first term is exponentially trivial, the second term is exponentially decay, and the third term is exponentially vanishing. Therefore, $A_5$ is exponentially trivial. 

Putting the various terms together we get, for any $\theta \in (\theta_0 - \delta_0, \theta_0 + \delta_0)$ and $t_0 > 0$,
\begin{align}
\liminf_{n\rightarrow \infty} \frac{1}{n}\log \int_{\cTo_0^\delta} H_1 \d t \ge & \frac{1}{4} (t_0-\delta_0)^2  + 
J(\sigma_{\rm sc}, \rho(\theta ), \theta) - g_\theta (2 \eps_0 + 4 \delta_0) \\
&+ \inf_{x \in \ball(\rho(\theta ) , \eps_0 + 2 \delta_0)} [\Phi_\star (x)-1/4 \cdot x^2].
\end{align}

For any fixed $\theta \in (\theta_0 - \delta_0, \theta_0+\delta_0)$, letting $\eps_0, \delta_0 \rightarrow 0$ and using the continuity of $\Phi_\star(x)$ and $J(\sigma_{\rm sc}, x, \theta)$ in variable $x$ (eee Eqs. (\ref{eq:PhiDef}) and (\ref{eqn:JDef})), we have
\begin{equation}
\begin{aligned}
\liminf_{n\rightarrow \infty} \frac{1}{n}\log \int_{\cTo_0^\delta} H_1 \d t \ge & \frac{1}{4} t_0^2 + J(\sigma_{\rm sc}, \rho(\theta ), \theta)  + \Phi_\star (\rho(\theta ))-1/4 \cdot \rho(\theta )^2.
\end{aligned}
\end{equation}

Note that $\{ 1/n \cdot \log \int_{\cTo_0^\delta} H_1 \d t \}_{n \in \N_+}$ are convex functions and uniformly bounded in $\theta$. Therefore, according to Lemma \ref{lem:cvx_uniform}, the above inequality hold uniformly for $\theta \in (\theta_0 - \delta_0, \theta_0 + \delta_0)$. That is 
\begin{equation}
\begin{aligned}
\liminf_{n\rightarrow \infty} \frac{1}{n}\log \int_{\cEo_0^\delta} H_1 \d \theta \d t \ge & \liminf_{n\rightarrow \infty} \inf_{\theta \in \cUo_0^{\delta_0}}\frac{1}{n}\log \int_{\cEo_0^\delta} H_1 \d t\\
\ge & \inf_{\theta \in \cUo_0^{\delta_0}} \Big[\frac{1}{4} t_0^2 + J(\sigma_{\rm sc}, \rho(\theta ), \theta)  + \Phi_\star (\rho(\theta ))-1/4 \cdot \rho(\theta )^2\Big].
\end{aligned}
\end{equation}
Letting $\delta_0 \rightarrow 0$ gives the desired result. 

{\bf Case 2.2: } $t_0 \in (-2, 2)$ and $\theta_0 < 1$. We further assume $t_0 > 0$, as the case $t_0<0$. For any fixed small $\eps_0,\delta'>0$, we have lower bound
\begin{equation}
\begin{aligned}
\E\{ \l \det(\He_n) \l \} =&\frac{ Z_{n}^0}{Z_{n}^\theta} \int_{\R^n} \exp\{n \cdot \Phi(L_{n}, t) \} \cdot I_n(\theta, x_1^n)   \cdot \d \P_{n}^0\\
\geq&\frac{ Z_{n}^0}{Z_{n}^\theta} \int_{\R^n} \exp\{n \cdot \Phi(L_{n}, t)  \} \cdot I_n(\theta, x_1^n)   \cdot \ones\{L_{n} \in \ball_{2 + \eps_0}(\sigma_{\rm sc}, \delta') \}  \cdot \d \P_{n}^0\\
\ge &   \frac{ Z_{n}^0}{Z_{n}^\theta} \cdot \Big\{\underbrace{\int_{\R^n} \exp\{n \cdot \Phi(L_{n}, t)  \}    \cdot \ones\{L_{n} \in \ball_{2 + \eps_0}(\sigma_{\rm sc}, \delta') \}  \cdot \d \P_{n}^0 }_{F_1} \cdot \underbrace{  \inf_{L_n \in \ball_{2 + \eps_0}(\sigma_{\rm sc}, \delta') } I_n(\theta, x_1^n)  }_{F_2}\Big\}\\
\end{aligned}
\end{equation}

For the term $F_1$, we have 
\begin{equation}
\begin{aligned}
F_1 \ge \underbrace{\int_{[-2 - \eps_0, 2 + \eps_0]^n} \exp\{n \cdot \Phi(L_{n}, t)  \} \cdot \d \P_{n}^0}_{F_3} - \underbrace{\int_{\R^n} \exp\{n \cdot \Phi(L_{n}, t)  \}    \cdot \ones\{L_{n} \notin \ball(\sigma_{\rm sc}, \delta') \}  \cdot \d \P_{n}^0}_{F_4}.
\end{aligned}
\end{equation}

According to Lemma \ref{lem:exp_vanishing}, $F_4 = B_n^1$ is exponentially vanishing on compact sets.   For the term $F_3$, letting $0 < \delta_0 < \delta$, we have 
\begin{equation}
\begin{aligned}
& \int_{t \in \cTo_0^\delta} \d t\int_{[-2 - \eps_0, 2 + \eps_0]^n} \exp\{n \cdot \Phi(L_{n}, t)  \} \cdot \d \P_{n}^0 \\
=& \frac{1}{Z_n^0}\int_{t \in \cTo_0^\delta} \int_{[-2 - \eps_0, 2 + \eps_0]^n} \prod_{t = 1}^n \l t - x_i \l \cdot \prod_{1 \le i < j \le n} \l x_i - x_j \l  \cdot \exp\Big\{- \frac{n}{4}\sum_{i=1}^n x_i^2 \Big\} \cdot \prod_{i=1}^n \d x_i  \cdot \d t\\
=& \frac{1}{Z_n^0}\int_{x_0 \in \cTo_0^\delta} \int_{[-2 - \eps_0, 2 + \eps_0]^n} \prod_{0 \le i < j \le n} \l x_i - x_j \l  \cdot \exp\Big\{- \frac{n}{4}\sum_{i=0}^n x_i^2 \Big\} \cdot \prod_{i=0}^n \d x_i \cdot \exp\Big\{ \frac{n}{4} x_0^2 \Big\}\\
\ge & \frac{1}{Z_n^0}\Big(1+ \frac{1}{n}\Big)^{\frac{(n+1)(n+2)}{4}}\int_{y_0 \in \sqrt{\frac{n}{n+1}} \cTo_0^{\delta_0}} \int_{[-2 - \eps_0/2, 2 + \eps_0/2]^n} \prod_{0 \le i < j \le n} \l y_i - y_j \l  \cdot \exp\Big\{- \frac{n+1}{4}\sum_{i=0}^n y_i^2 \Big\} \cdot \prod_{i=0}^n \d y_i \\
&\times \exp\Big\{ \frac{n}{4}  (t_0-\delta_0)^2 \Big\}\\
= & \frac{Z_{n+1}^0}{Z_n^0}\Big(1+ \frac{1}{n}\Big)^{\frac{(n+1)(n+2)}{4}} \E_{\rm GOE}^{n+1} \Big[\frac{1}{n+1}\# \Big\{ \lambda_i: \lambda_i \in \sqrt{\frac{n}{n+1}}\cTo_0^{\delta_0}   \Big\} \cdot \ones\{ \max\l \lambda_i \l \le 2+\eps_0/2\} \Big] \\
& \times \exp\Big\{ \frac{n}{4} (t_0-\delta_0)^2\Big\}\\
\end{aligned}
\end{equation}

Using the Selberg's integral formula, we have 
\begin{align}
\lim_{n\rightarrow \infty} \frac{1}{n} \log \Big\{ \frac{Z_{n+1}^0}{Z_n^0}\Big(1+ \frac{1}{n}\Big)^{\frac{(n+1)(n+2)}{4}} \Big\} = -\frac{1}{2}. 
\end{align}

Similar to the method dealing with the term $A_5$ in Eq. (\ref{eqn:decomp1}), we have
\begin{align}
\lim_{n\rightarrow \infty} \frac{1}{n} \log \E_{\rm GOE}^{n+1} \Big[\frac{1}{n+1}\# \Big\{ \lambda_i: \lambda_i \in \sqrt{\frac{n}{n+1}}\cTo_0^{\delta_0}   \Big\} \cdot \ones\{ \max\l \lambda_i \l \le 2+\eps_0/2\} \Big]  = 0.
\end{align}

Now we turn to look at the term $F_2$. For any fixed $\theta \in \cUo_0^\delta$, there is a margin between $\theta$ and $1$, so we can find $\eta$ small enough so that $\theta \in \cup_{\mu \in \ball(\sigma_{\rm sc}, \delta ')} H_\mu([-2 - \eps_0 - \eta, 2 + \eps_0 + \eta]^c)$ as $\eps_0, \delta'$ is small enough. Due to the continuity of the spherical integral, cf. Lemma \ref{lem:continuity_si} and \ref{lem:ldp_si}, there exists $g_{\theta, \eta}(\delta)> 0$ as $\delta > 0$, and $\lim_{\delta \rightarrow 0} g_{\theta, \eta}(\delta)= 0$, such that for all $n$ large enough,
\begin{align}
\frac{1}{n} \log  \inf_{L_n \in \ball_{2 + \eps_0}(\sigma_{\rm sc}, \delta') } I_n(\theta, x_1^n) \ge J(\sigma_{\rm sc}, 2 + \eps_0,\theta) - g_{\theta, \eta}(\delta').
\end{align}
Using the right-continuity of function $J(\sigma_{\rm sc}, x, \theta)$ with respect to $x$ at $x = 2$, we have
\begin{align}
\liminf_{\eps_0, \delta'\rightarrow 0_+}\liminf_{n\rightarrow \infty} \frac{1}{n} \log  \inf_{L_n \in \ball_{2 + \eps_0}(\sigma_{\rm sc}, \delta') } I_n(\theta, x_1^n) \ge J(\sigma_{\rm sc}, 2, \theta).
\end{align}

Therefore for any fixed $\theta \in \cUo_0^\delta$,
\begin{equation}
\begin{aligned}
&\liminf_{n \rightarrow \infty} \frac{1}{n}\log\int_{t \in \cTo_0^\delta}  \int_{\R^n} \exp\{n \cdot \Phi(L_{n}, t)  \} \cdot I_n(\theta, x_1^n)   \cdot \d \P_{n}^0 \\
\ge& \limsup_{\delta_0\rightarrow 0}  \Big\{J(\sigma_{\rm sc}, 2,\theta) + \frac{1}{4}(t_0 - \delta_0)^2 - \frac{1}{2}\Big\}= J(\sigma_{\rm sc}, 2, \theta) + \Phi_\star(t_0). 
\end{aligned}
\end{equation}
Since $1/n \cdot \log\int_{t \in \cTo_0^\delta}  \int_{\R^n} \exp\{n \cdot \Phi(L_{n}, t)  \} \cdot I_n(\theta, x_1^n)   \cdot \d \P_{n}^0$ is convex in $\theta$, according to Lemma \ref{lem:cvx_uniform}, we have
\begin{equation}
\begin{aligned}
&\liminf_{n \rightarrow \infty} \frac{1}{n}\log\int_{(\theta, t) \in \cEo_0^\delta}  \int_{\R^n} \exp\{n \cdot \Phi(L_{n}, t)  \} \cdot I_n(\theta, x_1^n)   \cdot \d \P_{n}^0 \ge J(\sigma_{\rm sc},2, \theta_0) + \Phi_\star(t_0). 
\end{aligned}
\end{equation}

By \cite[Proposition 3.1]{maida2007large}, we have for fixed $\theta \in (\theta_0 - \delta, \theta_0 + \delta)$, we have
\begin{align}
\limsup_{n\rightarrow \infty} \frac{1}{n}\log(Z_n^\theta/ Z_n^0)~\le~ & J(\sigma_{\rm sc}, 2, \theta).
\end{align}
By the convexity of $\sup_{t \in \cTo_0^\delta} \frac{1}{n} \log( Z_n^\theta/ Z_n^0)$ as a function of $\theta$, we have 
\begin{align}
\limsup_{\delta \rightarrow 0_+} \limsup_{n\rightarrow \infty} \sup_{\theta \in \cUo_0^\delta}\frac{1}{n}\log(Z_n^\theta/ Z_n^0)~\le~ & J(\sigma_{\rm sc}, 2, \theta_0).
\end{align}

Therefore, as $t_0 \in (-2,2)$ and $\theta_0 < 1$, we have
\begin{equation}
\begin{aligned}
&\liminf_{n\rightarrow \infty}  \frac{1}{n}\log \int_{(\theta, t) \in \cEo_0^\delta}  \E\{\l \det(\He_n) \l \} \d \theta \d t \ge  \Phi_\star(t_0).
\end{aligned}
\end{equation}

\subsection{Proof of Theorem \ref{thm:local_maxima}}
\label{sec:ProofMaxima}

\begin{proposition}\label{prop:local_maxima}
The following statements hold
\begin{enumerate}[label = (\alph*)]
\item Exponential tightness. 
\begin{align}\label{eqn:exponential_tightness2}
\lim_{z \rightarrow \infty}\limsup_{n\rightarrow \infty}\frac{1}{n}\log\E\{\Crt_{n,\knot}([-1,1], ( -\infty, -z] \cup [z, \infty))\}  = -\infty.
\end{align}
\item Upper bound. For any fixed large $U_0>0$ and $T_0>0$, denote $\cUb_0 \subset [- U_0, U_0]$ and $\cTb_0 \subset [ -T_0, T_0]$ to be two compact sets, and denote $\cEb_0 \eqndef \cUb_0  \times  \cTb_0$. Then we have
\begin{align}\label{eqn:local_maxima_ub}
\limsup_{n\rightarrow \infty} \sup_{(\theta, t) \in \cEb_0} \frac{1}{n}\log \E\{  \l \det( \He_n )\l \cdot \ones \{ \He_n \preceq 0 \} \} \leq \sup_{ (\theta,t) \in \cEb_0}  [\Phi_\star(t) - L(\theta,t)]
\end{align}
\item Lower bound. For any fixed $\delta > 0$, $\theta_0$ and $t_0$, denote $\cUo_0^\delta = (\theta_0 - \delta, \theta_0 + \delta)$ and $\cTo_0^\delta = (t_0 - \delta, t_0 + \delta)$, and denote $\cEo_0^\delta \eqndef \cUo_0^\delta \times \cTo_0^\delta$. Then we have
\begin{align}
\liminf_{n\rightarrow \infty}  \frac{1}{n}\log \int_{(\theta, t) \in \cEo_0^\delta} \E\{\l \det(\He_n) \l \cdot \ones\{\He_n \preceq 0\} \} \d \theta \d t \geq  \Phi_\star(t_0) - L(\theta_0, t_0).
\end{align}
\end{enumerate}
\end{proposition}

Assume this proposition holds, we are in a good position to prove Theorem \ref{thm:local_maxima}. 

\begin{proof}

Because of the exponential tightness property, we only need to consider the case when $E$ is bounded. 

\noindent
{\bf Step 1. Upper bound. }
Denoting $E_0 = (x_0 - \delta_0, x_0 + \delta_0)$. Using the same argument as in the proof of upper bound in Theorem \ref{thm:critical_point}, we just need to show that
\begin{equation}\label{eqn:Thm2UpperFirst}
\begin{aligned}
\lim_{\delta_0 \rightarrow 0_+}\limsup_{n\rightarrow \infty}\frac{1}{n}\log \E\{\Crt_{n, \knot}(M,E_0) \} \leq \sup_{m \in \Mb} S_\knot(m, x_0). 
\end{aligned}
\end{equation}

For any small $\delta > 0$, define 
\begin{equation}
\begin{aligned}
\Mb_\delta =& \{ m: d(m, M) \leq \delta \},\\
\Eb_\delta =& \{ x: d(x, E_0) \leq \delta \},\\
\cUb_\delta = & \{ \theta: \theta = \sqrt{2k(k-1)}  \cdot \lambda m^{k-2} (1-m^2), \, m \in \Mb_\delta \},\\
\cTb_\delta = & \{ t: t = \sqrt{2k/(k-1)}  \cdot x, \, x\in \Eb_\delta \},\\
\cEb_\delta =& \cUb_\delta \times\cTb_\delta.
\end{aligned}
\end{equation}
Since $E_0$ is bounded, we can define finite constants $R_0 = \sup\{\l x \l: x \in E_0 \}$, $U_0 = 2 \sup\{\l \sqrt{2k/(k-1)}  \cdot x \l: x\in E_0 \}$ and $T_0 = 2 \sup\{ \l \sqrt{2k(k-1)} \cdot \lambda m^{k-2} (1-m^2) \l: m \in M \}$. Therefore, as $\delta$ is sufficiently small, we have $\cUb_\delta \subset [-U_0, U_0]$ and $\cTb_\delta \subset [-T_0, T_0]$. 

We only prove the case for $(M,E_0)$ such that $\sup_{ (\theta,t) \in \cEb_0}  [\Phi_\star(t) - L(\theta,t)]  > -\infty$. For $(M, E_0)$ such that $\sup_{ (\theta,t) \in \cEb_0}  [\Phi_\star(t) - L(\theta,t)]  = -\infty$, we can prove it using similar arguments.

According to Proposition \ref{prop:local_maxima}.(b), for any $\eps > 0$ and $\delta > 0$, there exists $N_{\eps, \delta}$ large enough, such that $t_n(x) \in \cTb_\delta$ and $\theta_n(m) \in \cUb_\delta$ for all $(m,x) \in M \times E_0$, and for all $n \geq N_{\eps, \delta}$,
\begin{equation}
\begin{aligned}
&\sup_{m \in \Mb, x \in \Eb_0}\E\{ \l \det(\theta_n(m) \cdot \be_1 \be_1^\sT + W_{n-1} - t_n(x) \cdot \id_{n-1}) \l \cdot \ones\{H_n \preceq 0\}\} \\
\leq& \sup_{(\theta, t) \in \cEb_\delta}\E\{ \l \det(\He_{n-1}) \l \cdot \ones\{ \He_{n-1} \preceq 0 \} \} \leq \exp\{ (n-1) [\sup_{(\theta,t) \in \cEb_\delta}\Phi_\star(t) - L(\theta,t)+ \eps]  \} 
\end{aligned}
\end{equation}
Therefore, using Eq. (\ref{eqn:cp_st}) in Lemma \ref{lem:exact}, we have
\begin{equation}
\begin{aligned}
&\E\{\Crt_{n, \knot}(M,E_0) \} \\
\leq& \sup_{m \in M, x \in E_0} 2 \PC_n \cdot R_0 \times \exp\Big\{n\Big[ \frac{1}{2} (\log(k-1) + 1) - k \lambda^2 m^{2k-2} (1-m^2)  - (x - \lambda m^k)^2 \Big]\Big\}  \\
& \times \exp\Big\{(n-3)\Big[\frac{1}{2} \log(1-m^2) \Big] + (n-1) \cdot \sup_{(\theta,t) \in \cEb_\delta} [\Phi_\star(t) - L(\theta,t)+ \eps]\Big\}.
\end{aligned}
\end{equation}
Note that the pre-constant $2 \PC_n R_0$ is exponentially trivial. We have
\begin{equation}
\begin{aligned}
&\limsup_{n\rightarrow \infty}\frac{1}{n}\log \E\{\Crt_{n, \knot}(M,E_0) \} \\
\leq &\sup_{m \in \Mb, x \in \Eb_0} \frac{1}{2}(\log(k-1)+1) + \frac{1}{2}\log(1-m^2) - k \lambda^2 m^{2k-2} (1-m^2) - (x- \lambda m^k) ^2 \\
&+ \sup_{(\theta,t) \in \cEb_\delta} \Phi_\star(t) - L(\theta,t) + \eps.
\end{aligned}
\end{equation}
Letting $\eps, \delta \rightarrow 0_+$, and using the upper semi-continuity of $\Phi_\star(t) - L(\theta,t)$ and compactness of $\cEb_0$, we have
\begin{equation}
\begin{aligned}
&\limsup_{n\rightarrow \infty}\frac{1}{n}\log \E\{\Crt_{n, \knot}(M,E_0) \} \\
\leq &\sup_{m \in \Mb, x \in \Eb_0} \frac{1}{2}(\log(k-1)+1) + \frac{1}{2}\log(1-m^2) - k \lambda^2 m^{2k-2} (1-m^2) - (x- \lambda m^k) ^2 \\
&+ \sup_{(\theta,t) \in \cEb_0} \Phi_\star(t) - L(\theta,t).
\end{aligned}
\end{equation}
Note that we took $E_0 = (x_0 - \delta_0, x_0 + \delta_0)$, letting $\delta_0 \rightarrow 0$ and using the upper semi-continuity of $\Phi_\star(t) - L(\theta,t)$ gives Eq. (\ref{eqn:Thm2UpperFirst}). 

\noindent
{\bf Step 2. Lower bound. }

Suffice to consider the case when $\sup_{(m,x) \in M^o \times E^o} S_\knot(m,x) > -\infty$, otherwise the inequality holds trivially. 

For any Borel sets $M \subset [-1,1]$ and $E \subset \R$, and for any $\eps > 0$, there exists $(m_0, x_0) \in M^o \times E^o$ such that 
\begin{align}
S_\knot(m_0, x_0) \ge \sup_{(m,x) \in M^o \times E^o} S_\knot(m,x) - \eps.
\end{align}
For this choice of $(m_0, x_0)$, denote $\theta_0 = \theta(m_0)$ and $t_0 = t(x_0)$. For a given small $\delta > 0$, define 
\begin{equation}
\begin{aligned}
M_0^\delta \eqndef& (m_0 - \delta, m_0 + \delta),\\
E_0^\delta \eqndef& (x_0 - \delta, x_0 + \delta),\\
\cB_0^\delta \eqndef& M_0^\delta \times E_0^\delta,\\
\cUo_{n}^\delta \eqndef&\{ \theta: \theta = \sqrt{2k(k-1)n/(n-1)}  \cdot \lambda m^{k-2} (1-m^2), \, m \in M_0^\delta  \},\\
\cTo_{n}^\delta \eqndef & \{ t: t = \sqrt{2kn/((k-1)(n-1))} \cdot x, \, x \in E_0^\delta \},\\
\cEo_n^\delta \eqndef& \cUo_n^\delta \times \cTo_n^\delta.
\end{aligned}
\end{equation}
We fix $\delta$ sufficiently small, so that $M_n^\delta \subset M^o$ and $E_n^\delta \subset E^o$. 

For this choice of $\delta$ and $\eps$, according to Proposition \ref{prop:local_maxima}.(c), for any $\eps_0 > 0$, we can find $N_{\eps, \eps_0, \delta}$ and $\delta_0 > 0$ such that as $n \geq N_{\eps, \eps_0, \delta}$, 
\begin{equation}
\begin{aligned}
\cEo_{0}^{\delta_0} \eqndef&~ (\theta_0 - \delta_0, \theta_0 + \delta_0) \times (t_0 - \delta_0, t_0 + \delta_0) \subset  \cEo_n^\delta,\\
\end{aligned}
\end{equation}
and 
\begin{equation}
\begin{aligned}
 \int_{(\theta, t) \in \cEo_{0}^{\delta_0}} \E\{\l \det(\He_{n-1} ) \l \cdot \ones\{ \He_{n-1} \preceq 0 \} \} \d \theta \d t \geq \exp\{ (n-1) [\Phi(t_0)- L(\theta_0, t_0) - \eps_0]\}.
\end{aligned}
\end{equation}

According to the expression for the expected number of critical point as in Eq. (\ref{eqn:lm_st}) in Lemma \ref{lem:exact}, 
\begin{equation}
\begin{aligned}
&\E\{\Crt_{n, \knot}(M,E) \}  \geq \E\{\Crt_{n, \knot}(M_0^\delta, E_0^\delta) \}\\
\ge & \PC_n \cdot \int_{\cB_0^\delta} \E \{ \l \det (H_n) \l \cdot \ones\{H_n \preceq 0\} \} \d x \d m \times \inf_{(m, x) \in \cB_0^\delta} \exp \Big\{ (n-3) \cdot \Big[ \frac{1}{2} \log(1-m^2)\Big] \\
& \quad \quad +n\Big[ \frac{1}{2} (\log(k-1) + 1) - k \lambda^2 m^{2k-2} (1-m^2)  - (x - \lambda m^k)^2 \Big]\Big\}  \\
\ge & \PC_n \cdot \int_{\cEo_{0}^{\delta_0}} \E \{ \l \det(\He_{n-1}) \l \cdot \ones\{\He_{n-1} \preceq 0\} \} \frac{n-1}{2k\lambda n [(k-2) \cdot m(\theta)^{k-3} - k\cdot m(\theta)^{k-1}]}\d \theta \d t\\
& \times \inf_{(m, x) \in \cB_0^\delta} \exp \Big\{ n\Big[ \frac{1}{2} (\log(k-1) + 1) +\frac{1}{2} \log(1-m^2) - k \lambda^2 m^{2k-2} (1-m^2)  - (x - \lambda m^k)^2 \Big]\Big\}  \\
\ge & \frac{\PC_n}{8k^2 \lambda} \cdot  \exp\Big\{ (n-1) \cdot [\Phi_\knot(t_0) - L(\theta_0, t_0) - \eps_0] \Big\}\\
& \times \inf_{(m, x) \in \cB_0^\delta} \exp \Big\{ n\Big[ \frac{1}{2} (\log(k-1) + 1) +\frac{1}{2} \log(1-m^2) - k \lambda^2 m^{2k-2} (1-m^2)  - (x - \lambda m^k)^2 \Big]\Big\}.
\end{aligned}
\end{equation}
Note that the pre-constant $\PC_n /8k^2$ is exponentially trivial on compact set. We have
\begin{equation}
\begin{aligned}
&\liminf_{n\rightarrow \infty}\frac{1}{n}\log \E\{\Crt_{n,  \knot}(M,E) \} \\
\geq &\inf_{(m,x) \in  \cB_0^\delta } \Big\{\frac{1}{2}(\log(k-1)+1) + \frac{1}{2}\log(1-m^2) - k \lambda^2 m^{2k-2} (1-m^2) - (x- \lambda m^k) ^2 \Big\} \\
&~~~~ +  \Phi_\star(t_0) -L(\theta_0, t_0)- \eps_0.
\end{aligned}
\end{equation}
Letting $\eps_0, \delta \rightarrow 0_+$, we have
\begin{align}
\liminf_{n\rightarrow \infty}\frac{1}{n}\log \E\{\Crt_{n, \knot}(M,E) \} \geq  S_\knot(m_0,x_0) \ge \sup_{m \in M^o, x \in E^o} S_\knot(m,x) - \eps.
\end{align}
Letting $\eps \rightarrow 0_+$ gives the desired result.

\end{proof}

For Proposition \ref{prop:local_maxima}, the exponential tightness is trivial since we have the exponential tightness of the expected number of critical points. In the following, we will prove the upper bound and the lower bound.

\subsubsection{Part (1). Upper bound}

We decompose 
\begin{equation}
\begin{aligned}
&\E\{\l \det(\He_n) \l \cdot \ones\{ \He_n \preceq 0 \}\}\\
 \leq & \underbrace{\E\{\l \det(\He_n) \l \cdot \ones\{ \He_n \preceq 0 \}; L_n \in \ball(\sigma_{\rm sc},\delta) \}}_{F_1} + \underbrace{\E\{\l \det(\He_n) \l, L_n \notin \ball(\sigma_{\rm sc},\delta) \}}_{E_2} \\
\end{aligned}
\end{equation}
where $\delta >0 $ is an arbitrary small number. 

According to Lemma \ref{lem:exp_vanishing}, $E_2 = B_n^3 / A_n^2$ as a function of $(\theta, t)$ is exponentially vanishing on compact set. We just need to consider the term $F_1$.

%
%
%
%

For the term $F_1$, we have
\begin{equation}
\begin{aligned}
F_1 =& \frac{Z_n^0}{Z_{n}^\theta} \int_{\R^n}\exp\{ n\cdot \Phi(L_n, t) \} \cdot \ones\{ \max_{i\in [n]}\{x_i\} \leq t , L_n \in \ball(\sigma_{\rm sc}, \delta)\} I_n(\theta, x_1^n) \cdot \d \P_n^0 \\
\leq & \exp\{ n\cdot \sup_{\mu \in \ball(\sigma_{\rm sc}, \delta)} \Phi(\mu, t)\} \cdot \frac{Z_n^0}{Z_{n}^\theta} \int_{\R^n}  \ones\{ \max_{i\in [n]}\{x_i\} \leq t\} I_n(\theta, x_1^n) \cdot \d \P_n^0 \\
=&  \exp\{ n\cdot \sup_{\mu \in \ball(\sigma_{\rm sc}, \delta)} \Phi(\mu, t)\} \cdot \P(\lambda_{\max}(X_n) \leq t). 
\end{aligned}
\end{equation} 

According to Lemma \ref{lem:LDP_spiked}, and note that $\P(\lambda_{\max} (X_n) \leq t)$ is a coordinate-wise monotone function with respect to $(\theta,t)$,  we have
\begin{align}
\limsup_{n\rightarrow \infty} \sup_{(\theta, t) \in \cEb_0} \frac{1}{n} \log \P( \lambda_{\max}(X_n) \leq t) \leq - \inf_{(\theta, t) \in \cEb_0} L(\theta,t).
\end{align}
Consequently, 
\begin{align}
\lim_{\delta\rightarrow 0_+} \limsup_{n \rightarrow \infty}  \sup_{ (\theta, t) \in \cEb_0}  \frac{1}{n} \log F_1   \leq \lim_{\delta\rightarrow 0_+} \sup_{ (\theta,t) \in \cEb_0,\mu \in \ball(\sigma_{\rm sc},\delta)} [\Phi(\mu, t)    - L(\theta,t)]\leq \sup_{(\theta,t) \in \cEb_0}[ \Phi_\star(t) - L(\theta,t)].
\end{align}
Therefore, we have
\begin{align}
 \limsup_{n \rightarrow \infty}  \sup_{ (\theta, t) \in \cEb_0}  \frac{1}{n} \log \E\{\l\det(\He_n)\l \cdot \ones\{\He_n \preceq 0\}\}   \leq \sup_{(\theta,t) \in \cEb_0}[ \Phi_\star(t) - L(\theta,t)].
\end{align}

\subsubsection{Part (2). Lower bound}

For the lower bound, since $\Phi_\star (t) - L(\theta, t)$ is upper semi-continuous, we only need to prove it for those $(\theta_0, t_0)$ in a dense subset of $\R^2$.
Since as $t_0 \in (-\infty, 2)$, we have $L(\theta, t_0) = \infty$ for any $\theta$. So we only need to consider the case when $t_0 > 2 $. 

Fix $t_0 > 2$, choose $\delta_0 > 0$ and $\eps_0 > 0$ such that $\delta_0 < \delta$, and $2 <  t_0 - \delta_0 - \eps_0 < t_0 - \delta_0$. We have
\begin{equation}
\begin{aligned}
&\E\{ \l \det(\He_n) \l \cdot \ones\{\He_n \preceq 0\}\}\\
=&\frac{ Z_{n}^0}{Z_{n}^\theta} \int_{\R^n} \exp\{n \cdot \Phi(L_n, t)  \}  \cdot \ones\{ x_{\max} \le t\}\cdot I_n(\theta, x_1^n)   \cdot \d \P_{n}^0(x_1^n)\\
\geq&\frac{ Z_{n}^0}{Z_{n}^\theta} \int_{\R^n} \exp\{n \cdot \Phi(L_n, t)  \}  \cdot I_n(\theta, x_1^n)   \cdot \ones\{ x_{\max} \leq \min \{ t, t_0 - \delta_0 - \eps_0\}, L_n \in \ball(\sigma_{\rm sc}, \delta') \}  \cdot \d \P_{n}^0\\
\geq&\frac{ Z_{n}^0}{Z_{n}^\theta} \int_{\R^n} I_n(\theta, x_1^n)\ones\{x_{\max} \leq \min \{ t, t_0 - \delta_0 - \eps_0\}, L_{n} \in \ball(\sigma_{\rm sc}, \delta') \} \cdot  \d \P_{n}^0 \\
&\times \exp\Big\{n \Big[ \inf_{L_n \in \ball(\sigma_{\rm sc}, \delta'), x_{\max} \leq \min\{t, t_0 - \delta_0 - \eps_0\}}\Phi(L_n, t) \Big] \Big\}\\
\ge &\Big\{\P(\lambda_{\max}(X_n) \leq \min \{ t, t_0 - \delta_0 - \eps_0\}) - \underbrace{\frac{ Z_{n}^0}{Z_{n}^\theta} \int_{\R^n} I_n(\theta, x_1^n) \ones\{L_{n} \notin \ball(\sigma_{\rm sc}, \delta')\}  \} \d \P_{n}^0}_{G_2}\Big\} \\ 
&\times \exp\{n [ \inf_{L_n \in \ball(\sigma_{\rm sc}, \delta'), x_{\max} \leq \min\{ t, t_0 - \delta_0 - \eps_0\}}\Phi(L_n, t) ] \}\\
\end{aligned}
\end{equation}
According to Lemma \ref{lem:exp_vanishing}, $G_2 = B_n^2/A_n^2$ is exponential vanishing on compact set, so we can drop this term. 

According to Lemma \ref{lem:LDP_spiked}, and note that $\P(\lambda_{\max} (X_n) \leq t)$ is a coordinate-wise monotone function with respect to $(\theta,t)$ and $L(\theta,t)$ are continuous as $t > 2$,  we have
\begin{align}
\liminf_{n\rightarrow \infty} \inf_{(\theta, t) \in \cEo_0^{\delta_0}} \frac{1}{n} \log  \P( \lambda_{\max}(X_n) < \min \{ t, t_0 - \delta_0 - \eps_0\}) \geq - L(\theta_0 + \delta_0, t_0 - \delta_0 - \eps_0).
\end{align}

This gives
\begin{equation}
\begin{aligned}
& \liminf_{n\rightarrow \infty} \frac{1}{n} \log \int_{(\theta, t) \in \cEo_0^\delta}  \E\{ \l \det(\He_n) \l \cdot \ones\{\He_n \preceq 0\}\}  \d \theta \d t\\
\geq& \lim_{\delta' \rightarrow 0_+}\inf_{t \in \cTo_0^{\delta_0}, L_n \in \ball(\sigma_{\rm sc},\delta'), \lambda_{\max} \leq t_0 - \delta_0 - \eps_0} \Phi(L_n,t)  - L(\theta_0 + \delta_0, t_0 - \delta_0 - \eps_0)\\
= & \Phi_\star(t_0 + \delta_0 /2)  - L(\theta_0 + \delta_0, t_0 - \delta_0 - \eps_0)\\
\end{aligned}
\end{equation}
Since $\Phi_\star (t) - L(\theta, t)$ is continuous with respect to $(\theta, t)$ on $\R \times (2, \infty)$, letting first $\eps_0 \rightarrow 0_+$ and then $\delta_0 \rightarrow 0_+$, we have 
\begin{equation}
\begin{aligned}
&\liminf_{n\rightarrow \infty} \frac{1}{n} \log \int_{(\theta, t) \in \cEo_0^\delta}  \E\{ \l \det(\He_n) \l \cdot \ones\{\He_n \preceq 0\}\}  \d \theta \d t\\
\ge & \limsup_{\delta_0 \rightarrow 0_+ } \limsup_{\eps_0 \rightarrow 0_+} \Phi_\star(t_0 + \delta_0 /2)  - L(\theta_0 + \delta_0, t_0 - \delta_0 - \eps_0) =  \Phi_\star (t_0) - L(\theta_0, t_0).
\end{aligned}
\end{equation}

\section*{Acknowledgements}

S.M. was supported by an Office of Technology Licensing Stanford Graduate Fellowship.
A. M. was partially supported by grants NSF CCF-1714305 and NSF
IIS-1741162. M. N. was supported by a Postdoctoral Fellowship from the Natural Sciences and Engineering Research Council of Canada.

\appendix

\section{Technical lemmas}

\begin{lemma}\label{lem:cvx_uniform}
Let $\{ f_n(x) \}_{n \in \N_+}$ be a series of real valued functions defined on the same compact interval $[a, b]$. Suppose that each of the $f_n(x)$ are convex, and $\{f_n(x)\}_{n \in \N_+}$ are uniformly bounded. Then we have 
\begin{align}
\liminf_{n\rightarrow \infty} \inf_{x \in [a, b]}f_n(x) = \inf_{x \in [a,b]} \liminf_{n\rightarrow \infty} f_n(x). 
\end{align}
\end{lemma}

\begin{proof}
It is obvious that the left hand side is smaller or equal to the right hand side. Suffice to prove that the left hand side is bigger or equal to the right hand side. 

Proof by contradiction. Assume the left hand side is smaller than the right hand side by a margin $\eps$. We call the righthand side $f_\star$. Then we have an increasing sequence $n_k \in \N_+$, such that 
\begin{align}
f_n(x_{n_k}) \le f_\star - \eps. 
\end{align}
The sequence $x_{n_k}$ have an accumulation point $x_\star \in [a,b]$. Since $f_n(x)$ are uniformly bounded, let $\sup_{x \in [a, b], n \in \N_+} f_n(x) - f_\star \le U$. Consider the interval $\cI = [x_\star - (b - x_\star)\eps/(2U+\eps), x_\star + (x_\star - a)\eps / (2U + \eps)]$. Since $\lim_{k\rightarrow \infty} x_{n_k} = x_\star$, then there exists $K$ large enough such that as $k \ge K$, we have $x_{n_k} \in \cI$. For any $x \in \cI \cap [x_\star, b]$, because of the convexity of $f_n$, we have
\begin{align}
f_n(x_\star) \le (x-x_\star)/(x - a) \cdot f_n(a) + (x_\star - a)/(x - a) \cdot f_n(x). 
\end{align}
For $k$ such that $x_{n_k} \in \cI \cap [x_\star, b]$, we have 
\begin{equation}
\begin{aligned}
f_{n_k}(x_\star) \le& (x_{n_k} - x_\star)/(x_{n_k} - a) \cdot (f_\star + U) + (x_\star - a) / (x_{n_k} - a) \cdot (f_\star - \eps)\\
\le& f_\star + [(x_{n_k} - x_\star) U - (x_\star - a) \eps]/(x_{n_k} - a)\\
\le & f_\star + [U (x_\star - a)\eps / (2U + \eps) - (x_\star - a) \eps]/((x_\star - a)\eps / (2U + \eps) + x_\star - a)\\
\le & f_\star + \eps [U / (2U + \eps) - 1]/[\eps / (2U + \eps) + 1] \le f_\star - \eps / 2.\\
\end{aligned}
\end{equation}
Similarly, for $k$ such that $x_{n_k} \in \cI \cap [a, x_\star]$, we also have $f_{n_k}(x_\star) \le f_\star - \eps/2$. Therefore
\begin{align}
\liminf_{n\rightarrow \infty} f_n(x_\star) \le \liminf_{k\rightarrow \infty} f_{n_k}(x_\star) \le f_\star - \eps/2,
\end{align}
which contradict the definition of $f_\star$. 
 
\end{proof}

The following lemma is from \cite[Lemma 6.3.]{arous2001aging}. 

\begin{lemma}[Concentration of operator norm of GOE matrix.] \label{lem:concentration_goe}
Let $\bW_N \sim \GOE(n)$.  Then there exists a constant $t_0$ such that, for all $t\ge t_0$ and all $n$ large enough, we have
\begin{align}\label{eqn:concentration_norm_goe}
\P(\dl \bW_n \dl_{\op} \geq t) \leq \exp\{ - n t^2/9 \}.
\end{align}
\end{lemma}

\begin{lemma}\label{lem:exp_tight_bound}
We have 
\begin{align}
\lim_{z\rightarrow \infty}\lim_{n\rightarrow \infty} \frac{1}{n}\log \int_{z}^\infty  x^n \exp\{ -n x^2 \}\d x = - \infty,
\end{align}
\end{lemma}

\begin{proof}
For large enough $x$, we have $x^2/2 \leq x^2 - \log x$. Therefore, for large enough $z$, the following holds
\begin{equation}
\begin{aligned}
\int_{z}^\infty  x^n \exp\{ -n x^2 \}\d x =& \int_z^\infty \exp\{ -n (x^2 - \log x) \}\d x \leq  \int_z^\infty  \exp\{ -n x^2/2 \}\d x \\
\leq & \int_z^\infty x\cdot \exp\{ -n x^2/2 \}\d x = \frac{1}{n} \exp\{ -nz^2/2 \}.
\end{aligned}
\end{equation}
This proves the claim. 
\end{proof}

\begin{lemma}\label{lem:exp_vanishing}
For the following quantities as functions of $(\theta,t)$, we have $A_n^1$, $A_n^2$, and $A_n^3$ are exponentially finite on any compact set, and $B_n^1$, $B_n^2$, and $B_n^3$ are exponentially vanishing on any compact set.
\begin{equation}
\begin{aligned}
&A_n^1 = \int_{\R^n} \exp\{ n \Phi(L_n, t) \} \d \P_n^0, &~~ &B_n^1 = \int_{\R^n} \exp\{ n \Phi(L_n, t) \} \ones\{ L_n \notin \ball(\sigma_{\rm sc}, \delta) \} \d \P_n^0,\\
&A_n^2 = \int_{\R^n} I_n(\theta, x_1^n) \d \P_n^0, &~~ &B_n^2 = \int_{\R^n} I_n(\theta, x_1^n) \ones\{ L_n \notin \ball(\sigma_{\rm sc}, \delta) \} \d \P_n^0,\\
&A_n^3 = \int_{\R^n} \exp\{ n \Phi(L_n, t) \} I_n(\theta, x_1^n) \d \P_n^0, &~~ &B_n^3 = \int_{\R^n} \exp\{ n \Phi(L_n, t) \} I_n(\theta, x_1^n)  \ones\{ L_n \notin \ball(\sigma_{\rm sc}, \delta) \} \d \P_n^0.\\
\end{aligned}
\end{equation}
\end{lemma}

\begin{proof}

We prove $B_n^3$ as an example. 
\begin{equation}
\begin{aligned}
B_n^3 \le & \int_{\R^n} \exp\{ n \Phi(L_n, t) \} I_n(\theta, x_1^n)  \ones\{ L_n \notin \ball(\sigma_{\rm sc}, \delta) \} \d \P_n^0\\
 =&  \underbrace{\int_{\R^n} \exp\{ n \Phi(L_n, t) \} I_n(\theta, x_1^n)  \ones\{  \max_{i \in [n]} \l x_i \l \ge R\} \d \P_n^0}_{E_1} \\
 &+ \underbrace{\int_{\R^n} \exp\{ n \Phi(L_n, t) \} I_n(\theta, x_1^n)  \ones\{ L_n \notin \ball(\sigma_{\rm sc}, \delta), \max_{i \in [n]} \l x_i \l \le R\} \d \P_n^0}_{E_2}.
\end{aligned}
\end{equation}

\noindent
{\bf Step 1. Bound for $E_1$. }

Let $\cUb_0 = [-U_0, U_0]$, $\cTb_0 = [-T_0, T_0]$ and $\cEb_0 = \cUb_0 \times \cTb_0$. Then
\begin{equation}
\begin{aligned}
&\lim_{R \rightarrow \infty} \limsup_{n\rightarrow \infty} \sup_{(\theta, t) \in \cEb_0} \frac{1}{n} \log E_1\\
\le&\lim_{R \rightarrow \infty} \limsup_{n\rightarrow \infty}\frac{1}{n} \log \E \{ (\dl W_n \dl_\op + T_0)^n \cdot \exp\{ U_0 \dl W_n \dl_{\op}\} ; \dl W_{n} \dl_{\op} \geq R \} = - \infty.
\end{aligned}
\end{equation}
For any $L$, we choose an $R > 0$ large enough such that 
\begin{align}
\limsup_{n\rightarrow \infty} \sup_{(\theta, t) \in \cEb_0} \frac{1}{n}\log  E_1 \leq L.
\end{align}

\noindent
{\bf Step 2. Bound for $E_2$: use the LDP of the empirical distribution of eigenvalues of GOE matrix. }

To bound $E_2$, we resort to the large deviation result for $L_n$:
\begin{align}
\lim_{n\rightarrow \infty}\frac{1}{n} \log \P_n^0(L_n \notin \ball(\sigma_{\rm sc},\delta)) = -\infty.
\end{align}
Therefore, we have upper bound
\begin{equation}
\begin{aligned}
\sup_{(\theta, t) \in \cEb_0}E_2  \leq& \P(L_n \notin \ball(\sigma_{\rm sc}, \delta)) \cdot \exp\{n [\log(R_0 + T_0) + U_0 R_0] \},
\end{aligned}
\end{equation}
which gives
\begin{align}
\lim_{n\rightarrow \infty}\sup_{(\theta, t) \in \cEb_0}\frac{1}{n} \log E_2 = -\infty.
\end{align}

Therefore, we have 
\begin{align}
\limsup_{n\rightarrow \infty}\sup_{(\theta, t) \in \cEb_0}\frac{1}{n} \log B_3^n \le L.
\end{align}
Sending $L \rightarrow -\infty$ gives the desired result.

\end{proof}

\section{Derivation of explicit formula for $S_\star$} \label{sec:calc_proofs}

\begin{proof} (Of Proposition \ref{prop:S_star_formula}) The function $S_{\star}(m,x)$ can be written as:
\begin{align}
S_{\star}(m,x) = & \frac{1}{2}\log(k-1)+\frac{1}{2}\log(1-m^{2})-k\lambda^{2}m^{2k-2}(1-m^{2})-(x-\lambda m^{k})^{2} \\
 & +\frac{2k}{k-1}\frac{x^{2}}{4}-\frac{1}{2}\intop_{2}^{\sqrt{\frac{2k}{k-1}}|x|}\sqrt{y^{2}-4}\d y \cdot 1_{ \{ {|x|} \geq \sqrt{\frac{2(k-1)}{k}} \} }
\end{align}

Isolating the dependence on $x$, to maximize $S_{\star}(m,x)$ we
must do the optimization problem:
\begin{equation}
-\frac{1}{2}\min_{u}\left\{ \frac{k-2}{2k}u^{2}-4(\lambda m^{k})\sqrt{\frac{k-1}{2k}}u+\intop_{2}^{u}\sqrt{y^{2}-4}\d y \cdot 1_{ \{ {|u|} \geq 2 \} } \right\} 
\end{equation}
where we have made the substitution $u=\sqrt{\frac{2k}{k-1}}x$. This
is exactly the setting of Lemma \ref{lem:maximization_problem} with
$a=\frac{k-2}{2k}$ and $b=4\lambda m^{k}\sqrt{\frac{k-1}{2k}}$. The
consideration $b\gtrless 4a$ leads to the definition of $m_{c}$ and
the two separate solutions $S_{U},S_{G}$. When $b<4a$, the formula
for $S_{U}$ is follows using the solution to the maximization problem
in this region: $-b^{2}/4a=-\frac{4(k-1)}{k-2}\lambda^{2}m^{2k}$ and
simplifying the resulting expression.

In the other region, when $b>4a$, using our $a,b$ values we compute
that the maximizing $u$ is:
\[
u^{\ast}=\frac{k}{\sqrt{k-1}}\sqrt{\frac{1}{2} k(\lambda m^{k})^{2}+1}-\sqrt{\frac{k}{2(k-1)}}(k-2)\lambda m^{k}
\]

The min value is (after some simplifying):
\begin{align*}
-\frac{1}{2} bu^{\ast}-2\log\left((\frac{1}{2}-a)u^{\ast}+\frac{1}{2} b\right) = & -2\sqrt{\frac{1}{2} k}\lambda m^{k}\sqrt{\frac{1}{2} k(\lambda m^{k})^{2}+1}+(k-2)\left(\lambda m^{k}\right)^{2}\\
  & -2\log\left(\sqrt{\frac{1}{2} k(\lambda m^{k})^{2}+1}+\sqrt{\frac{1}{2} k}\lambda m^{k}\right)+\log(k-1)\\
  = & -\sqrt{\frac{1}{2} k}\lambda m^{k}2\sqrt{\frac{1}{2} k(\lambda m^{k})^{2}+1}+(k-2)\left(\lambda m^{k}\right)^{2}\\
   & -2\sinh^{-1}\left(\sqrt{\frac{1}{2} k}\lambda m^{k}\right)+\log(k-1)
\end{align*}
(We have used $\sinh^{-1}(x)=\log(x+\sqrt{x^{2}+1})$. Plugging
back into the formula for $S_\star(m,x)$ now gives the desired result for $S_{G}$ .
\end{proof}

\begin{lemma}
\label{lem:maximization_problem}Let $a>0$ and $b>0$
be parameters. Let 
\[
g(x)=ax^{2}-bx+\intop_{2}^{|x|}\sqrt{y^{2}-4}\d y\cdot1_{\{|x|>2\}}.
\]

Then:
\[
\arg\min_{x}\left\{ g(x)\right\} =\begin{cases}
\frac{b}{2a} & 0<b<4a\\
x^{\ast} & b>4a
\end{cases},
\]

where:
\[
x^{\ast}:=\frac{2ab-\sqrt{b^{2}+4-16a^{2}}}{4a^{2}-1}.
\]

Moreover:
\[
\min_{x}\left\{ g(x)\right\} =\begin{cases}
-\frac{b^{2}}{4a} & 0<b<4a\\
-\frac{1}{2} bx^{\ast}-2\log\left(\left(\frac{1}{2}-a\right)x^{\ast}+\frac{1}{2} b\right) & b>4a
\end{cases}
\]
\end{lemma}
\begin{proof}
We notice that $g^{\prime}$ is monotone increasing, and hence
$g$ has a unique minimum which occurs when:
\[
b-2ax=\text{sgn}(x)\sqrt{|x|^{2}-4}\cdot1_{\left\{ |x|>2\right\} }
\]

If $-4a<b<4a$, then this occurs at $b/2a$. Otherwise, since we consider
only the case $b>0$, we have a solution with $x>2$ and have $b-2ax=\sqrt{x^{2}-4}.$
The quadratic formula then gives the formula for argmin. To see the
formula for the minimum, we use the closed form for the integral:
\[
\intop_{2}^{x}\sqrt{y^{2}-4}\d y=\frac{1}{2} x\sqrt{x^{2}-4}-2\log\left(\frac{x}{2}+\frac{1}{2}\sqrt{x^{2}-4}\right)
\]

Substituting the identity $b-2ax=\sqrt{x^{2}-4}$ then gives
the formula for $\min(g(x))$. 
\end{proof}

\begin{proof} (Of Proposition \ref{prop:S_G_analysis}) For any value of $\alpha \in(0,1)$, define:
\[
f_{\alpha}(x):=\frac{1}{2}\ln(1-\alpha)-\frac{2x^{2}}{\alpha}+x^{2}+x\sqrt{1+x^{2}}+\sinh^{-1}(x)
\]

It can be verified by computing the derivative, that this function
has exactly one maximum at $x_{\alpha}::=\frac{1}{2}\frac{\alpha}{\sqrt{1-\alpha}}$
and that $f_{\alpha}\left(x_{\alpha}\right)=0$. In particular, $f_{\alpha}(x)\leq0$
for all $x$. Now notice that we may write:
\[
S_{G}(m)=f_{m^{2}}\left(\sqrt{\frac{1}{2} k}\left(\lambda m^{k}\right)\right)
\]

This shows $S_{G}(m)\leq0$. The consideration about the zeros of
$f_{c}$ shows that $S_{G}$ has a zero only when $\sqrt{\frac{1}{2} k}\left(\lambda m^{k}\right)=\frac{1}{2}\frac{m^{2}}{\sqrt{1-m^{2}}}$,
which is equivalent to equation (\ref{eq:good_location}). Elementary calculus reveals that the polynomial $m^{2k-4}(1-m^{2})$ achieves a maximum value of  $\frac{(k-2)^{k-2}}{(k-1)^{(k-1)}}$ and this observation yields the desired properties for $\lambda_{c}$.
\end{proof}

\bibliographystyle{amsalpha}
\newcommand{\etalchar}[1]{$^{#1}$}
\ifx\undefined\bysame
\newcommand{\bysame}{\leavevmode\hbox to3em{\hrulefill}\,}
\fi

\end{document}